\documentclass[final]{siamltex}

\usepackage{amsmath}
\usepackage{amssymb}
\usepackage{cases}
\usepackage{commath}
\usepackage{enumerate}
\usepackage{xcolor}
\usepackage{graphicx}
\usepackage{float}
\usepackage{bm}
\usepackage{hyperref}

\usepackage{comment}
\excludecomment{PHDENABLED}
\includecomment{PHDDISABLED}

\newcommand{\bbB}{\mathbb{B}}
\newcommand{\bbR}{\mathbb{R}}
\newcommand{\bbS}{\mathbb{S}}
\newcommand{\bbU}{\mathbb{U}}
\newcommand{\bbV}{\mathbb{V}}
\newcommand{\bbW}{\mathbb{W}}

\newcommand{\calB}{\mathcal{B}}
\newcommand{\calD}{\mathcal{D}}

\newcommand{\calI}{\mathcal{I}}
\newcommand{\calP}{\mathcal{P}}
\newcommand{\calT}{\mathcal{T}}
\newcommand{\calU}{\mathcal{U}}

\newcommand{\calJ}{\mathcal{J}}

\newcommand{\aev}{\env{a.e. }}

\newcommand{\bdy}{\partial}

\newcommand{\cone}[1]{{\mathcal C\del{#1}}}

\DeclareMathOperator{\divgText}{div}
\newcommand{\divg}[1]{\divgText}
\newcommand{\difdir}[1]{\langle #1 \rangle}
\newcommand{\env}[1]{{\textnormal{#1}}}

\renewcommand{\grad}{\nabla}

\newcommand{\Ltwo}[1]{{L^2\del{#1}}}
\newcommand{\Linf}[1]{{L^\infty\del{#1}}}
\newcommand{\Lone}[1]{{L^1\del{#1}}}
\newcommand{\lapl}{\Delta}

\newcommand{\nd}{{\partial_{\nu}}}

\newcommand{\normL}[3]{\norm{#1}_{L^{#2}\del{#3}}}
\newcommand{\normLt}[2]{\norm{#1}_\Ltwo{#2}}
\newcommand{\normLi}[2]{\norm{#1}_\Linf{#2}}
\newcommand{\normLo}[2]{\norm{#1}_\Lone{#2}}

\newcommand{\normS}[4]{\norm{#1}_\sob{#2}{#3}{#4}}
\newcommand{\normSZ}[4]{\abs{#1}_\sob{#2}{#3}{#4}}
\newcommand{\of}[1]{(#1)}
\newcommand{\pair}[1]{\left\langle #1 \right\rangle}

\newcommand{\pairLt}[2]{{\mathop{\pair{#1}}}_{\Ltwo{#2}\times\Ltwo{#2}}}

\newcommand{\pairSob}[5]{{\mathop{\pair{#1}}}_{\sob{-#3}{#2}{#5}\times\sobZ{#3}{#4}{#5}}}
\renewcommand{\restriction}{\vert}

\newcommand{\setdef}[2]{\cbr{#1\,:\, #2}}

\newcommand{\sob}[3]{{W^{#1}_{#2}\del{#3}}}
\newcommand{\sobZ}[3]{{\mathring{W}^{#1}_{#2}\del{#3}}}

\newcommand{\totalD}{\operatorname{d}\!}

\renewcommand{\vec}[1]{\bm{{#1}}}

\newtheorem{thm}[theorem]{Theorem}
\newtheorem{lem}[theorem]{Lemma}
\newtheorem{prop}[theorem]{Proposition}
\newtheorem{cor}[theorem]{Corollary}

\newtheorem{rem}[theorem]{Remark}

\newcommand{\BO}{\calB_\Omega}
\newcommand{\BG}{\calB_\Gamma}
\newcommand{\DO}{\calD_\Omega}

\newcommand{\costF}{\mathcal{J}\del{\gamma, y, u}}

\newcommand{\Omegag}{{\Omega_\gamma}}

\renewcommand{\restriction}{\vert}
\newcommand{\yop}{\bar{y}}
\newcommand{\uop}{\bar{u}}

\newcommand{\gop}{\bar{\gamma}}
\newcommand{\gtl}{\widetilde{\gamma}}
\newcommand{\ytl}{\widetilde{y}}
\newcommand{\rop}{\bar{r}}

\newcommand{\sop}{\bar{s}}

\newcommand{\Yop}{\bar{Y}}
\newcommand{\Gop}{\bar{G}}
\newcommand{\Uop}{\bar{U}}
\newcommand{\Rop}{\bar{R}}
\newcommand{\Sop}{\bar{S}}

\newcommand{\Uad}{{\mathcal{U}_{ad}}}

\newcommand{\HA}[1]{\textcolor{black}{#1}}
\newcommand{\RHN}[1]{\textcolor{black}{#1}}

\title{Optimal Control of a  Free Boundary Problem
  with Surface Tension Effects: A Priori Error Analysis
\thanks{This research was supported in part by NSF grants DMS-0807811 and DMS-1109325.}
}

\author{Harbir Antil\thanks{Department of Mathematical Sciences. George Mason University, Fairfax, VA 22030, USA ({\tt hantil@gmu.edu})}. 
\and Ricardo H. Nochetto\thanks{Department of Mathematics and Institute for Physical Science and Technology, University of Maryland
College Park, MD 20742, USA ({\tt rhn@math.umd.edu})}. 
\and Patrick Sodr{\'e}\thanks{Department of Mathematics, 
University of Maryland College Park, MD 20742, USA ({\tt sodre@math.umd.edu})}.}

\begin{document}
\maketitle
\begin{abstract}
We present a finite element method along with its analysis for the
optimal control of a model free boundary 
problem with surface tension effects, formulated and studied in 
\cite{HAntil_RHNochetto_PSodre_2014a}. The state system couples
the Laplace equation in the bulk with the Young-Laplace equation on the free boundary to
account for surface tension. We first prove that the state and adjoint system
have the requisite regularity for the error analysis (strong solutions). 
We discretize the state, adjoint and control variables via
piecewise linear finite elements and show optimal $O(h)$ error
estimates for all variables, including the control. This entails using
the second order sufficient optimality conditions of 
\cite{HAntil_RHNochetto_PSodre_2014a}, and the first order necessary optimality
conditions for both the continuous and discrete systems.
We conclude with two numerical examples which examine the various error estimates.
\end{abstract}


\section{Introduction} \label{s:intro}

\HA{
Two-phase fluids are ubiquitous in nature, yet they represent a
formidable modeling, analytical, and computational challenge. 
{\it Ferrofluids} are roughly a large number of tiny permanent magnets 
floating in a liquid or carrier (usually some kind of oil). They can
be actuated, and so controlled, by means of external magnetic fields, 
which gives rise to many applications such as
instrumentation, vacuum technology, lubrication,
vibration damping, radar absorbing materials, magnetic manipulation of
microchannel flows, nanomotors, and nanopumps
\cite{MZahn_2001a,HHartshorne_JBackhouse_WLLee_2004a,KJVinoy_1996a, MMiwa_HHarita_TNishigami_RKaneko_2003a}.
If $(\vec u,p)$ are the velocity-pressure pair and $(\vec H,\vec B)$ denote the
magnetic field and magnetic induction, then the fluid stress tensor
reads \cite[Eq. (6)]{OLavrova_LTobiska_2006a}
\[
\vec \sigma_{ij}(\vec u,p, \vec H) = 2\eta \vec \varepsilon(\vec
u)_{ij} - \del{p+\frac{\mu_0}{2} |\vec  H|^2} \delta_{ij} + \vec B_i
\vec H_j
\quad i,j = 1, \dots, d,
\]
where 
$\vec \varepsilon(\vec u)_{ij} = \frac{1}{2} \del{\frac{\partial
    u_i}{\partial x_j} + \frac{\partial u_j}{\partial x_i} }$ is the
symmetric gradient, $\eta$ is the viscosity parameter and $\mu_0$ is 
the permeability constant. The interface condition on the free boundary
$\gamma$ separating the two fluids, say ferrofluid and air, reduces to a
balance of forces: the jump of the normal stress equals the surface
tension force \cite[Eq. (7)]{OLavrova_LTobiska_2006a}
\begin{equation}\label{balance-of-forces}
[\vec \sigma(\vec u, p, \vec H)] \vec \nu
= \kappa \mathcal{H} \vec \nu,
\end{equation}
where $\vec \nu$ is the unit normal vector to $\gamma$, 
$\mathcal{H}$ is the sum of principal curvatures of $\gamma$ and
$\kappa > 0$ is a surface tension coefficient. 
}

\HA{
{\it Electrowetting} on dielectric refers to the local modification of the surface
tension between two immiscible fluids via electric actuation \cite{FMugele_JCBaret_2005a,FMugele_MDuits_DVandenEnde_2010a,SWalker_ABonito_RNochetto_2010a}. This allows for change
of shape and wetting behavior of a two-fluid system and, thus, for its manipulation
and control. The existence of such a phenomenon was originally discovered by Lippmann
more than a century ago \cite{GLippmann_1875a}, but electrowetting 
has found recently a wide spectrum of applications, specially in the realm of
micro-fluidics (reprogrammable lab-on-chip
systems, auto-focus cell phone lenses, colored oil pixels and video speed smart
paper \cite{FMugele_JCBaret_2005a,JHeikenfeld_2011a,BBerge_2000a,RBFair_2007a}). The interface condition is somewhat similar to 
\eqref{balance-of-forces} with $\vec H$ replaced by the electric field 
$\vec E$; we refer to \cite[eq (2.13)]{RHNochetto_AJSalgado_SWWalker_2011a} for a
diffuse interface model (see also \cite{RHNochetto_AJSalgado_SWWalker_2011a,HAbels_HGarcke_GGruen_2010a,HLu_KGlasner_ABertozzi_CKim_2007a, TQian_XPWang_PSheng_2003a, TQian_XPWang_PSheng_2006a}). 
If the device is confined to be within a Hele-Shaw cell,
then dimension reduction of the Navier-Stokes equations leads to a
harmonic pressure $p$ inside the droplet with boundary condition
 \cite[Eq. (2)]{SWalker_ABonito_RNochetto_2010a}
\[
p = \kappa \mathcal{H} + e + \lambda.
\]
This again exhibits the surface tension effect as well as the electric forcing
$e$, which corresponds to a change of contact angle due to varying
voltage, and the contact line pinning effect 
$\lambda \in \lambda_0\textrm{sign} (\vec u\cdot\vec \nu)$ with
$\lambda_0 \ge 0$ constant.
}

\HA{
The preceding examples are free boundary problems (FBPs) with an
interface condition containing both the mean curvature $\mathcal{H}$
and a forcing function (or control) which can be actuated on to drive
the system. There is some supporting theory for these FBPs
\cite{OLavrova_LTobiska_2006a,BJJin_MPadula_2004a, BJJin_2005a, MPadula_VASolonnikov_2010, HBae_2011a}, but
unfortunately formulated in H\"older spaces rather than low-order
Sobolev spaces. The latter are essential for a variational approach.
}

\HA{
In \cite{HAntil_RHNochetto_PSodre_2014a} we explore the variational 
formulation of a simplified model FBP in graph form, formulated by Saavedra and
Scott \cite{PSaavedra_RScott_1991}, which still exhibits some of
the key difficulties of a fluid flow FBP but a simpler and tractable PDE structure.
The state variables $(y,\gamma)$ satisfy the Laplace equation in the
bulk $\Omega_\gamma$ and the Young-Laplace equation on the free
boundary $\Gamma_\gamma$ modified by an additive control $u$. 
}

\HA{
\vskip-0.3cm
\begin{figure}[H] \label{f:domain}
\centering
\includegraphics[width=0.3\textwidth,height=0.25\textwidth]{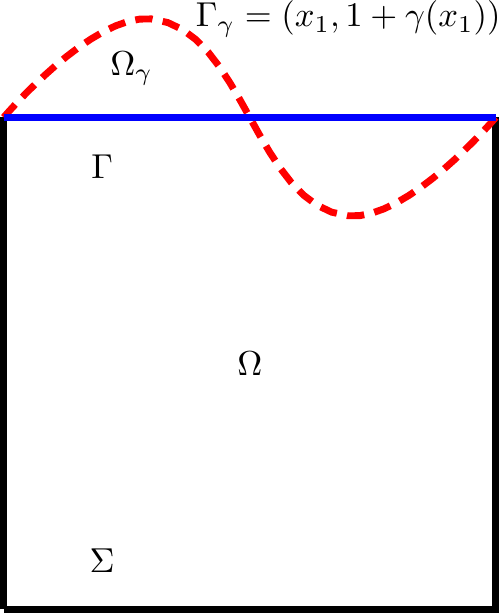}
\caption{\RHN{$\Omega_{\gamma}$ is the physical domain with boundary 
               $\partial \Omega_{\gamma} = \Sigma \cup \Gamma_{\gamma}$. Here $\Sigma$ includes the lateral and the 
               bottom boundary and is assumed to be fixed. Moreover, the top boundary $\Gamma_{\gamma}$ (dotted line) 
               is ``free" and is assumed to be a graph of the form $(x_1,1+\gamma(x_1))$, where $\gamma \in \sobZ{1}{\infty}{0,1}$ 
               denotes a parametrization. The free boundary $\Gamma_{\gamma}$ is further mapped to a fixed boundary $\Gamma = (0,1) \times \{ 1 \}$ and $\Omega_{\gamma}$ is in turn mapped to a reference domain $\Omega = (0,1)^2$, where all 
               computations are carried out.}}
\vskip-0.3cm
\end{figure}
The geometric configuration is depicted in \figref{f:domain}.
Let $\gamma \in \sobZ{1}{\infty}{0,1}$ denote a parametrization of the
top boundary of the physical domain 
$\Omega_{\gamma} \subset \Omega^* \subset \mathbb{R}^2$ with boundary 
$\partial \Omega_{\gamma} := \Gamma_{\gamma} \cup \Sigma$, defined as 
\begin{subequations}
	\begin{align*}
 		\Omega^* &:= (0,1) \times (0,2), \\
 		\Omegag &:= \setdef{(x_1,x_2)}{0 < x_1 < 1,\;0< x_2 < 1+ \gamma(x_1)}, \\
 		\Gamma_{\gamma} &:= \setdef{(x_1,x_2)}{ 0 < x_1 < 1,\; x_2=1+\gamma(x_1)}, \\ 
 		\Sigma &:= \bdy\Omega_{\gamma} \setminus \Gamma_{\gamma}, \\
		 \Gamma &:= \setdef{(x_1,x_2)}{ 0 < x_1 < 1,\; x_2=1} .
	\end{align*}
\end{subequations}
Here, $\Omega^*$ and $\Sigma$ are fixed while $\Omega_{\gamma}$ and $\Gamma_{\gamma}$ deform
according to $\gamma$. 
}

\HA{
We want to
find an optimal control $u \in \Uad \subset \Ltwo{0,1}$ so that the 
the state variables $(y,\gamma)$ are
 the best least squares fit of desired boundary $\gamma_d: \del{0,1} \to \mathbb{R}$ and 
bulk $y_d: \Omega^* \to \mathbb{R}$ configurations. This amounts to
solving the problem: minimize
}
\begin{subequations}\label{eq:ocfbp_1}
	\begin{equation}\label{eq:cost_fcn}
	     \costF := \frac{1}{2}\normLt{\gamma - \gamma_d}{0,1}^2  
				+ \frac{\HA{\mu}}{2}\normLt{y - y_d}{\Omegag}^2 
				+ \frac{\lambda}{2}\normLt{u}{0,1}^2,
	\end{equation}	
	subject to the state equations 
	\begin{equation}\label{eq:state_eqn}\left\{
		\begin{aligned} 
			-\lapl y &= 0 && \env{in }\Omegag \\ 
	    		y &= v  && \env{on } \bdy\Omegag \\
	    	-\kappa \mathcal H\sbr{\gamma} 
	    	+ \nd y\del{\cdot,1+\gamma}  &=  u
	    		&& \env{on } \del{0,1} \\
	    	\gamma(0) = \gamma(1) &= 0,
		\end{aligned}
	\right.\end{equation}
	the state constraints
	\begin{align}\label{eq:state_const_a}
		 \abs{\totalD_{x_1}\gamma} \leq 1 \quad \aev \env{in } \del{0,1},
	\end{align}
	with $\totalD_{x_1}$ being the total derivative with respect to $x_1$, 
	and the control constraint 
	\begin{align}\label{eq:control_const_a}
		u &\in \Uad
	\end{align}
	dictated by $\Uad$, a closed ball in $\Ltwo{0,1}$, to be
        specified later in \eqref{eq:v_u_const_set}. 
\end{subequations}
Here $\lambda > 0$ is the stabilization parameter; \HA{$\mu \ge 0$};  
$v$ is given data which in principle 
could act as a Dirichlet boundary control; 
$$
\mathcal H\sbr{\gamma} := \totalD_{x_1}\del{\frac{\totalD_{x_1}\gamma}{\sqrt{1+\abs{\totalD_{x_1}\gamma}^2}}}
$$ 
is the \emph{curvature} of $\gamma$; and $\kappa > 0$ plays the role of surface tension coefficient. 

\HA{
We remark that it is customary to take $u \in \Uad$, as the amount of control in practice is 
limited. On the other hand it may happen that several $u \in \Uad$ realize optimal state $y$;
the $\lambda$ term in cost \eqref{eq:cost_fcn} helps to realize one of these controls with 
the smallest norm. Mathematically speaking, the choice of $\lambda$ term in \eqref{eq:cost_fcn} ensures the second order sufficient condition \cite[Corollary 5.8]{HAntil_RHNochetto_PSodre_2014a} and thus the local uniqueness of the optimal control. This further leads to optimal a priori error estimates for the control in \thmref{eq:error_est_control}. The control constraint $u \in \Uad$ is essential to have the well-posedness of the state system \eqref{eq:state_eqn}, see \cite[Theorem 4.5]{HAntil_RHNochetto_PSodre_2014a}. This has further repercussions and provides a well-defined control-to-state map \cite[Eq. (4.1)]{HAntil_RHNochetto_PSodre_2014a} and a reduced functional \cite[Section 5.1]{HAntil_RHNochetto_PSodre_2014a}. The reduced functional approach is a standard technique to solve the optimal control problems \cite{FTroltzsch_2010a, MDGunzburger_2003a}.
}

We use a fixed domain approach to solve the optimal control free boundary problem (OC-FBP). In fact, we transform $\Omega_\gamma$ to $\Omega=(0,1)^2$ and $\Gamma_\gamma$ to $\Gamma=(0,1)\times\set{1}$ (see Figure~\ref{f:domain}), at the expense of having a governing PDE with rough coefficients. 
One of the challenges of an OC-FBP is dealing with 
state constraints which may allow or prevent 
topological changes of the domain. Our analysis 
in \cite{HAntil_RHNochetto_PSodre_2014a} yields the
control constraint \eqref{eq:control_const_a},
which always enforce the state constraint \eqref{eq:state_const_a}
i.e., 
we can treat OC-FBP as a simpler \emph{control constrained}
problem. Moreover, we proved novel second
order sufficient conditions for the optimal control problem for small
data $v$. As a consequence we obtained that 
the above minimization has a (locally) unique solution.

\HA{
In this paper we continue our investigation of \cite{HAntil_RHNochetto_PSodre_2014a}.
We introduce a} fully discrete optimization problem, 
using piecewise linear finite elements, and show that it converges with an 
optimal rate $O(h)$ for all variables. In fact, 
the convergence analysis for the control requires a-priori error estimates
for both the state and the adjoint equations. The state equations
error estimates were developed by P. Saavedra and R. Scott in
\cite{PSaavedra_RScott_1991} under the assumption that the continuous
state equations have suitable second order regularity (strong
Sobolev solutions). Our first goal is to prove such a regularity for 
$v \in W^2_p(\Omega), p>2,$ via a fixed-point argument, as well as
to extend the analysis to the continuous adjoint
equations. Our analysis of strong solutions for both the
state and adjoint equations is novel in Sobolev spaces but not in 
H\"older spaces \cite{VASolonnikov_1982a}.
We exploit this second order regularity to derive a-priori error
estimates for the state and adjoint variables based on
\cite{PSaavedra_RScott_1991}; the former are a direct extension of
\cite{PSaavedra_RScott_1991} whereas the latter are new.
An important difference with \cite{PSaavedra_RScott_1991} is the
presence of the control variable $u$, for which we obtain also
a novel a-priori error estimate.

There are two approaches for dealing with a discrete optimal
control problem with PDE constraints. Both rely on an agnostic
discretization of the state and adjoint equations, perhaps by
  the finite element method; they differ on whether or not the admissible set of controls is discretized as well. 
The first approach \cite{NArada_ECasas_FTroeltzsch_2002a, ECasas_MMateos_FTroeltzsch_2005a, ARoesch_2006a} follows a more physically appealing idea and also discretizes the admissible control set.
The second approach \cite{MHinze_2005a} induces a discretization of the optimal control by projecting the discrete adjoint state into the admissible control set. 
From an implementation perspective this projection may lead to a
control which is not discrete in the current mesh and thus requires
an independent mesh. Its key advantage
is obtaining an optimal quadratic rate of convergence \cite[Theorem
2.4]{MHinze_2005a} for the control. We point that
no such improvements can be inferred for our problem. 
This is due to the highly nonlinear nature of the state equations and 
the necessity to discretize the (rough) coefficients. 
We will provide more details on this topic in Section \ref{s:control_numerics} below.

We follow the first approach and discretize the entire optimization problem
using piecewise linear finite elements. However, we exploit the
structure of the admissible control set and show optimal
convergence rates for the state, adjoint, and control variables.
We largely base our error analysis
of the state and adjoint equations on the work by P. Saavedra and
L. R. Scott \cite{PSaavedra_RScott_1991}. 
For analyzing the discrete control we use the second order sufficient
condition we developed in \cite[Theorem
5.6]{HAntil_RHNochetto_PSodre_2014a} together with the first order
necessary conditions for the continuous and discrete systems.

To prove well-posedness of the discrete state and adjoint systems we rely on the inf-sup theory and a fixed point argument. Therefore, the smallness assumption on the data $v$ is still required as in the continuous case \cite[Theorem 4.5]{HAntil_RHNochetto_PSodre_2014a}. 
However, with the aid of simulations we are able to explore the
control problem beyond theory and test it for large data.

The outline of this paper is as follows: In Section \ref{s:oc_fbp}, we state the variational form for the OC-FBP. We summarize the first order necessary and 
second order sufficient conditions in Sections \ref{s:oc_first} and \ref{s:oc_second}.
For boundary data in $W^2_p$, $p>2$, we show that the state and the adjoint systems
have strong solutions in Section \ref{s:high_regularity}. We introduce a finite element
discretization of the system in Section \ref{s:oc_disc_fbp} and 
prove error estimates in $W^1_p\times W^1_\infty$ for the state
variables and in $W^1_q\times W^1_1$ for the adjoint variables. 
We derive an $L^2$-error estimate for the optimal control in
Section \ref{s:control_numerics}. We conclude with two numerical examples in 
Section \ref{s:oc_numerics_fbp} which confirm our theoretical findings:
the first example explores the unconstrained problem, whereas
the second example deals with the constrained one.

\section{Continuous Optimal Control Problem}\label{s:summary}
The purpose of this section is to recall the continuous optimal
control problem in its variational form 
along with its first and second optimality conditions derived in
\cite{HAntil_RHNochetto_PSodre_2014a}.
We denote by $\lesssim$ the inequality $\le C$ with a constant independent of the
quantities of interest.

\subsection{Problem Formulation} \label{s:oc_fbp}

We choose to present the formulation directly in its variational form
on the reference domain $\Omega$ after having linearized the curvature
$\mathcal H$, and scaled the control $u$. These assumptions, not
  being crucial \cite[Section~2, ${\bf A}_{1}$-${\bf
    A}_2$]{HAntil_RHNochetto_PSodre_2014a}, result in an optimal
control problem subject to a nonlinear PDE constraint with (rough)
coefficients depending on $\gamma$ but without an explicit interface. 
We denote by $\BG:\sobZ 1 \infty {0,1} \times \sobZ 1 1 {0,1} \to\bbR$
and  $\BO:\sobZ1p\Omega\times\sobZ1q\Omega\to\bbR$ the bilinear forms 
	\begin{equation}\begin{aligned} \label{eq:bilinear_form}
    		\BG\sbr{\gamma,\xi} &:= \kappa \int_{0}^1 {\totalD_{x_1}\gamma}(x_1) {\totalD_{x_1}\xi}(x_1) \dif x_1, \\
    		\BO\sbr{y,z;A\sbr{\gamma}} &:= \int_\Omega A\sbr{\gamma} \grad y \cdot \grad z \dif x,
	\end{aligned}\end{equation}
with surface tension constant $\kappa$. The coefficient matrix
$A\sbr{\gamma}$ arises from mapping the physical domain $\Omegag$ to
the reference domain $\Omega$ and is given by 
\cite[Section 2]{HAntil_RHNochetto_PSodre_2014a}
	\begin{align} \label{eq:op_A}
		A\sbr{\gamma} &
					  = \begin{bmatrix} 1+\gamma(x_1) & -\totalD_{x_1}\gamma(x_1)x_2 \\	
		  			  -\totalD_{x_1}\gamma(x_1)x_2 & \frac{1+		\del{\totalD_{x_1}\gamma(x_1)x_2}^2}{1+\gamma(x_1)} \end{bmatrix} .
	\end{align} 
Let $E:\sobZ{1}{1}{0,1} \rightarrow \sob{1}{q}{\Omega}$, $1 < q < 2$, be a
continuous linear extension operator, namely,
	\begin{align} \label{eq:w1_wq_ext}
  		E \xi \restriction_{\Gamma} = \xi,\quad
  		E\xi \restriction_{\Sigma} = 0, \quad
  		\abs{E\xi}_{\sob{1}{q}{\Omega}} \le C_E \abs{\xi}_{\sob{1}{1}{0,1}},
	\end{align}
where $C_E$ is the stability constant. It is convenient to
introduce the product space:
\begin{equation}\label{eq:product_space1}
    \bbW^{1,1}_{t,s} := \sobZ 1 t{0,1} \times \sobZ 1 s{\Omega} \quad
    1\le t,s \le \infty.
\end{equation}

Given $v \in \sob{1}{p}\Omega$, a lifting of the boundary data to
$\Omega$, $y_d \in L^2(\Omega^*)$, $\gamma_d \in L^2(0,1)$,
let $\delta y :=  y + v - y_d$, $\delta \gamma := \gamma - \gamma_d$. 
The optimal control problem is to minimize \looseness=-1
\begin{subequations}\label{eq:ocfbp}
        \begin{equation}\label{eq:lin_cost}
			 \costF := \frac{1}{2}\normLt{\delta\gamma}{0,1}^2 
                         + \frac{\HA{\mu}}{2}\normLt{\delta y\sqrt{1+\gamma}}{\Omega}^2 
			 + \frac{\lambda}{2}\normLt{u}{0,1}^2 ,
		\end{equation}
with state variable $\del{\gamma,y}\in \bbW^{1,1}_{\infty,p}$,
$p > 2$ satisfying the state equations
    	\begin{equation}\label{eq:lin_st_const}
			\BG\sbr{\gamma, \xi} + \BO\sbr{y + v, z + E\xi; A\sbr{\gamma}} 
				= \pairSob{u,\xi}{\infty}{1}{1}{0,1}  \quad \forall \del{\xi, z} \in \bbW^{1,1}_{1,q}
		,\end{equation}
the state constraint
    \begin{align}\label{eq:state_const}
		&\abs{\totalD_{x_1}\gamma(x_1)} \leq 1 \quad \aev x_1 \in \del{0,1},
	\end{align}   
with $\totalD_{x_1}$ being the total derivative with respect to $x_1$, 
	and the control constraint
    \begin{align}\label{eq:control_const} 
		u &\in \Uad.
	\end{align}
\end{subequations}
The set  $\Uad$ of {\it admissible controls} is a closed ball in
  $\Ltwo{0,1}$ defined as follows: if $\theta_1 \in (0,1)$ is chosen as in \cite[Lemma 4.3]{HAntil_RHNochetto_PSodre_2014a}, and $\alpha$ is the inf-sup constant for $\BG$ 
\cite[Proposition 4.1]{HAntil_RHNochetto_PSodre_2014a}, then
\begin{align}    \label{eq:v_u_const_set}
\Uad &:= \set{ u \in \Ltwo{0,1}\;:\; \normLt{u}{0,1} \leq \theta_1 /2\alpha}.
\end{align}
We also need the open ball $\calU := \set{u \in L^2(0,1) : \norm{u}_{L^2(0,1)} < \theta_1/\alpha}$, so that $\Uad \subset \calU$. 
In view of the regularity of $u$ and $\xi$, the duality pairing
$\pairSob{u,\xi}{\infty}{1}{1}{0,1}$ reduces to $\int_0^1 u \zeta$. We
refer to \cite[Equations 4.5-4.6]{HAntil_RHNochetto_PSodre_2014a} for
details. Since $\Uad$ is not open, we need to define a proper set of admissible directions to compute derivatives 
with respect to $u$.
Given $u \in \Uad$, the convex cone $\cone{u}$ comprises all
directions $h \in \Ltwo{{0,1}}$
such that $u+th \in \Uad$, $t>0$, i.e.,
	\[
    	\cone{u} := \set{h \in \Ltwo{{0,1}}:\; u+th \in \mathcal
          U_{ad}, t>0 }.
	\]   

The proof of existence and local uniqueness of a minimizer of \eqref{eq:ocfbp} involves multiple steps. The first step \cite[Subsection 4.1.1]{HAntil_RHNochetto_PSodre_2014a} is to impose a smallness condition on the data $v$ and restrict the radius of the $\Ltwo{0,1}$ ball $\Uad$ to solve the nonlinear system \eqref{eq:lin_st_const}-\eqref{eq:state_const}. The second step \cite[Subsection 4.1.2]{HAntil_RHNochetto_PSodre_2014a} is to improve the regularity of $\gamma$ from Lipschitz continuous to the fractional Sobolev space $\sob{2-1/p}{p}{0,1}$. This additional regularity is in turn used to prove the existence of a minimizer in \cite[Subsection 5.1]{HAntil_RHNochetto_PSodre_2014a}. The last two steps \cite[Subsection 5.2 and 5.3]{HAntil_RHNochetto_PSodre_2014a} consist of computing the first and second order optimality conditions.
 
The optimality conditions are essential tools for this paper. From the
simulation perspective, the first order condition yields a way to
compute a minimizing sequence. From a numerical analysis perspective,
the second order condition is the starting point for proving an
a-priori error estimate. We recall now these conditions and
prove a new result, \lemref{L:Jacobian-state}, which is instrumental for
implementing the adjoint system.

\subsection{First-order Optimality Condition} \label{s:oc_first}
The purpose of this section is to state the first order necessary optimality conditions using the reduced cost functional approach; a formal Lagrange multiplier approach is also presented in \cite[Section 3]{HAntil_RHNochetto_PSodre_2014a}. 

Given $u \in \calU$, there exists a unique solution $(\gamma,y) \in \bbW^{1,1}_{\infty,p}$ of \eqref{eq:lin_st_const}-\eqref{eq:state_const}. This induces the so-called \emph{control-to-state} map $G_{v} : \calU \rightarrow \bbW^{1,1}_{\infty,p}$, where $G_v(u) = \del{\gamma(u),y(u)}$, and we can write the cost functional $\calJ$ in \eqref{eq:lin_cost} in the reduced form as 
\begin{equation}\label{eq:reduced_cost}
    \calJ(u) := \calJ_1(G_v(u)) + \calJ_2(u)
\end{equation}
with
\[
    \calJ_1(\gamma,y) := \frac{1}{2}\norm{\delta\gamma}_{L^2(0,1)}^2 + \frac{\HA{\mu}}{2}\norm{\delta y\sqrt{1+\gamma}}^2_{L^2(\Omega)}, \quad
    \calJ_2(u) = \frac{\lambda}{2}\norm{u}_{L^2(0,1)}^2. 
\]    
Therefore, if $\uop \in \Uad$ is a minimizer of $\calJ$, then
  $\uop$ satisfies the variational inequality
	\begin{align}  \label{eq:var_ineq}
       \pairLt{\mathcal{J}'(\uop), u - \uop}{0,1}  &\ge 0  \quad  \forall u \in 		\Uad.
	\end{align}

Deriving an expression for $\calJ'(\uop)$ is one of the main
difficulties of an optimization problem. It turns out that
$\mathcal{J}'(\uop) = \lambda \uop + \sop$
\cite[Sections 3 and 5.2]{HAntil_RHNochetto_PSodre_2014a}, whence
	\begin{align} \label{eq:lin_uop_var_ineq}
   		\pairLt{\lambda \uop + \sop, u-\uop }{0,1} \geq 0 \qquad \forall u \in \mathcal{U}_{ad},
	\end{align}  
where $\del{\sop, \rop-E\sop}\in\bbW^{1,1}_{1,q}$ and $(\sop, \rop)$ satisfies
  the \emph{adjoint equations} in variational form
\begin{equation}\label{eq:lin_rop_sop} 	
\begin{aligned}
		\BG\sbr{\xi, \sop} + \DO\sbr{(\xi, z), \rop ; \gop, \yop} 
                & = \pair{\calJ_1'(\gop,\yop),(\xi,z)} \\
                & = \pair{\xi, \delta\gop
		+ \frac{\HA{\mu}}{2} \int_0^1 \abs{\delta\yop}^2 \dif x_2 }
                + \pair{z, \HA{\mu}\delta\yop \del{1+\gop}} ,
\end{aligned}
\end{equation}
for all $\del{\xi, z}$ in $\bbW^{1,1}_{\infty,p}$ with 
$\del{\gop,\yop}:=\del{\gamma(\uop), y(\uop)}$,
where $\DO$ is the parametrized form 
	\begin{equation}\label{eq:DO_form}
		\DO\sbr{(\xi, z), \rop; \gop, \yop} := \BO\sbr{z, \rop;A\sbr{\gop}} + \BO\sbr{\yop + v, \rop;\Dif A\sbr\gop \difdir\xi}
	\end{equation}
with derivative of $A$ with respect to $\gamma$ given by 
$\Dif A\sbr{\gop}\difdir{h} = A_1\sbr{\gop}h + A_2\sbr{\gop}\totalD_{x_1} h$
\cite[(2.4)]{HAntil_RHNochetto_PSodre_2014a}.
Moreover, the duality pairings on the right hand side of \eqref{eq:lin_rop_sop} are reduced to standard integrals due to the $L^2$-regularity imposed by the cost functional.

The assembly of \eqref{eq:lin_rop_sop} is nontrivial.
In view of \eqref{eq:bilinear_form} and \eqref{eq:DO_form} we have
\begin{align}\label{int-in-x2}
    \BO\sbr{\yop + v, \rop;\Dif A\sbr\gop \difdir\xi} 
    &= \int_0^1 \del{\int_0^1  A_1\sbr{\gop} \grad (\yop+v) \cdot \grad \rop \dif x_2} \xi \dif x_1 \\
    &\quad+ \int_0^1 \del{\int_0^1  A_2\sbr{\gop} \grad (\yop+v) \cdot
      \grad \rop \dif x_2} \xi_{x_1} \dif x_1 , 
\end{align}
and the computation of the inner integrals in the
variable $x_2$ alone might seem like a daunting task. However we can
circumvent this issue altogether with the following simple
observation. The control-to-state map $G_v:\calU \rightarrow\bbW^{1,1}_{\infty,p}$
admits a Fr\'echet derivative $(\gamma,y)=G_v'(\bar u) h \in
\bbW^{1,1}_{\infty,p}$ which satisfies the linear variational system
\begin{equation} \label{eq:g_diff_u}
	\BG\sbr{\gamma, \zeta} + \DO\sbr{\del{\gamma, y}, z+ E\zeta; \gop, \yop} = \int_0^1 h \zeta
	\quad \forall \del{\zeta, z} \in \bbW^{1,1}_{1,q},
\end{equation}
for every $\bar u \in\calU$ and $h\in L^2(0,1)$. This formal
differentiation of \eqref{eq:lin_st_const} is rigorously justified in
\cite[Theorem~4.12]{HAntil_RHNochetto_PSodre_2014a}. The system
\eqref{eq:g_diff_u} consists of the two equations
\begin{equation*}
\begin{alignedat}{2}
\BO[y,z,A\sbr{\gop}] + \BO[\bar y+v,z; \Dif
A\sbr{\gop}\difdir{\gamma}] &=0 \quad &&\forall z\in\sobZ1q\Omega
\\
\BO[y,E\zeta;A\sbr{\gop}] + \BO[\bar y + v, E\zeta; \Dif A\sbr{\gop}\difdir{\gamma}]
+ \BG[\gamma,\zeta] &= \langle h,\zeta \rangle
\quad && \forall \zeta \in \sobZ11{0,1}.
\end{alignedat}
\end{equation*}
The formal adjoint of this system, obtained upon regarding
$(\gamma,y)\in \bbW^{1,1}_{\infty,p}$ as test functions and 
$(\zeta,z)\in\bbW^{1,1}_{1,q}$ as unknowns, reads as follows:
\begin{equation}\label{Frechet-control-to-state}
\begin{aligned}
\BO[y,z,A\sbr{\gop}] + \BO[y,E\zeta;A\sbr{\gop}] 
&= \langle f,y \rangle
\\
\BO[\bar y+v,z; \Dif A\sbr{\gop}\difdir{\gamma}] 
+ \BO[\bar y + v, E\zeta; \Dif A\sbr{\gop}\difdir{\gamma}]
+ \BG[\gamma,\zeta] &= \langle g,\gamma \rangle.
\end{aligned}
\end{equation}

\begin{lemma}[relation between \eqref{eq:lin_rop_sop} and \eqref{Frechet-control-to-state}]\label{L:Jacobian-state}
The linear system \eqref{eq:lin_rop_sop} coincides with
\eqref{Frechet-control-to-state} provided $f = \HA{\mu}\delta\bar y(1+\bar\gamma)$
and $g = \delta\bar\gamma +\frac{\HA{\mu}}2 \int_0^1 |\delta \bar y|^2 \dif x_2$.
\end{lemma}
\begin{proof}
It suffices to check that $(\zeta,z+E\zeta)$ satisfies
\eqref{eq:lin_rop_sop} and invoke the uniqueness of
\eqref{eq:lin_rop_sop}, the latter being a consequence of
\cite[Lemma 5.6]{HAntil_RHNochetto_PSodre_2014a}.
\end{proof}

Lemma \ref{L:Jacobian-state} is instrumental for the implementation of
this control problem. Later in Section~\ref{s:oc_numerics_fbp}, we
choose Newton's method instead of a fixed point iteration to solve the
{\it state} equations, because it is locally a second order
method. Secondly, computing a Newton direction $(\gamma,y)$ requires
solving a linear system of type \eqref{eq:g_diff_u}. By transposition,
the same matrix can be used to solve for the {\it adjoint} variables
thereby making the seemingly complicated coupling $\DO$ rather simple
to deal with.

\subsection{Second-order Sufficient Conditions} \label{s:oc_second}
It is well known that the first order necessary optimality conditions
are also sufficient for the well-posedness of a convex optimization
problem with linear constraints. Unfortunately, we \emph{ cannot }
assert the convexity of our problem due to the highly nonlinear nature
of \eqref{eq:lin_cost}-\eqref{eq:lin_st_const}. A large portion of our
previous work \cite[Theorem 5.7]{HAntil_RHNochetto_PSodre_2014a} was 
devoted to proving a second-order sufficient condition.
We restrict ourselves to merely restating that result.

\begin{thm}[second-order sufficient conditions] \label{thm:second_order_suff}
If $\normSZ{v}{1}{p}{\Omega}$ is small enough and $\uop$ in $\Uad$ is an optimal control, then there exists a neighborhood of $\uop$ such that 
	\begin{align}   \label{eq:second_order_suff}
		\calJ''(\uop) \del{u-\uop}^2 &\ge \frac{\lambda}{2} \normLt{u-\uop}{0,1}^2 \quad \forall u \in \uop + \cone\uop .
	\end{align}
Furthermore, the following two conditions hold: \emph{local quadratic growth}
	\begin{align} \label{eq:J_quad_growth}
		\calJ(u) &\ge \calJ(\uop) + \frac{\lambda}{8}
                \normLt{u-\uop}{0,1}^2
                \quad \forall u \in \uop + \cone\uop,
\end{align}
and \emph{local convexity}
	\begin{align}  \label{eq:gradJ_quad_growth}
		\pairLt{\calJ'(u) - \calJ'(\uop), u-\uop}{0,1} \ge \frac{\lambda}{4} \normLt{u - \uop}{0,1}^2
		\quad \forall u \in \uop + \cone\uop . 
	\end{align}
\end{thm}

\section{Strong Solutions: Second-order Regularity}\label{s:high_regularity}
The goal of this section is to prove the existence of strong solutions
to the state and adjoint equations. The underlying second-order
  Sobolev regularity is crucial for the a-priori error
estimates of Section \ref{s:oc_disc_apriori}, and thus an important
contribution of this paper. 

\subsection{Second-order Regularity in the Square}

We start with an auxiliary regularity result for the square
$\Omega$. This type of results are well known for $C^{1,1}$ domains
\cite[Theorem 9.15 and Lemma 9.17]{DGilbarg_NTrudinger_2001a}.

\begin{lemma}[second-order regularity in the square]\label{L:2-reg}
Let $\Omega = (0,1)^2$ and $A = (a_{ij})_{ij=1}^2 \in
W^1_\infty(\Omega)$. If $f \in L^p(\Omega)$, with $1 < p < \infty$,
then there exists a unique solution $w \in \sob2p\Omega \cap \sobZ1p\Omega$ to 
\begin{equation}\label{div-form}
    -\divgText(A\nabla w) = f \quad \mbox{in } \Omega,
\end{equation}
and a constant $C_\#$ depending on $\|A\|_{W^1_\infty(\Omega)}$
  and $p$ such that
\begin{equation}\label{apriori-div}
    \normS{w}2p\Omega 
        \le C_\# \norm{f}_{L^p(\Omega)} . 
\end{equation}
\end{lemma}
\begin{proof}
We proceed in several steps. First we prove \eqref{apriori-div}
assuming that there is a solution in $W^2_p(\Omega)$, and next we show
the existence of such a solution.

1. {\it Reflection}:
We introduce an odd reflection $\widetilde{f}$ of $f$ and an even reflection 
$\widetilde{A}$ of $A$ to the adjacent unit squares of $\Omega$
so that the extended domain $\widetilde\Omega$ is a square with
vertices $(-1,-1)$, $(2,-1)$, $(2,2)$, $(-1,2)$. We observe that
$\widetilde{f} \in L^p(\widetilde\Omega)$ and 
$\widetilde{A} \in W^1_\infty(\widetilde\Omega)$.
Since $\|\widetilde{f}\|_{H^{-1}(\widetilde{\Omega})}\lesssim
\|\widetilde{f}\|_{L^p(\widetilde{\Omega})}$ by Sobolev
embedding in two dimensions, there exists a unique solution 
$\widetilde{w}\in H^1_0(\widetilde{\Omega})$ to 
the extended problem
\begin{equation}\label{ext-div-form}
-\divgText(\widetilde{A}\nabla \widetilde{w}) = 
\widetilde{f} \quad \mbox{in } \widetilde{\Omega}.
\end{equation}
Such $\widetilde{w}$
is an odd reflection to $\widetilde\Omega$ of its restriction $w$ to
$\Omega$, whence $w$ has a vanishing trace on $\partial\Omega$. 
Moreover, $\|\widetilde{w}\|_{L^p(\widetilde{\Omega})}\lesssim
\|\widetilde{w}\|_{H^1_0(\widetilde{\Omega})}$ because of Sobolev
embedding and 
\begin{equation}\label{Lp-norms}
\|\widetilde{w}\|_{L^p(\widetilde\Omega)} \lesssim \|w\|_{L^p(\Omega)}.
\end{equation}

2. {\it A priori $W^2_p(\Omega)$-estimate}: If
$\widetilde{w} \in W^2_{p,{\tt loc}}(\widetilde{\Omega}) \cap L^p(\widetilde{\Omega})$ is a
solution of \eqref{ext-div-form}, 
then we write \eqref{ext-div-form} in nondivergence form
\[
-\widetilde{A}:D^2\widetilde{w} - \divgText\widetilde{A}\cdot\nabla
\widetilde{w} = \widetilde{f}
\]
and apply the interior $W^2_p$ estimates of 
\cite[Theorem 9.11]{DGilbarg_NTrudinger_2001a} to write
\begin{equation}\label{localW2p}
\|w\|_{\sob2p\Omega} = 
\|\widetilde{w}\|_{\sob2p\Omega} \lesssim \|\widetilde{w}\|_{L^p(\widetilde{\Omega})}
+ \|\widetilde{f}\|_{L^p(\widetilde{\Omega})}
\lesssim \|w\|_{L^p(\Omega)} + \|f\|_{L^p(\Omega)}
\end{equation}
and $w\in\sobZ1p\Omega$.
To show that this estimate implies \eqref{apriori-div}, we argue by
contradiction as in \cite[Lemma 9.17]{DGilbarg_NTrudinger_2001a}. Let
$\{w_m\}\subset\sob2p\Omega\cap\sobZ1p\Omega$ be a sequence satisfying
\[
\|w_m\|_{\sob2p\Omega} = 1, \qquad
\|f_m\|_{L^p(\Omega)} \rightarrow 0
\]
as $m\to\infty$, where $f_m=-\divgText(A\nabla w_m)$. 
Since the unit ball in $\sob2p\Omega$ is weakly compact for
$1<p<\infty$, there exists a
subsequence, still labeled $w_m$, that converges weakly in
$\sob2p\Omega$ and strongly in $\sob1p\Omega$ to a function 
$w\in\sob2p\Omega\cap\sobZ1p\Omega$. Therefore
\[
\int_\Omega v f_m = - \int_\Omega v \del{A:D^2 w_m + \divgText A\cdot\nabla w_m}
\rightarrow - \int_\Omega v\del{A:D^2 w + \divgText A \cdot\nabla w} = 0
\]
for all $v\in L^q(\Omega)$, whence $-\divgText(A\nabla w)=0$ and $w=0$
because of uniqueness. On the
other hand, \eqref{localW2p} yields $1\lesssim\|w\|_{L^p(\Omega)}$,
which is a contradiction. This thus shows the validity of \eqref{apriori-div}.

3. {\it Existence}.
It remains to show that there is a solution $W^2_{p,{\tt loc}}(\widetilde{\Omega})$
of \eqref{ext-div-form}.
If $\widetilde{f}\in L^2(\widetilde{\Omega})$, then the unique
solution $\widetilde{w}\in H^1_0(\widetilde{\Omega})$ of \eqref{ext-div-form}
belongs to $H^2_{\tt loc}(\widetilde{\Omega})$ and
\[
\|\widetilde{w}\|_{H^2(\Omega')} \lesssim \|\widetilde{w}\|_{H^1_0(\widetilde{\Omega})} + 
\|\widetilde{f}\|_{L^2(\widetilde{\Omega})} \lesssim
\|f\|_{L^2(\Omega)},
\]
for all $\Omega'$ compactly contained in $\widetilde{\Omega}$
\cite[Theorem 8.8]{DGilbarg_NTrudinger_2001a}.
We first let $p>2$ and lift the regularity of $\widetilde{w}$ to 
$W^2_{p,{\tt loc}}(\widetilde{\Omega})$ upon applying 
\cite[Lemma 9.16]{DGilbarg_NTrudinger_2001a} which gives the
estimate
\begin{equation}\label{Omega'}
\|\widetilde{w}\|_{W^2_p(\Omega')} \lesssim 
\|\widetilde{w}\|_{L^p(\widetilde{\Omega})} + \|\widetilde{f}\|_{L^p(\widetilde{\Omega})}.
\end{equation}
If $1<p<2$ instead, we approximate $\widetilde{f}\in L^p(\widetilde{\Omega})$
by a sequence $\{\widetilde{f}_m\}\subset L^2(\widetilde{\Omega})$. Since
$\widetilde{w}_m\in H^2_{\tt loc}(\widetilde{\Omega})\subset
\sob2{p,{\tt loc}}{\widetilde{\Omega}}$, we can 
apply the interior $W^2_p$ estimates of 
\cite[Theorem 9.11]{DGilbarg_NTrudinger_2001a} to deduce
\eqref{Omega'} again for $\widetilde{w}_m$. Moreover, \eqref{Omega'}
shows that $\{\widetilde{w}_m\}$ is a Cauchy sequence in
$W^2_p(\Omega')$ for any $\Omega'$, whence \eqref{Omega'} remains valid for
the limit $w$. This finishes the proof.
\end{proof}
%

\subsection{State Equations}
We resort to the fixed point argument in \cite[Section
2]{PSaavedra_RScott_1991} and 
\cite[Section 4.1.1]{HAntil_RHNochetto_PSodre_2014a}. 
It consists of three steps: defining a convex set which acts as
domain for the fixed point iterator,  linearizing the free boundary
problem by freezing one variable, and identifying conditions to guarantee a contraction on the convex set. 

To deal with second-order regularity we need to introduce, besides 
the space $\bbW^{1,1}_{\infty,p}$, the second order Banach subspace product
\[
    \bbW^{2,2}_{\infty,p} :=  \bbW^{1,1}_{\infty,p} 
\cap \big(\sob2\infty{0,1} \times \sob2 p\Omega \big),
\]
and endow both $\bbW^{1,1}_{\infty,p}$ and $\bbW^{2,2}_{\infty,p}$ with the norms
    \[
		\norm{\del{\gamma, y}}_{\bbW^{1,1}_{\infty,p}} 
		    := \del{1+\beta C_A} \normSZ{v}1p\Omega\normSZ{\gamma}1\infty{0,1} + \normSZ{y}1p\Omega ,
	\]	
and
	\[
		\norm{\del{\gamma, y}}_{\bbW^{2,2}_{\infty,p}} 
		    := \del{1+2C_A C_\#} \normS{v}2p\Omega\normS{\gamma}2\infty{0,1} + \normS{y}2p\Omega ,
	\]	
respectively; hereafter $C_A$ is a bound 
in $L^\infty(\Omega)$ on the
operator $A$ and its first and second order derivatives
\cite[Proposition 2.1]{HAntil_RHNochetto_PSodre_2014a}, and 
$C_\#$ is the constant in \eqref{apriori-div}.
To guarantee that the assumptions for the first-order regularity results in
\cite[Section 4.1.1]{HAntil_RHNochetto_PSodre_2014a} hold,
we must iterate on a subset of
\cite[Eq. (4.12)]{HAntil_RHNochetto_PSodre_2014a}
	\[
    \bbB_1  := \set{(\gamma,y) \in \bbW^{1,1}_{\infty,p} :\;
		\normSZ{y}{1}{p}{\Omega} \leq \beta C_A \normSZ v 1 p \Omega,\; \normSZ{\gamma}1\infty{0,1} \leq 1},
	\]
where $\beta$ is the inf-sup constant for $\BO$ in $\sob1p\Omega$
\cite[Proposition 4.1]{HAntil_RHNochetto_PSodre_2014a}.
For the purpose of finding a strong solution, we further restrict $\bbB_1$ as follows:
	\begin{align*}
		\bbB_2 := \set{ \del{\gamma, y} \in \bbB_1 \cap \bbW^{2,2}_{\infty,p} :\;
			\normS{y}{2}{p}{\Omega} \leq 2C_A C_\#
                        \normS v 2 p \Omega,\; 
			\normSZ{\gamma}2\infty{0,1} \leq 1} .
	\end{align*}

We linearize the free boundary problem by considering the following operator $T:\bbB_2 \rightarrow \bbW^{2,2}_{\infty,p}$ defined as
	\begin{equation} \label{eq:T_map}
		T(\gamma, y) :=  \del{T_1(\gamma, y), T_2(\gamma, y)}
                = \del{\gtl, \ytl}
                \quad \forall \del{\gamma,y} \in \bbB_2,
	\end{equation}
where $\gtl = T_1(\gamma, y) \in \sob2\infty{0,1}\cap\sobZ1\infty{0,1}$ is the unique solution to
	\begin{equation}
\begin{PHDDISABLED}
	\label{eq:t1_var}
\end{PHDDISABLED}
		- \kappa \totalD_{x_1}^2\gtl 
		    = A\sbr{\gamma}\grad (y+v) \cdot \nu + u
			\quad \mbox{in } \del{0,1},
	\end{equation}
and $\ytl = T_2(\gamma,y) \in \sob2p\Omega \cap \sobZ1p\Omega$ is the unique solution to
	\begin{equation}
\begin{PHDDISABLED}
	\label{eq:t2_var}
\end{PHDDISABLED}
      -\divg{}\big(A\sbr{T_1(\gamma, y)}\nabla\ytl\big)
      = \divg{}\big(A\sbr{T_1(\gamma, y)} \grad v \big)\quad \mbox{in } \Omega,
	\end{equation}
where $\divg{}A\sbr{T_1(\gamma, y)}$ is computed row-wise.
The operator $T$ maps $\bbB_1$ into itself provided $v$ and $u$ are
restricted to verify
\begin{equation}\label{B1-restriction}
|v|_{\sob1p\Omega} \le \frac{1-\theta_1}{\alpha C_E C_A (1+\beta C_A)},
\qquad
\|u\|_{L^2(0,1)} \le \frac{\theta_1}{\alpha},
\end{equation}
for some $\theta_1 \in (\beta C_A/(1+\beta C_A), 1)$ 
\cite[Lemma 4.3]{HAntil_RHNochetto_PSodre_2014a}. 
We now investigate additional conditions for $T$ to map $\bbB_2$ into itself.

\begin{lem}[range of $T$]\label{lem:B2_to_B2}
Let $C_S$ be the Sobolev embedding constant between $\sob2p\Omega$ and
  $\sob1\infty\Omega$.
The operator $T$ maps $\bbB_2$ to $\bbB_2$ if, in addition to
\eqref{B1-restriction}, the following relation holds 
\begin{equation}\label{eq:state_map}
    C_A C_S\del{1+2C_AC_\#}\normS{v}2p\Omega + \normLi{u}{0,1} \leq \kappa .
\end{equation}
\end{lem}
\begin{proof}
In view of \eqref{B1-restriction} we have $T(\bbB_2)\subset\bbB_1$.
Given $\del{\gamma,y} \in \bbB_2$, we readily get from \eqref{eq:t1_var}
\[
   \kappa\norm{ \totalD_{x_1}^2\gtl }_{L^\infty(0,1)} 
		    \le \norm{A\sbr{\gamma}}_{L^\infty(0,1)} \norm{\grad (y+v)}_{L^\infty(\Omega)} +\norm{u}_{L^\infty(0,1)} .
\]
As $\norm{A\sbr{\gamma}}_{L^\infty(0,1)} \le C_A$, by definition of
$C_A$, and for $p>2$, $\sob2p\Omega$ is continuously embedded in
$\sob1\infty\Omega$ with constant
$C_S$, we deduce 
\[
    \kappa\norm{ \totalD_{x_1}^2\gtl }_{L^\infty(0,1)} 
        \le C_A C_S \normS{y+v}2p\Omega + \norm{u}_{L^\infty(0,1)}
    \le \kappa
\]    
because of \eqref{eq:state_map}. This implies 
$\normSZ{\gtl}2\infty{0,1}\le 1$, which is consistent with $\bbB_2$.

To deal with \eqref{eq:t2_var}, we invoke the a priori estimate 
\eqref{apriori-div} with $f=\divg{}\big(A[\gtl]\nabla v\big)$
\[
    \normS{\ytl}2p\Omega 
        \le C_\# \del{\|A[\gtl]\|_{L^\infty(\Omega)}
            \normS{v}2p\Omega +
            \norm{\divg{}A\sbr{\gtl}}_{L^\infty(\Omega)}
            \normSZ{v}1p\Omega}.
\]
This gives
\[
 \normS{\ytl}2p\Omega \le 2C_A C_\# \normS{v}2p\Omega,  
\]
and, together with the previous bound on $\gtl$, yields
$\del{\gtl,\ytl} \in \bbB_2$, as asserted. 
\end{proof}
    \begin{rem}[boundedness of $\uop$]
     \rm
        We point out that, 
        within the context of the optimal control problem, the 
        $L^\infty$-estimate requirement on $\uop$ in \eqref{eq:state_map} can be satisfied. 
        The reason is that the variational inequality \eqref{eq:lin_uop_var_ineq} implies 
        that 
        \[
            \uop = \left\{ \begin{array}{ll}
                             -\frac{\sop}{\lambda} \ ,
                                     & \mbox{if } \lambda \uop + \sop = 0 \\
                             -\frac{\theta_1 \sop}{2\alpha \norm{\sop}_{L^2(0,1)}} \ ,
                                     & \mbox{if } \lambda \uop + \sop \neq 0 
                           \end{array}
                   \right.
        \]
        i.e. the optimal control $\uop$ is proportional to the adjoint function 
        $\sop$ which in turn is absolutely continuous, i.e 
        $\sop \in \sobZ11{0,1} \subset \Linf{0,1}$.    
    \end{rem}
\begin{thm}[second-order regularity of the state variables]
\label{thm:second_order_state}
Let $C_S$ be the Sobolev embedding constant between $\sob2p\Omega$ and $\sob1\infty\Omega$, and 
\begin{align*}
    \Lambda = \kappa^{-1}C_AC_S \big(1+2C_A C_\# \big)^2.
\end{align*}    
If, in addition to \eqref{B1-restriction} and \eqref{eq:state_map}, 
the function $v$ further satisfies 
\begin{align}\label{eq:state_contraction}
\begin{aligned}
    \normS{v}2p\Omega &\le (1-\theta_2) \Lambda^{-1} , 
\end{aligned}         
\end{align}
for some $\theta_2 \in (0,1)$,
then the map $T$ defined in \eqref{eq:T_map} is a contraction on
$\bbB_2$ with constant $1-\theta_2$ for all $u\in \Uad$. 
\end{thm}
\begin{proof}
Let $(\gamma_1,y_1) , (\gamma_2,y_2) \in \bbB_2$ with $(\gamma_1,y_1)
\neq (\gamma_2,y_2)$, and set $\delta\gamma := \gamma_1-\gamma_2,
\delta y := y_1 - y_2$. Combining \eqref{eq:T_map} and
\eqref{eq:t1_var} we get an equation for $\delta\gtl:=\gtl_1-\gtl_2$
\[
    -\kappa \totalD_{x_1}^2 \delta\gtl 
        = \del{A\sbr{\gamma_1} - A\sbr{\gamma_2}} \grad (y_1+v) \cdot
        \nu +  A\sbr{\gamma_2} \grad \delta y \cdot \nu . 
\]
Since $\delta\gamma(0)=\delta\gamma(1)=0$ we infer that 
$\delta\gamma'(x_1)=0$  for some $x_1\in (0,1)$ and
$|\delta\gamma|_{W^1_\infty(0,1)} \le |\delta\gamma|_{W^2_\infty(0,1)}$,
whence
\begin{equation}\label{eq:thm_second_a}
 \| \delta\gtl \|_{W^2_\infty(0,1)} = \normSZ{\delta\gtl}2\infty{0,1} 
        \le \kappa^{-1} C_A C_S
            \norm{(\delta\gamma,\delta y)}_{\bbW^{2,2}_{\infty,p}} .
\end{equation}

We next estimate $\delta\ytl:=\ytl_1 - \ytl_2$. In view of \eqref{eq:t2_var} we see that 
\begin{align*}
    - \divg{}\big(A\sbr{\gtl_1} \nabla \delta\ytl \big)
        = \divg{} \big( (A\sbr{\gtl_1}-A\sbr{\gtl_2}) \grad (\ytl_2+v) \big).
\end{align*}
Invoking the a priori estimate \eqref{apriori-div}, we deduce
\begin{align*}
    \normS{\delta\ytl}2p\Omega 
        &\le C_\# \norm{A\sbr{\gtl_1}-A\sbr{\gtl_2}}_{L^\infty(0,1)} 
        |\ytl_2+v|_{W^2_p(\Omega)} \\
        & + C_\# \norm{\divg{}
          \del{A\sbr{\gtl_1}-A\sbr{\gtl_2}}}_{L^\infty(0,1)} 
        |\ytl_2+v|_{W^1_p(\Omega)}.
\end{align*}
Since $\normS{\ytl_2}2p\Omega \le 2C_A C_\# \normS{v}2p\Omega$, we infer that 
\begin{align*}
    \normS{\delta\ytl}2p\Omega 
        \le 2C_A C_\# \big(1+2C_A C_\# \big) \normS{v}2p\Omega 
        \normS{\delta\gtl}2\infty{0,1},
\end{align*}   
and, applying \eqref{eq:thm_second_a}, we obtain
\begin{align}\label{eq:thm_second_b}
    \normS{\delta\ytl}2p\Omega \le
        2 \kappa^{-1} C_A^2 C_S
        C_\# \Big(1+2C_A C_\# \big)\normS{v}2p\Omega
        \norm{(\delta\gamma,\delta y)}_{\bbW^{2,2}_{\infty,p}}.
\end{align}
The definition of $\bbW^{2,2}_{\infty,p}$ norm, together with
\eqref{eq:thm_second_a} and \eqref{eq:thm_second_b}, leads to
\begin{align*}
    \norm{(\delta\gtl,\delta\ytl)}_{\bbW^{2,2}_{\infty,p}}
    \le \kappa^{-1}C_AC_S \big(1+2C_A C_\# \big)^2
    \normS{v}2p\Omega \norm{(\delta\gamma,\delta y)}_{\bbW^{2,2}_{\infty,p}},
\end{align*}
and \eqref{eq:state_contraction} gives
$\norm{(\delta\gtl,\delta\ytl)}_{\bbW^{2,2}_{\infty,p}}\le \theta_2
\norm{(\delta\gamma,\delta y)}_{\bbW^{2,2}_{\infty,p}}$,
which is the assertion.
\end{proof}
%

\subsection{Adjoint Equations}
We begin by assuming that $\del{\gop, \yop}$ belongs to
$\bbB_2$ and rewriting the adjoint equations
\eqref{eq:lin_rop_sop} in \emph{strong divergence} form in $\Omega$ 
\begin{subequations}\label{eq_adjoint}
\begin{equation}\label{eq:r_adjoint_strong}
  -\divg{} \big(A\sbr{\bar\gamma} \nabla \rop \big)
        = \HA{\mu}\delta\yop \del{1+\gop},
\end{equation}
and in $(0,1)$
  \begin{equation}\label{eq:s_adjoint_strong}
  \begin{aligned}
     -\kappa \totalD^2_{x_1}\sop 
	& = \delta\gop + \frac{\HA{\mu}}{2}\int_0^1 \abs{\delta\yop}^2 \dif x_2 
				\\
	&- \int_0^1 A_1\sbr{\gop}\grad\del{\yop + v} \cdot \grad \rop\dif x_2 
         + \totalD_{x_1} \int_0^1 A_2\sbr{\gop}\grad\del{\yop+v} 
           \cdot \grad \rop\dif x_2
  \end{aligned}
  \end{equation}
\end{subequations}
together with the boundary conditions $\rop = 0$ on $\Sigma$, $\rop = \sop$ on $\Gamma$, and $\sop(0) = \sop(1) = 0$.

\begin{thm}[second-order regularity of adjoint variables]
The solution $\del{\sop, \rop}$ to \eqref{eq_adjoint} satisfies 
$\del{\sop, \rop-E\sop}\in \bbW^{2,2}_{1,q}$ along with 
the following a-priori estimates
    \begin{align*}
        \normS{\sop}21{0,1} + \normS{\rop}2q\Omega
            \lesssim
                \normLo{ \delta\gop }{0,1}
                + \normLt{ \delta\yop }\Omega^2
                + \normL{\delta\yop }q\Omega,
    \end{align*}
provided the function $v$ satisfies
\begin{equation}\label{small-v}
4 C_A C_E (1+ 2C_\# ) \Big(2 (1+2C_A C_\# ) + \alpha 
(1+\beta C_A) \Big)
\|v\|_{W^2_p(\Omega)} \le 1.
\end{equation}
\end{thm}
\begin{proof}
Setting $\rop = \rop_0 + E\sop$, where $E:\sob{2}{1}{0,1} \cap
  \sobZ{1}{1}{0,1} \rightarrow \sob{2}{q}{\Omega}$ is the extension
  operator for  $q < 2$, we can rewrite \eqref{eq:r_adjoint_strong} as 
\begin{equation*}
     -\divg{} \big(A\sbr{\bar\gamma} \nabla\rop_0 \big) =
        \divg{}\big( A\sbr{\bar\gamma} \grad E\sop \big)
        + \HA{\mu}\delta\yop \del{1+\gop},
  \end{equation*}
with $\rop_0\restriction_{\bdy\Omega} = 0$.   
We apply \eqref{apriori-div} to obtain $\rop_0\in W^2_q(\Omega)$ and
   \begin{align*}
        \normS{\rop_0}2q\Omega 
           \le 2 C_\# C_E \normS{\sop}{2}{1}{0,1}
            + 2 C_\# \HA{\mu} \normL{\delta\yop}q\Omega,
    \end{align*}   
and similarly for $\sop$
    \begin{align*}
        \normSZ{\sop}21{0,1}
            \le
            \normLo{\delta\gop}{0,1} 
            + \frac{\HA{\mu}}2 \normLt{\delta\yop}\Omega^2 
            + 4 C_A \|\yop + v\|_{W^2_p(\Omega)} \|\rop\|_{W^2_q(\Omega)}.
    \end{align*}  
Using the fact $\yop \in \bbB_2$ we deduce               
    \begin{align*}
        \normSZ{\sop}21{0,1}
            \le
            \normLo{\delta\gop}{0,1} 
            + \frac{\HA{\mu}}2 \normLt{\delta\yop}\Omega^2
            + 4 C_A \big(1+2C_A C_\# \big)\normS{v}2p\Omega\normS{\rop}2q\Omega.
    \end{align*}
Recalling the estimate for $\normSZ{\sop}11{0,1}$ from 
\cite[Lemma 5.5]{HAntil_RHNochetto_PSodre_2014a}
\[
    \normSZ{\sop}11{0,1} \le \alpha \normLo{\delta\gop}{0,1} 
            + \frac{\alpha\HA{\mu}}{2}\normLt{\delta\yop}\Omega^2 
            + \alpha C_A \big(1+\beta C_A\big)
            \normSZ{v}1p\Omega\normSZ{\rop}1q\Omega ,
\]
and using that $\|\sop\|_{L^1(0,1)} \le \normSZ{\sop}11{0,1}$, we end
up with
\begin{align*}
    \normS{\sop}21{0,1} 
        \le
            (1+2\alpha) \normLo{\delta\gop}{0,1} 
            + \frac{\HA{\mu}}2 (1+2\alpha)\normLt{\delta\yop}\Omega^2
            + \lambda
            \normS{v}2p\Omega\normS{\rop}2q\Omega ,
\end{align*} 
where $\lambda = 2C_A\big(2 (1+2 C_A C_\# ) + \alpha (1+\beta C_A)\big)$.
Inserting the estimate for $\normS{\rop_0}2q\Omega$ into this
estimate, we get
\begin{align*}
    \normS{\sop}21{0,1} 
        &\le
            (1+2\alpha) \normLo{\delta\gop}{0,1} 
            + \frac{\HA{\mu}}2 (1+2\alpha)\normLt{\delta\yop}\Omega^2
            \\
            & + \lambda C_E(1+ 2C_\#)
            \normS{v}2p\Omega \normS{\sop}21\Omega 
            + 2\HA{\mu}\lambda C_\#\normS{v}2p\Omega \|\delta\yop\|_{L^q(\Omega)}
\end{align*} 
and the desired estimate for $\normS{\sop}21{0,1}$ follows from
\eqref{small-v}. This, and the relation $\rop=\rop_0+ E\sop$ yields
the remaining estimate for $\|\rop\|_{W^2_q(\Omega)}$, and concludes
the proof.
\end{proof}

\section{Discrete Optimal Control Problem} \label{s:oc_disc_fbp}
The goal of this section is to introduce the discrete counterpart of
the optimization problem \eqref{eq:ocfbp}. The discretization uses
the finite element method and is classical.

Let $\calT$ denote a geometrically conforming rectangular quasi-uniform triangulation of the fixed
domain $\Omega$ such that $\overline{\Omega} = \cup_{K \in \calT} K$ and $h \approx h_K$ be the meshsize of
$\calT$.  Additionally, let 
$0 = \zeta_0 < \zeta_1 < \ldots < \zeta_{M+1}=1$ be a partition of
$[0,1]$ with nodes $\zeta_i$ compatible with $\calT$. Consider the following finite dimensional spaces,  where the capital letters stand for discrete objects:
\begin{subequations}  \label{eq:disc_spaces}
	\begin{align}
		\bbV_h &:= \left\{ Y \in C^0(\bar\Omega) : Y\restriction_{K} \in \mathcal{P}^1(K), K \in \calT  \right\}, \\
		\mathring{\bbV}_h &:= \bbV_h \cap \sobZ{1}{p}{\Omega}, \\
   		\bbS_h &:= \Big\{ G \in C^0([0,1]) : G\restriction_{\sbr{\zeta_{i},\zeta_{i+1}}} 
        	\in \mathcal{P}^1(\sbr{\zeta_{i},\zeta_{i+1}}), \ 0 \le i \le M   \Big\} , \\
		\mathring{\bbS}_h &:= \bbS_h \cap \sobZ{1}{\infty}{0,1},\\
	\bbU_{ad} &:= \bbS_h \cap \Uad ,
\end{align}
\end{subequations}
and $\calP^1(D)$ stands for bilinear polynomials on an element 
$D=K \in \calT$ or linear polynomials on an interval $D=\intcc{\zeta_i, \zeta_{i+1}}$.
The spaces $\mathring\bbV_h, \mathring\bbS_h$ and $\bbU_{ad}$ in
\eqref{eq:disc_spaces} will be used to approximate the continuous
solutions $(y,\gamma,u)$ of \eqref{eq:ocfbp}. 
This discretization is classical \cite[Chapter
3]{SCBrenner_RLScott_2008a}, except perhaps for 
the $L^2$ constraint in $\bbU_{ad}$, which we enforce by
scaling the functions with their $L^2$-norm; for more details we refer
to Section~\ref{s:oc_numerics_fbp}.

Next we present a discrete analog of the continuous extension \eqref{eq:w1_wq_ext}, namely 
	\[
		E_h G := ({\cal S}_h \circ E) ( G ), \quad \forall G \in \mathring{\bbS}_h.
	\]
The caveat is that functions in $\sob 1 q \Omega$ are not necessarily
continuous. This issue is addressed by utilizing the Scott-Zhang
interpolation operator ${\cal S}_h:\sob1q\Omega \to \bbV_h$. This operator satisfies the optimal estimate \cite{SCBrenner_RLScott_2008a},
	\begin{align}  \label{eq:inter_est_ext}
		\normSZ{w - {\cal S}_h w}{1}{q}{\Omega} &\lesssim  h \normSZ{w}{2}{q}{\Omega},\, \forall w \in \sob2q\Omega,\, 1 \le q \le \infty.
	\end{align}

For functions in $\sob 1p\Omega$ with $p > 2$, $\sob1\infty{0,1}$ and
$\sob11{0,1}$ we will use the standard Lagrange interpolation
  operator $\calI_h$. This is justified by the Sobolev embedding theorems, i.e. we can identify functions in those spaces with their continuous equivalents. Moreover, the following optimal interpolation estimates hold, 
	\begin{subequations} \begin{align}\label{eq:inter_est_soln_a} 
		\normSZ{y - \calI_h y}1p\Omega &\lesssim  h \normSZ{y}2p\Omega, 
			\, \forall y \in \sob2p\Omega, \quad  2 < p
                        \le \infty, \\
                        \label{eq:inter_est_soln_b} 
		\normSZ{\gamma - \calI_h \gamma}1p{0,1} &\lesssim  h \normSZ{\gamma}2p{0,1},
			\, \forall \gamma \in \sob2p{0,1},\quad 1 \leq p \leq \infty.
	\end{align}\end{subequations}

Next we state the discrete counterpart of the optimal control
problem \eqref{eq:lin_cost} in its variational form: if 
$\delta G := G -\gamma_d$, $\delta Y := Y +v-y_d$, then minimize
\begin{subequations}
	\begin{equation}\label{eq:disc_lin_cost}
	  \calJ_h(G,Y,U) := \frac{1}{2}\normLt{\delta G}{0,1}^2 
			+ \frac{\HA{\mu}}{2}\normLt{\delta Y \sqrt{1+G}}{\Omega}^2 
			+ \frac{\lambda}{2}\normLt{U}{0,1}^2 , 
	\end{equation}
    subject to the discrete \emph{state equation} $\del{G,Y}\in \mathring{\bbS}_h \times \mathring{\bbV}_h$
    \begin{equation}\label{eq:disc_state_var_pde} 
		\BG\sbr{G,\Xi} + \BO\sbr{Y + v,Z+E_h\Xi;A\sbr{G}} 
			= \int_0^1 U \Xi \quad \forall (\Xi, Z) \in \mathring{\bbS}_h\times \mathring{\bbV}_h ,
	\end{equation}
    the state constraints
	\begin{align}\label{eq:disc_state_const}		
		\abs{G'} \leq 1  \quad\mbox{on  } \intoo{\zeta_i, \zeta_{i+1}},\ i=0,...,M-1 ,
	\end{align}
	and the control constraints \[ U \in \bbU_{ad}.\]
\end{subequations}

We point out that $Y\restriction_{\bdy\Omega}=0$ in \eqref{eq:disc_state_var_pde}. This is not the standard approach
in finite element literature because it requires knowing an extension of $v$ to $\Omega$; we adopt this approach 
to simplify the exposition.
We must include the following mild regularity assumptions on data in order to obtain an order of convergence:

\medskip
\begin{itemize}
\item[{$\bf (A_3)$}] The given data satisfy
$v\in\sob{2}{p}{\Omega}$, $\gamma_d\in\Ltwo{0,1}$ and $y_d\in\Ltwo{\Omega^*}$.
\end{itemize}

\medskip
Now let $\Uop$ denote the optimal control to \eqref{eq:disc_lin_cost},
whose existence will be shown in 
Theorem~\ref{eq:existence_discrete_control}, and $\del{\Gop,\Yop}$ be the optimal 
state, which satisfy discrete \emph{state equations} in variational form
\eqref{eq:disc_state_var_pde}. 
The discrete \emph{adjoint equations} in variational form read: find
$(\Sop,\Rop)$ such that  
$(\Sop,\Rop-E_h\Sop)\in \mathring{\bbS}_h \times \mathring{\bbV}_h$ 
and for every $(\Xi, Z) \in \mathring\bbS_h \times \mathring{\bbV}_h$, 
	\begin{equation}\label{eq:disc_adj_var_pde} 
		\BG\sbr{\Xi,\Sop} + \DO\sbr{\Xi, Z,\Rop; \Gop, \Yop}
			= \Big\langle \Xi, \delta\Gop +
                        \frac{\HA{\mu}}{2}\int_0^1 \abs{\delta\Yop}^2 \dif x_2 \Big\rangle
                        + \pair{Z, \HA{\mu}\delta\Yop\del{1+\Gop}}.
	\end{equation}	
Finally, the optimal control $\Uop$ satisfies the variational inequality
 	\begin{align} \label{eq:disc_var_ineq}
		\pairLt{ \calJ_h'(\Uop), U - \Uop }{0,1} &\ge 0 \quad
                \forall U \in \bbU_{ad},
	\end{align}
where $\calJ_h'(\Uop) = \Sop +  \lambda \Uop$. Therefore \eqref{eq:disc_var_ineq} reads
	\begin{align}  \label{eq:disc_lin_uop}
		\pairLt{ \Sop +  \lambda \Uop, U - \Uop}{0,1} &\ge 0  \quad \forall U \in \bbU_{ad} .
	\end{align}       

The following discrete estimates mimic the continuous inf-sup
\begin{PHDENABLED}
in \propref{prop:b_infsup}
\end{PHDENABLED}
\begin{PHDDISABLED}
\cite[Proposition 4.1]{HAntil_RHNochetto_PSodre_2014a}.
\end{PHDDISABLED}
\begin{prop}[discrete inf-sup] \label{prop:disc_b_infsup}
The following two statements hold:
\begin{enumerate}[(i)]
\item\label{item:disc_b_infsup_i}
There exists constant $0 < \alpha < \infty$ independent of $h$ such that
	\begin{subequations}
		\begin{align}\label{eq:disc_b1_infsup_a}
			\normSZ{G}{1}{\infty}{0,1} &\leq \alpha \sup_{0\neq \Xi\in \mathring{\bbS}_h}
				\frac{\BG\sbr{G,\Xi}}{\normSZ{\Xi}{1}{1}{0,1}} ,  \\
			\label{eq:disc_b1_infsup_b}
			\normSZ{S}{1}{1}{0,1} &\leq \alpha \sup_{0\neq \Xi \in \mathring{\bbS}_h}
				\frac{\BG\sbr{\Xi,S}}{\normSZ{\Xi}{1}{\infty}{0,1}} . 
		\end{align}
	\end{subequations}
\item\label{item:disc_b_infsup_ii}
There exists constant $0 < \beta < \infty$ independent of $h$ and
constants $Q < 2 < P$, $h_0 > 0$, such that for $p \in \sbr{Q,P}$ and $0 < h \le h_0$
	     \begin{align}\label{eq:disc_b2_infsup}
		  \normSZ{Y}{1}{p}{\Omega} &\leq \beta \sup_{0\neq Z \in \mathring{\bbV}_h}
		           \frac{\BO\sbr{Y,Z;A\sbr{G}}}{\normSZ{Z}{1}{q}{\Omega}}.
              \end{align}
\end{enumerate}
\end{prop}
\begin{proof}
We refer to \cite[Proposition 3.2]{PSaavedra_RScott_1991} for 
a proof of \eqref{eq:disc_b1_infsup_a} and to
\cite[Proposition 8.6.2]{SCBrenner_RLScott_2008a} 
for a proof of \eqref{eq:disc_b2_infsup}.
The technique of \cite{PSaavedra_RScott_1991} extends to
\eqref{eq:disc_b1_infsup_b}.
\end{proof}

\HA{The constant $\alpha$ is equal to $2/\kappa$, see \cite[Proposition 2.2]{PSaavedra_RScott_1991}.}

Existence and uniqueness of solutions to the \emph{state} and
\emph{adjoint equations} can be shown similarly to the continuous case 
\begin{PHDENABLED}
Corollary~\ref{cor:g_operator} and \thmref{thm:t_adj_contraction}
\end{PHDENABLED}
\begin{PHDDISABLED}
\cite[Corollary~4.6, and Theorem~5.6]{HAntil_RHNochetto_PSodre_2014a} 
\end{PHDDISABLED}
provided $U \in \bbU_{ad}$ and $\normSZ{v}{1}{p}{\Omega}$ is
small. We will next prove
existence of an optimal control $\Uop$ solving \eqref{eq:disc_lin_cost}.

    \begin{thm}[existence of optimal control]\label{eq:existence_discrete_control}
        There exists a discrete optimal control $\Uop\in\bbU_{ad}$ which solves
        \eqref{eq:disc_lin_cost}. 
    \end{thm}
    \begin{proof}
        The proof follows by using a minimizing sequence argument
        similar to the continuous proof 
        \cite[Theorem 5.1]{HAntil_RHNochetto_PSodre_2014a}.
        However, weak convergence of a minimizing sequence $\{U_n\}$
        yields strong convergence in finite dimensional spaces.
        Following \cite[Theorem 4.8]{HAntil_RHNochetto_PSodre_2014a} it is routine to show
        that the discrete control-to-state map is Lipschitz continuous. Together with this 
        Lipschitz continuity and the strong convergence of $U_n$, 
      we also obtain strong convergence of the associated state sequence $\{\del{G_n, Y_n}\}$. 
    \end{proof}

\section{A-priori Error Estimates: State and Adjoint Variables} 
\label{s:oc_disc_apriori}
The goal of this section is to derive a-priori error estimates between the continuous and discrete solutions of the \emph{state} and \emph{adjoint}
equations for given functions $u \in \Uad$ and $U \in \bbU_{ad}$. This is the content of Lemmas \ref{lem:gamma_apriori_1} through \ref{cor:s_apriori}.
These estimates are the stepping stone for the $L^2$
estimate of $\uop-\Uop$ in Theorem~\ref{thm:control_apriori_estimate}.

\begin{PHDENABLED}
First we rewrite the adjoint equations in a compact but equivalent form,
$\del{\sop, \rop}$ in $\sobZ{1}{1}{0,1} \times \sobZ{1}{q}{\Omega}$ satisfies the \emph{adjoint equation} in variational form
	\begin{align}\label{eq:lin_rop_sop} 
		\BG\sbr{\xi, \sop} + \DO\sbr{\xi, z, \rop + E\sop; \gop, \yop} 
				=\pair{\xi, \gop - \gamma_d} + \pair{z, \yop + v - y_d},
	\end{align}
for all $\del{\xi, z}$ in $\sob{1}{\infty}{0,1} \times \sob{1}{p}{\Omega}$ with $\del{\gop,\yop}$ set to $\del{\gamma(\uop), y(\uop)}$.
The parametrized form 
	\begin{equation}
		\DO\sbr{\xi, z, r; \gop, \yop} := \BO\sbr{z, r;A\sbr{\gop}} + \BO\sbr{\yop + v, r;\Dif A\sbr\gop \difdir\xi}.
	\end{equation}
$\DO$ was introduced while analysing the Fr\'echet differentiability of the control-to-state map \secref{s:g_differentiable}.
Moreover, the duality pairings on the right-hand-side of \eqref{eq:lin_rop_sop} are reduced to standard integrals due to the $L^2$-regularity imposed by the cost functional.

\end{PHDENABLED}

\begin{lem}[preliminary error estimate for $\gamma$]\label{lem:gamma_apriori_1} 
Given $u\in\Uad$ and $U\in\bbU_{ad}$, 
let $\del{\gamma,y}$ and $\del{G,Y}$ solve \eqref{eq:lin_st_const} and
\eqref{eq:disc_state_var_pde} respectively for $v \in \sob{2}{p}{\Omega}$ with 
$\normSZ{v}1p\Omega$ small. 
Then the following  error estimate for $\gamma - G$ holds 
   \begin{align*}
   	\normSZ{\gamma - G}1\infty{0,1}
  		&\lesssim h\normSZ{\gamma}2\infty{0,1} +
  			 \normSZ{y-Y}1p\Omega + \normLt{u-U}{0,1}.
   \end{align*}
\end{lem}
\begin{proof}
We use the \emph{discrete} inf-sup \eqref{eq:disc_b1_infsup_a}
to infer that
   \begin{align*}
   	\normSZ{\calI_h \gamma - G}1\infty{0,1} \lesssim 
	\sup_{0\neq \Xi\in \mathring{\bbS}_h} \frac{\BG\sbr{\calI_h\gamma-G,
			  \Xi}}{\normSZ{\Xi}11{0,1}} .
   \end{align*}
   Next, we rewrite  $\BG\sbr{\calI_h\gamma-G, \Xi} =
   \BG\sbr{\calI_h\gamma - \gamma, \Xi} + \BG\sbr{\gamma-G,\Xi}$,
and estimate the first
   term using H\"older's inequality and \eqref{eq:inter_est_soln_b}.
For the second term we set $w = y+v$ and $W = Y + v$, use that $\gamma$ and $G$ satisfy \eqref{eq:lin_st_const} and \eqref{eq:disc_state_var_pde} respectively,
   and the fact that $\normSZ{y}{1}{p}{\Omega} \lesssim
   \normSZ{v}{1}{p}{\Omega}$, to obtain
   \begin{align*}
   		\BG\sbr{\gamma-G, \Xi}
   			&= -\BO\sbr{w, E_h\Xi; A\sbr{\gamma}} + \BO\sbr{W, E_h\Xi; A\sbr{G}} +
  				\pair{u-U, \Xi} \\
  			&= \BO\sbr{w, E_h\Xi; -A\sbr{\gamma}+A\sbr{G}} 
  			    + \BO\sbr{W-w, E_h\Xi; A\sbr{G}} +
                            \pair{u-U, \Xi} \\
  			&\lesssim \del{
  			    \normSZ{\gamma - G}1\infty{0,1} \normSZ{v}1p\Omega
  			    + \normSZ{y-Y}1p\Omega
  			    + \normLt{u-U}{0,1} }\normSZ{\Xi}11{0,1} ,
   \end{align*}   
   where $\pair{u-U, \Xi}=\pairLt{u-U, \Xi}{0,1}$.
   Combining the above two estimates with the triangle inequality and 
   \eqref{eq:inter_est_soln_b}, we end up with
   \begin{align*}
   	\normSZ{\gamma - G}1\infty{0,1}
   		&\lesssim h\normSZ{\gamma}2\infty{0,1} \\
                &+  \normSZ{\gamma - G}1\infty{0,1} \normSZ{v}{1}{p}{\Omega}
			  +  \normSZ{y-Y}1p\Omega  + \normLt{u-U}{0,1}.  
   \end{align*}
   Using that $\normSZ{v}{1}{p}{\Omega}$ is small finally yields the desired result. 
\end{proof}

\begin{lem}[error estimate for $y$]\label{lem:y_apriori}
Given $u\in\Uad$ and $U\in\bbU_{ad}$, 
let $\del{\gamma,y}$ and $\del{G,Y}$ solve \eqref{eq:lin_st_const} and
\eqref{eq:disc_state_var_pde} respectively with 
$\normSZ{v}{1}{p}{\Omega}$ small. Then the following
estimate for $y - Y$ holds
   \begin{align*}
   	\normSZ{y - Y}1p\Omega
  		&\lesssim h\del{\normSZ \gamma 2\infty{0,1} + \normSZ y2p\Omega } 
		          + \normSZ{v}1p\Omega \normLt{u-U}{0,1} .
    \end{align*}
\end{lem}
\begin{proof}
We proceed as in Lemma \ref{lem:gamma_apriori_1}. We use the
\emph{discrete} inf-sup followed by the interpolation estimate 
\eqref{eq:inter_est_soln_a},
together with the state constraint $\normSZ{G}{1}{\infty}{{0,1}} \le 1$, to obtain
\begin{align*}
  \normSZ{\calI_h y - Y}1p\Omega &\lesssim 
  \sup_{0\neq Z \in \mathring{\bbV}_h}
  \frac{\BO\sbr{\calI_h y - Y,Z;A\sbr{G}}}{\normSZ{Z}{1}{q}{\Omega}}
  \\
  &\lesssim h\normSZ{y}2p\Omega 
  + \sup_{0\neq Z \in \mathring{\bbV}_h}
  \frac{\BO\sbr{y - Y,Z;A\sbr{G}}}{\normSZ{Z}{1}{q}{\Omega}}.
\end{align*}
	We handle the last term by using that $y$ and $Y$ are solutions to \eqref{eq:lin_st_const} and \eqref{eq:disc_state_var_pde}, i.e.
	\begin{align*}
		\BO\sbr{y -  Y, Z; A\sbr{G}}
			&= 
			    \BO\sbr{y + v, Z; A\sbr{G}} - \BO\sbr{Y + v, Z; A\sbr{G}} \\
			&=
			    \BO\sbr{y+v, Z; A\sbr{G}-A\sbr{\gamma}},
    \end{align*}
followed by the bound $\normSZ{y}1p\Omega \lesssim \normSZ{v}1p\Omega$ in the definition of 
    $\bbB_1$ to yield    
    \begin{align*}    
        \BO\sbr{y -  Y, Z; A\sbr{G}}
            &\lesssim
                \normSZ{\gamma-G}1\infty{0,1}\normSZ{v}1p\Omega \normSZ{Z}1q\Omega.
    \end{align*}                
Combining the above estimates with \lemref{lem:gamma_apriori_1}, 
and using the triangle inequality in conjunction with 
\eqref{eq:inter_est_soln_a}, we obtain 
	\begin{align*}
		\normSZ{y - Y}1p\Omega 
			&\lesssim h \normSZ y2p\Omega 
			    + \normSZ{\gamma - G}1\infty{0,1} \normSZ{v}{1}{p}{\Omega} \\
			&\lesssim 
			    h\del{\normSZ{\gamma}2\infty{0,1}\normSZ{v}1p\Omega + \normSZ y2p\Omega} 
			    +\normLt{u-U}{0,1} \normSZ{v}{1}{p}{\Omega} \\
			 &\quad + \normSZ{y-Y}{1}{p}{\Omega}\normSZ{v}1p\Omega . 
	\end{align*}
The desired estimate is a consequence of the smallness
assumption on $\normSZ{v}{1}{p}{\Omega}$.
\end{proof}

\begin{lem}[error estimate for $\gamma$]\label{cor:gamma_apriori}
Given $u\in\Uad$ and $U\in\bbU_{ad}$, 
let $\del{\gamma,y}$ and $\del{G,Y}$ solve \eqref{eq:lin_st_const} and
\eqref{eq:disc_state_var_pde} respectively with
$\normSZ{v}{1}{p}{\Omega}$ small. Then the following 
error estimate for $\gamma - G$ holds
   \begin{align*}
   	\normSZ{\gamma - G}1\infty{0,1}
  		&\lesssim h\del{\normSZ{\gamma}2\infty{0,1} 
  		                + \normSZ y2p\Omega} 
  		    + \normLt{u-U}{0,1}  .
   \end{align*}
\end{lem}
\begin{proof} The assertion follows by combining \lemref{lem:y_apriori} with \lemref{lem:gamma_apriori_1}.
\end{proof}

\begin{lem}[preliminary error estimate for $s$]\label{lem:s_apriori_1}
Given $u\in\Uad$ and $U\in\bbU_{ad}$, 
let $\del{s,r - Es} \in \sobZ 1 1 {0,1} \times \sobZ 1 q \Omega$ 
satisfy the continuous adjoint system 
\eqref{eq:lin_rop_sop}, 
and $\del{S, R-E_hS}\in\mathring{\bbS}_h\times\mathring{\bbV}_h$
satisfy the discrete counterpart
\eqref{eq:disc_adj_var_pde}. Then the following error estimate for $s-S$ is valid
	\begin{align*}
		&\normSZ{s-S}11{0,1} \lesssim 
			    h \normSZ{s}21{0,1}
			    +  \del{1 + \normSZ{r}1q\Omega  \normSZ{v}{1}{p}{\Omega} }  
			             \normSZ{\gamma-G}1\infty{0,1}  \\
			&\qquad
			 + \del{\normLt{\delta y}\Omega + \normSZ{y-Y}1p\Omega
    			    + \normSZ{r}1q\Omega} \normSZ{y-Y}1p\Omega  
                            +  \normSZ{r-R}1q\Omega \normSZ{v}{1}{p}{\Omega}.
	\end{align*} 
\end{lem}
\begin{proof}
We again employ the \emph{discrete} inf-sup \eqref{eq:disc_b1_infsup_b}, 
now taking the form
	\begin{align*}
		\normSZ{\calI_h s - S}11{0,1} 
   			&\lesssim \sup_{0\neq \Xi\in \mathring{\bbS}_h} 
			      \frac{\BG\sbr{\Xi,\calI_h s-S}}{\normSZ{\Xi}1\infty{0,1}} \\
			&\lesssim 
			    h\normSZ{s}21{0,1} 
			  + \sup_{0\neq \Xi\in \mathring{\bbS}_h} 
			      \frac{\BG\sbr{\Xi,s-S}}{\normSZ{\Xi}1\infty{0,1}},
	\end{align*}
where the last inequality follows by adding and subtracting $s$,
the continuity of $\BG$, and the interpolation estimate 
\eqref{eq:inter_est_soln_b}
for $s - \calI_h s$. It remains to control the last term. We use that $s$ and $S$ satisfy 
equations \eqref{eq:lin_rop_sop} and \eqref{eq:disc_adj_var_pde} 
to obtain
	\begin{align*}
		\BG\sbr{\Xi, s-S} 
			&= \pair{\Xi, \gamma-\gamma_d} - \pair{\Xi, G - \gamma_d} \\
			&\quad+ \HA{\mu}\Big\langle\Xi, \frac{1}{2}\int_0^1\abs{y+v-y_d}^2 \dif x_2\Big\rangle - \HA{\mu}\Big\langle\Xi, \frac{1}{2}\int_0^1\abs{Y+v-y_d}^2 \dif x_2\Big\rangle \\
			&\quad
			- \BO\sbr{y+v,r;DA\sbr\gamma\difdir{\Xi}} 
      			              + \BO\sbr{Y + v, R; DA\sbr{G}\difdir{\Xi}}
      	      \\
            &= \pair{\Xi, \gamma - G} 
                + \HA{\mu}\Big\langle \Xi, \frac{1}{2}\int_0^1 \del{y-Y}
                   \del{y+Y+2v-2y_d} \dif x_2
                  \Big\rangle \\
            &\quad       
                - \BO\sbr{y-Y, r; DA\sbr\gamma\difdir{\Xi}}                              
                -\BO\sbr{Y+v, r; \del{DA\sbr\gamma - DA\sbr G}\difdir{\Xi}} 
                \\
            &\quad
                -\BO\sbr{Y+v, r-R; DA\sbr G\difdir{\Xi}}.
\end{align*}               
Consequently, after normalization $\normSZ{\Xi}1\infty{0,1} = 1$, we infer that
\begin{align*}              
      \abs{\BG\sbr{\Xi, s-S}}
              \lesssim
                 \normLo{\gamma - G}{0,1}   
                 + \HA{\mu}\normLt{y-Y}\Omega 
                     \del{2\norm{y+v-y_d}+\normLt{Y-y}\Omega} \\
                 + \normSZ{y-Y}1p\Omega\normSZ{r}1q\Omega 
                 + \normSZ{v}1p\Omega\del{
                     \normSZ{r}1q\Omega\normSZ{\gamma - G}1\infty{0,1} + 
                     \normSZ{r-R}1q\Omega 
                     } . 
    \end{align*}
The desired result for $\normSZ{\calI_h s - S}11{0,1}$ follows
after combining the above estimate with $\normLo{\gamma-G}{0,1} \lesssim \normSZ{\gamma - G}1\infty{0,1}$ and $\normLt{y-Y}\Omega \lesssim
    \normSZ{y-Y}1p\Omega$. 
Finally, applying the triangle inequality in conjunction with
  \eqref{eq:inter_est_soln_b} we deduce the asserted estimate for
$\normSZ{s - S}11{0,1}$.
\end{proof}

\begin{lem}[error estimate for $r$]\label{lem:r_apriori_est}
Given $u\in\Uad$ and $U\in\bbU_{ad}$, 
let $\del{s,r - Es} \in \sobZ 1 1 {0,1} \times \sobZ 1 q \Omega$ 
satisfy the continuous adjoint system 
\eqref{eq:lin_rop_sop}, 
and $\del{S, R-E_hS}\in\mathring{\bbS}_h\times\mathring{\bbV}_h$
satisfy the discrete counterpart \eqref{eq:disc_adj_var_pde}.
Then the following a-priori error estimate for $r-R$ holds
	\begin{align*}
		\normSZ{r-R}1q\Omega 
		    & \lesssim 
		        h\del{\normSZ{s}21{0,1} + \normSZ{r}2q\Omega} \\
		    &\quad
		        + \del{1+ \del{1+\normSZ{v}1p\Omega}\normSZ{r}1q\Omega
		               + \normL{\delta y}q\Omega  }
		               \normSZ{\gamma - G}1\infty{0,1} \\
		    &\quad
		        + \del{1+ \normLt{\delta y}\Omega + 
		               \normSZ{y-Y}1p\Omega +  \normSZ{r}1q\Omega }
		               \normSZ{y - Y}1p\Omega.
	\end{align*}
\end{lem}
\begin{proof}
Since the discrete inf-sup \eqref{eq:disc_b2_infsup} is for functions in $\mathring\bbV_h$, we write $r = r_0 + E s$, and $ R = R_0 + E_h S $, with $r_0 \in \sobZ1q\Omega$ and $R_0 \in \mathring\bbV_h$, to obtain 

    \begin{align*}
        \normSZ{r - R}1q\Omega 
            &\leq
                \normSZ{r_0 - \mathcal{S}_h r_0}1q\Omega 
                + \normSZ{\mathcal{S}_h r_0 - R_0}1q\Omega 
                + \normSZ{Es - E_h S}1q\Omega .
    \end{align*}
Consequently, applying \eqref{eq:disc_b2_infsup}   
    \begin{align*}                
        \normSZ{\mathcal{S}_h r_0 - R_0}1q\Omega 
           &\lesssim
                h\normSZ{r_0}2q\Omega 
                + \sup_{0\neq Z \in \mathring{\bbV}_h}
		           \frac{\BO\sbr{Z, r_0 - R_0;A\sbr{G}}}{\normSZ{Z}{1}{p}{\Omega}},
    \end{align*}
where we have added and subtracted $r_0$.
Moreover, we handle the last term as before, i.e.  
	\begin{align*}
		\BO\sbr{Z, r_0-R_0; A\sbr{G}} 
			&= 
			  \BO\sbr{Z, r_0 + Es; A\sbr\gamma} 			  - \BO\sbr{Z, R_0 + E_hS; A\sbr{G}} \\
			  &\quad
			  + \BO\sbr{Z, r_0+Es; A\sbr{G}-A\sbr\gamma} 
			  + \BO\sbr{Z, E_hS-Es; A\sbr{G}} .
    \end{align*}
Invoking the \emph{adjoint equations} \eqref{eq:lin_rop_sop} and 
\eqref{eq:disc_adj_var_pde}, we see that
    \begin{align*}
        \BO\sbr{Z, r_0 + Es; A\sbr\gamma} &- \BO\sbr{Z, R_0 + E_hS; A\sbr{G}}\\ 
         &= \HA{\mu}\pair{\del{y + v - y_d}\del{1+\gamma}, Z} 
           - \HA{\mu}\pair{\del{Y + v - y_d}\del{1+G}, Z} \\
         &= \HA{\mu}\pair{y-Y, Z} + \HA{\mu}\pair{y\gamma - YG, Z}
            + \HA{\mu}\pair{\del{v-y_d}\del{\gamma-G},Z}  .
    \end{align*}     
    Since $y\gamma - YG = y\del{\gamma-G}-\del{Y-y}G$, 
    after normalization $\normSZ{Z}1p\Omega = 1$ and using 
    \eqref{eq:disc_state_const}, we obtain 
    \begin{align*}
        \abs{\BO\sbr{Z, r_0-R_0; A\sbr{G}}}
			&\lesssim 
			  \normL{y-Y}q\Omega  
			  +\normSZ{\gamma-G}1\infty{0,1} 
			      \del{ \normSZ{r}1q\Omega + \normL{\delta y}q\Omega } \\
			&\quad  + \normSZ{Es - E_hS}1q\Omega.
	\end{align*}
Combining this together with $\normSZ{Es - E_hs}1q\Omega \lesssim h\normSZ{s}21{0,1}$ and $\normSZ{E_hs - E_hS}1q\Omega \lesssim \normSZ{s-S}11{0,1}$, we end up with
	\begin{align*}
		\normSZ{r-R}1q\Omega 
			&\lesssim 
			    h \del{\normSZ{s}21{0,1}  + \normSZ{r}2q\Omega }  \\
		    &\quad + \normL{y - Y}q\Omega 
			    + \normSZ{\gamma-G}1\infty{0,1}  
			       \del{ \normSZ{r}1q\Omega + \normL{\delta y}q\Omega } \\
			&\quad + \normSZ{s-S}11{0,1}.	
	\end{align*} 
Finally, under the smallness assumption on $\normSZ{v}{1}{p}{\Omega}$ and
$\normL{y-Y}q\Omega \lesssim \normSZ{y-Y}1p\Omega$, \lemref{lem:s_apriori_1} yields the desired result.
\end{proof}

\begin{lem}[error estimate for $s$] \label{cor:s_apriori} The following a-priori estimate for $s-S$ holds
         \begin{align*}
		\normSZ{s-S}11{0,1} 
		&\lesssim
                h \del{\normSZ{\gamma}2\infty{0,1} + \normSZ{y}2p\Omega 
                 + \normSZ{s}21{0,1} + \normSZ{r}2q\Omega } +  \normLt{u - U}{0,1}.
	\end{align*} 
\end{lem}
\begin{proof} We use Lemmas \ref{lem:s_apriori_1} and \ref{lem:r_apriori_est}, to obtain
  \begin{align*}
        \normSZ{s - S}11{0,1} 
            &\lesssim 
                h \del{\normSZ{s}21{0,1} + \normSZ{v}1p\Omega\del{\normSZ{s}21{0,1} + \normSZ{r}2q\Omega}} \\
            & +
               \del{c_1 + \normSZ{v}1p\Omega\del{c_1+c_3}}\normSZ{\gamma - G}1\infty{0,1} \\
            & +
                \del{c_2 + \normSZ{y-Y}1p\Omega +  \normSZ{v}1p\Omega \del{1+c_2 + \normSZ{y-Y}1p\Omega}}\normSZ{y - Y}1p\Omega 
    \end{align*}
where 
    \begin{align*}
        c_1 &= 1 + \normSZ{r}1q\Omega \normSZ{v}1p\Omega,  \\
        c_2 &=  \normSZ{r}1q\Omega + \normLt{\delta y}\Omega, \\
        c_3 &= \normSZ{r}1q\Omega + \normL{\delta y}q\Omega .
    \end{align*}    
The assertion follows by applying Lemmas \ref{lem:y_apriori} and \ref{cor:gamma_apriori}, together with $\normSZ{v}1q\Omega \le 1$.
\end{proof}

\section{A-priori Error Estimates: Optimal Control}\label{s:control_numerics}
Next we derive the a-priori error estimate between $\uop$ and $\Uop$. 
\begin{thm}[error estimate for $u$]\label{thm:control_apriori_estimate}
Let both $h_0$ and $\normSZ{v}1p\Omega$ be sufficiently small. If $h \le h_0$, then
\begin{align}  \label{eq:error_est_control}
   		 \normLt{\uop - \Uop}{0,1}
		 &\leq \frac{4}{\lambda} \normLt{s(\Uop) - S(\Uop)}{0,1} ,
	\end{align}  
where $s(\Uop)$ is the solution of the continuous adjoint equation 
\eqref{eq:lin_rop_sop} with $\del{\gamma(\Uop), y(\Uop)}$ solutions of the 
state equation \eqref{eq:lin_st_const} with control $\Uop$, and $S(\Uop)$ is the 
solution of the discrete adjoint equation \eqref{eq:disc_adj_var_pde}.
\end{thm}
\begin{proof} The proof relies primarily on the continuous quadratic growth condition \eqref{eq:gradJ_quad_growth} and on the continuous and discrete first-order optimality conditions \eqref{eq:lin_uop_var_ineq} and \eqref{eq:disc_var_ineq}. Since $\Uop \in \bbU_{ad}$ is admissible, according to (\ref{eq:disc_spaces}e), replacing $u$ by 
$\Uop$ in $\eqref{eq:gradJ_quad_growth}$ (\cite{ECasas_FTroeltzsch_2002a}) we get
	\begin{align*}
		\frac{\lambda}{4} \normLt{\Uop-\uop}{0,1}^2 
			 &\leq \pairLt{\calJ'(\Uop)  - \calJ'(\uop), \Uop - \uop}{0,1}  .
        \end{align*}
        Adding and subtracting $\calJ_h'(\Uop)$ gives
        \begin{align*}			
                 \frac{\lambda}{4} \normLt{\Uop-\uop}{0,1}^2 
			&\leq \pairLt{\calJ'(\Uop) - \calJ_h'(\Uop) , \Uop - \uop}{0,1} \\
               &\quad\quad	+ \pairLt{ \calJ_h'(\Uop) , \Uop - \uop}{0,1} 
			          + \pairLt{ \calJ'(\uop), \uop - \Uop }{0,1} . \\
		\end{align*}
    Since $\pairLt{\calJ'(\uop),\uop-\Uop}{0,1} \le 0$, according to 
    \eqref{eq:var_ineq}, we deduce
		\begin{align*}	          
			\frac{\lambda}{4} \normLt{\Uop-\uop}{0,1}^2 &\le \pairLt{\calJ'(\Uop) - \calJ_h'(\Uop), \Uop - \uop}{0,1} \\
				  &\quad + \pairLt{\calJ_h'(\Uop), \Uop - \uop}{0,1} .
       \end{align*}
       Add and subtract $\calP_h\uop$, the $L^2$ orthogonal projection of $\uop$
       onto $\bbU_{ad}$, to get     
        \begin{align*}			
        	\frac{\lambda}{4} \|\Uop &- \uop\|_{L^2(0,1)}^2 
        		\leq \pairLt{\calJ'(\Uop) - \calJ_h'(\Uop), \Uop - \uop}{0,1} \\
				& + \pairLt{\calJ_h'(\Uop), \calP_h \uop  - \uop }{0,1} +  \pairLt{\calJ_h'(\Uop), \Uop - \calP_h \uop}{0,1}.
		\end{align*}		
		Since $\calJ_h'(\Uop) \in \bbS_h$ the middle term vanishes. In view of \eqref{eq:disc_var_ineq} and the fact that $\calP_h\uop \in \bbU_{ad}$, we deduce 
		$\pairLt{\calJ'_h\of{\Uop}, \Uop - \calP_h\uop}{0,1} \leq 0$ and 
        \begin{align*}					
          \frac{\lambda}{4} \normLt{\Uop-\uop}{0,1}^2
          &\le \pairLt{\calJ'(\Uop) - \calJ_h'(\Uop), \Uop - \uop}{0,1},
        \end{align*}           
    The explicit expressions $\calJ'(\Uop) = \lambda \Uop + s(\Uop)$ and 
    $\calJ_h'(\Uop) = \lambda \Uop + S(\Uop)$ yield
        \begin{align*}
            \frac{\lambda}{4} \normLt{\Uop-\uop}{0,1}^2 
             &\le \pairLt{s(\Uop) - S(\Uop), \Uop - \uop}{0,1} ,
        \end{align*}
    which imply the desired estimate \eqref{eq:error_est_control}.          
\end{proof}

\begin{cor}[rate of convergence]\label{cor:final_estimate}
    Let both $h_0$ and $\normSZ{v}1p\Omega$ be sufficiently 
    small. Furthermore, let $(s(\Uop),r(\Uop))$ be the solutions of the 
    continuous adjoint equation \eqref{eq:lin_rop_sop} with 
    $\del{\gamma(\Uop), y(\Uop)}$ solutions for the continuous state equation 
    \eqref{eq:lin_st_const} with control $\Uop$. Let $\del{S(\Uop),R(\Uop)}$ solve
    the discrete adjoint equation \eqref{eq:disc_adj_var_pde} with 
    \linebreak
    $\del{G(\Uop),Y(\Uop)}$ solutions for the discrete state equation
    \eqref{eq:disc_state_var_pde} with control $\Uop$.  
    If $h \le h_0$, then there is a constant $C_0 \ge 1$, depending on    
    $\normS{\gamma}2\infty{0,1}$, $\normS{y}2p\Omega$, $\normS{s}21{0,1}$, 
    $\normS{r}2q\Omega$, $\normLt{\gamma_d}{0,1}$, 
    $\normLt{y_d}\Omega$, such that 
        \begin{equation}\label{eq:final_estimate}
        \begin{aligned}
            &\normSZ{\gamma(\Uop)-G(\Uop)}1\infty{0,1} 
              + \normSZ{y(\Uop)-Y(\Uop)}1p\Omega  \\
            & \quad+ \normSZ{s(\Uop)-S(\Uop)}11{0,1} 
              + \normSZ{r(\Uop)-R(\Uop)}1q\Omega
                          + \lambda \normLt{\uop-\Uop}{0,1} \le C_0 h .
        \end{aligned}   
        \end{equation}  
\end{cor}
\begin{proof}
    We combine the estimate
    \[
       \normLt{s(\Uop) - S(\Uop)}{0,1} \leq \normSZ{s(\Uop) - S(\Uop)}11{0,1},
    \]
    with \lemref{cor:s_apriori} for $u = U = \Uop$
    \[
        \normLt{s(\Uop) - S(\Uop)}{0,1} \leq C_1 h
    \]
    where the constant $C_1$ has the same dependencies as $C_0$. This together with \eqref{eq:error_est_control},
    implies the error estimate for $\normLt{\uop-\Uop}{0,1}$ in \eqref{eq:final_estimate}. For the 
    remaining estimates in \eqref{eq:final_estimate} set 
    $u = \Uop$, and $U = \Uop$ in Lemmas \ref{lem:y_apriori}, 
    \ref{cor:gamma_apriori}, \ref{lem:r_apriori_est} and 
    \ref{cor:s_apriori} to complete the proof.
\end{proof}

\begin{rem}[linear rate]
\label{rem:linear_rate}
    The first-order convergence rate of \eqref{eq:final_estimate} is 
    optimal for a piecewise-linear finite element discretization of $\del{\gamma,y,s,r}$.
    For a control $u$ in $L^2$, one might expect an increased rate of convergence. For example, 
    it would be possible to use the standard Aubin-Nitsche duality argument if we were in a traditional linear setting to obtain 
    \[
        \normLt{s(\Uop) - S(\Uop)}{0,1} \leq h^{1/2}\normSZ{s(\Uop) - S(\Uop)}11{0,1},
    \] 
which in turn would yield an optimal rate of convergence $h^{3/2}$ for $\uop-\Uop$ in the proof of Corollary~\ref{cor:final_estimate}. \HA{The duality method fails in our setting we provide an explanation below. Recall $s(\bar{U}) = s(\gamma,y,r,\bar{U})$. Adding and subtracting 
$s(\bar G,\bar Y,\bar R,\bar{U})$ and using triangle inequality leads to 
\begin{align*}
    \| s(\bar{U}) - S(\bar{U}) \|_{L^2(0,1)} 
    &\le \| s(\gamma,y,r,\bar{U}) - s(\bar G,\bar Y,\bar R,\bar{U}) \|_{L^2(0,1)} \\
    &\quad+ \| s(\bar G,\bar Y,\bar R,\bar{U}) - S(\bar{U}) \|_{L^2(0,1)} 
    = \textsf{I} + \textsf{II} . 
\end{align*}
\RHN{Using a duality argument it is possible to get 
$
    \textsf{II} \lesssim h^{3/2} .
$
}
In our case the troublemaker is \textsf{I}. Using equation \eqref{eq:lin_rop_sop}, we obtain 
\begin{align*}
    \textsf{I} &\le \| \gamma(\bar{U}) - G(\bar{U}) \|_{L^2(0,1)} + \\
               &\quad \Big\| \int_0^1 \nabla (y(\bar{U}) + v) A_1[\gamma(\bar{U})] \nabla r(\bar{U}) - 
                \int_0^1 \nabla (\bar Y + v) A_1[\bar G] \nabla \bar R \Big\|_{L^2(0,1)}  + ... 
\end{align*} 
In view of Lemma 5.2 and 5.3 it is impossible to get more than O(h) 
for} \RHN{second term.
We are thus left merely with the Sobolev embedding
$
        \normLt{s(\Uop) - S(\Uop)}{0,1} \lesssim \normSZ{s(\Uop) - S(\Uop)}11{0,1},
$
which yields the linear rate of convergence for $\uop - \Uop$.}
\end{rem}

\begin{rem}[dependence on $\kappa$]
Since $\kappa$ is the ellipticity constant for $s$, then the estimate of $u$ is inversely proportional to the surface tension coefficient $\kappa$ in view of \eqref{eq:error_est_control}.  
\end{rem}

\section{Simulations} \label{s:oc_numerics_fbp}
\HA{
In our computations we assume the cost functional $\calJ$ in \eqref{eq:lin_cost} to be independent of $y$ and $y_d$, i.e. $\mu = 0$; we thus have 
\begin{equation}\label{eq:J_used}
    \costF := \frac{1}{2}\normLt{\gamma - \gamma_d}{0,1}^2 +  \frac{\lambda}{2}\normLt{u}{0,1}^2 . 
\end{equation}     
Our goal is to compute an approximation to the optimization problem presented in \secref{eq:ocfbp} with the cost functional \eqref{eq:J_used}, the Dirichlet data $v = x_2(1-x_2)(1-2x_1)$ applied to the entire boundary of $\Omega$ and surface tension coefficient $\kappa = 1$. 
}



We discretize the state variables $\del{\gamma,y}$, the adjoint variables $\del{s,r}$ and the control 
$u$ using piecewise bi-linear finite elements. We remark that in our case the \emph{first optimize then discretize} approach is equivalent to \emph{first discretize then optimize} (see \cite[p. 160-164]{MHinze_RPinnau_MUlbrich_SUlbrich_2009a}). To solve 
the state equations we use an affine invariant Newton strategy from
\cite[NLEQ-ERR, p. 148-149]{PDeuflhard2004a} because of its local
quadratic convergence.
The weak adjoint equations \eqref{eq:lin_rop_sop}, or
\eqref{eq_adjoint} in strong form, involve the 
coupling between the 2d bulk and 1d interface. 
This seemingly complicated coupling might entail an unusual assembly 
procedure and geometric mesh restrictions to evaluate 
integrals in \eqref{int-in-x2}. This is fortunately not
the case because the matrix of the adjoint system happens to be the
transpose of the Jacobian of the state equations, according to
\lemref{L:Jacobian-state}, which is available to us from the Newton
method. The assembly can thus be done with ease.

\HA{
We use a gradient based optimization method in MATLAB$^\copyright$, \emph{fmincon}, 
to solve the optimization problem. At every iteration we need to solve the state and adjoint equations and assemble $\calJ'$. This is carried out using deal.II
finite element library \cite{WBangerth_RHartmann_GKanschat_2007a}. In order to solve the linearized state algebraic system (Newton's method) and the linear adjoint systems we use a direct solver. 
}

%

  
\RHN{
We present three examples (see Figure \ref{fig:desired_sine}). Example~1 illustrates the a priori estimates
from Corollary~\ref{cor:final_estimate}; Examples~2 and 3 deal with
unconstrained optimization problems
which lie outside our theory. Examples 1 and 2 have a smooth
$\gamma_d$ whereas Example~3 has a non-smooth target $\gamma_d$.} 
\begin{figure}[h!]
\centering
\includegraphics[width=0.22\textwidth]{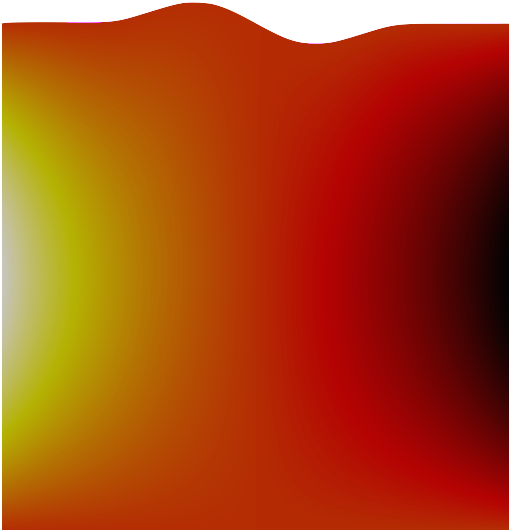}
\hskip1.cm
\includegraphics[width=0.22\textwidth]{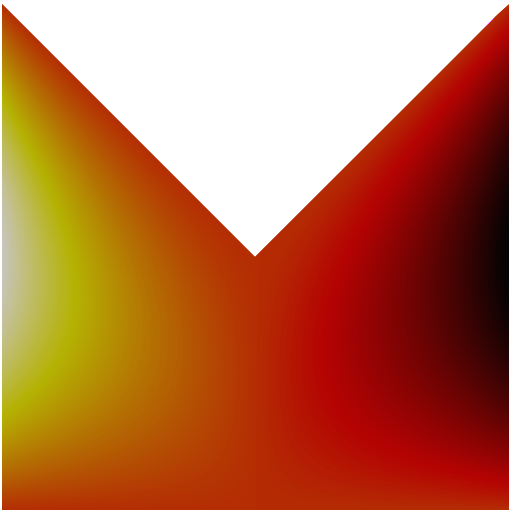}
\caption{\RHN{The desired shape of free boundary $\gamma_d = \frac{1}{16
    \pi}  \sin (2\pi x_1) - \frac{1}{16 \pi}  \sin(4 \pi x_1) +
  \frac{1}{32\pi} \sin(6 \pi x_1)$
for Examples 1 and 2 (left) and the inverted hat function $\gamma_d$
of Example 3 (right). The colors indicate the state $y$ corresponding to the 
configuration $\gamma_d$. In contrast to the former smooth profile,
the latter one $\gamma_d$ is not
achievable because the optimal solution $\bar\gamma$ satisfies
$\bar\gamma \in \sob2\infty{0,1}$.}}
\label{fig:desired_sine}
\end{figure}

\HA{We are interested in the following quantities:}
\begin{enumerate}[$\bullet$]
    \item The cost function value $\calJ\of{\bar u}$.
    \item The smallest eigenvalue of $\calJ''\of{\bar u}$, representing the 
    constant $\delta$ in the 2nd order sufficient condition $\calJ''(\uop)h^2 \ge \delta \normLt{h}{0,1}^2$. This metric is obtained in MATLAB$^\copyright$ through the approximated Hessian provided by the \emph{fmincon} function. \HA{For all three examples, we observe that this eigenvalue is close to 1. }
    \item The discrete $L^2$ norm of the optimal control $\bar u$ is equal to 
    $\del{\Uop^T M \Uop}^{1/2}$, where $M$ denotes the mass matrix corresponding
    to 1d problem in the interval $\intoo{0,1}$.
    \item \RHN{The experimental convergence rates of the state and
        optimal control variables as we uniformly refine the finite element mesh.  
    We first solve the problem on a very fine mesh, 7 uniform
    refinement cycles, and use it in place of a closed form
    solution. The latter is complicated and thus impractical
    for nonlinear optimization problems.}

\end{enumerate}

\subsection{\HA{Example 1}}\label{s:ex1}
Let $\gamma_d = \frac{1}{16 \pi}  \sin (2\pi x_1) - \frac{1}{16 \pi}
\sin(4 \pi x_1) + \frac{1}{32\pi} \sin(6 \pi x_1)$ be a smooth target
function, as depicted in \figref{fig:desired_sine}.
As $\theta_1$ in \eqref{eq:v_u_const_set} belongs to $(0,1)$, we
define the admissible control set $\Uad = \{u \in L^2(0,1) :
\normLt{u}{0,1} \le 0.9 \}$. Figure~\ref{fig:Ex1_control_rate}
displays the asymptotic rates for $\gamma$ in $W^1_\infty$ (left), $y$ 
in $W^1_p$ (middle) and $u$ in $L^2$ (right). The first two are
linear, and agree with Corollary \ref{cor:final_estimate}, whereas the
latter is quadratic and is better than predicted (see 
Corollary \ref{cor:final_estimate} and Remark \ref{rem:linear_rate}).

\begin{figure}[h!]
\centering
\includegraphics[width=0.32\textwidth]{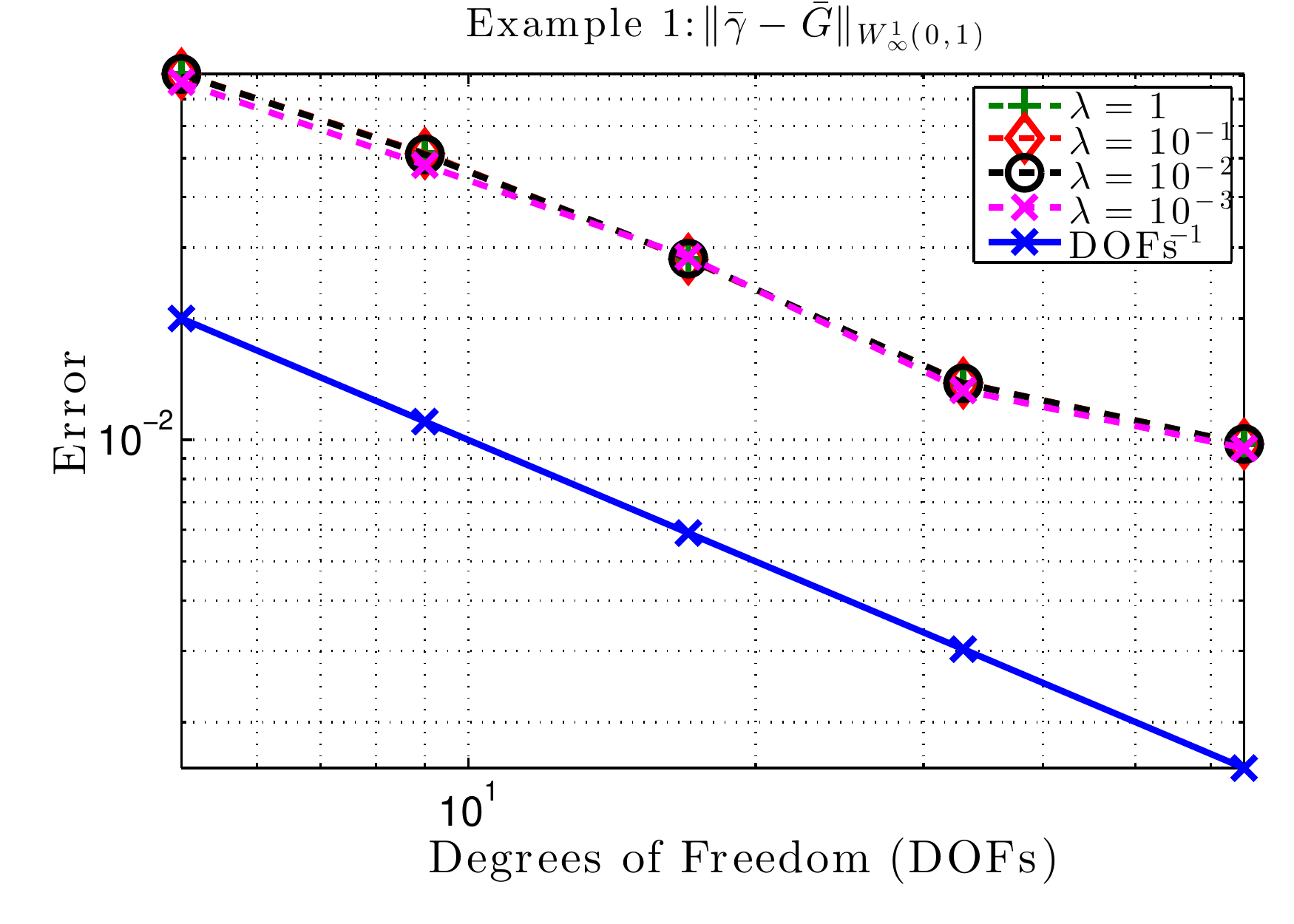}
\includegraphics[width=0.32\textwidth]{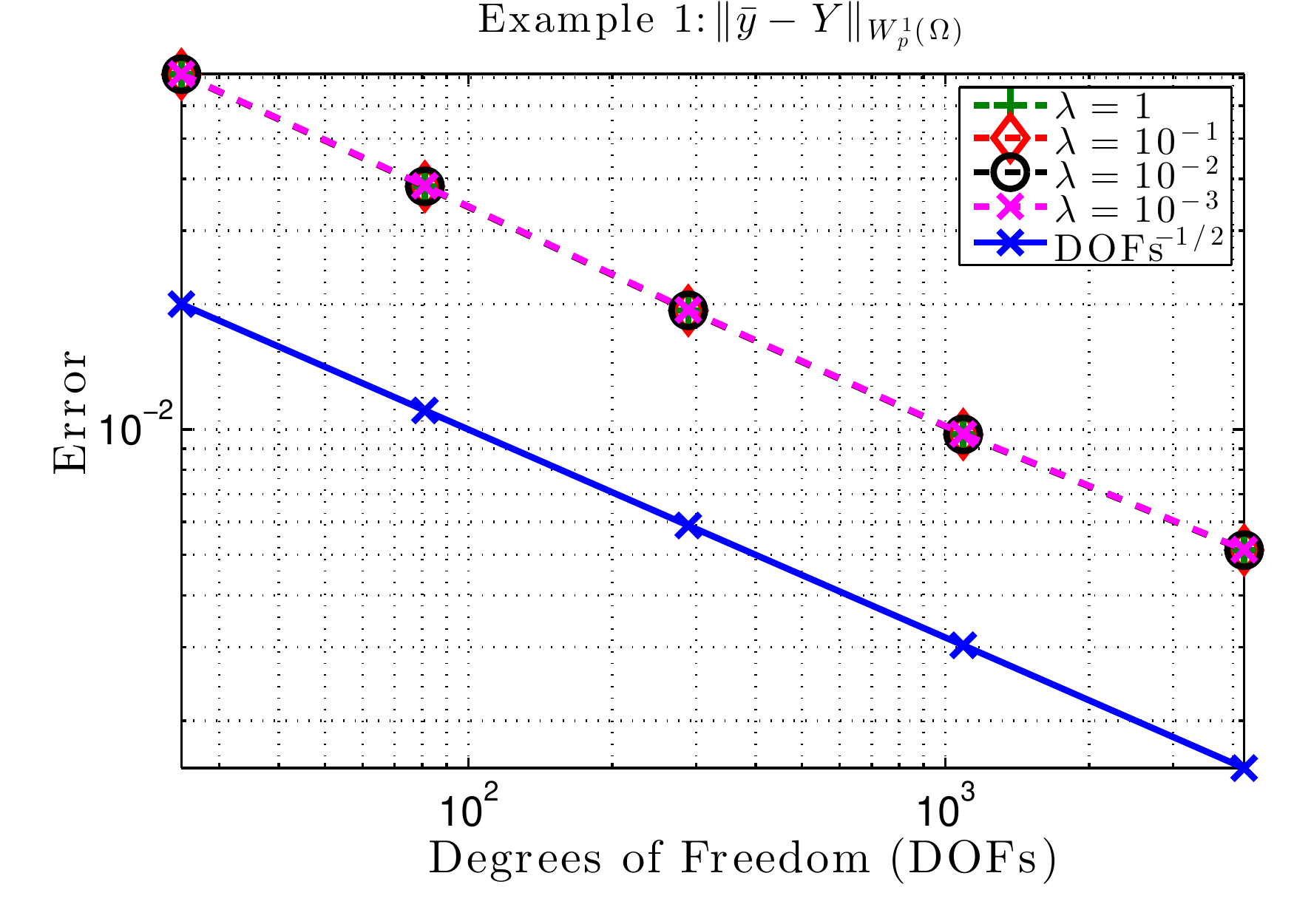}
\includegraphics[width=0.32\textwidth]{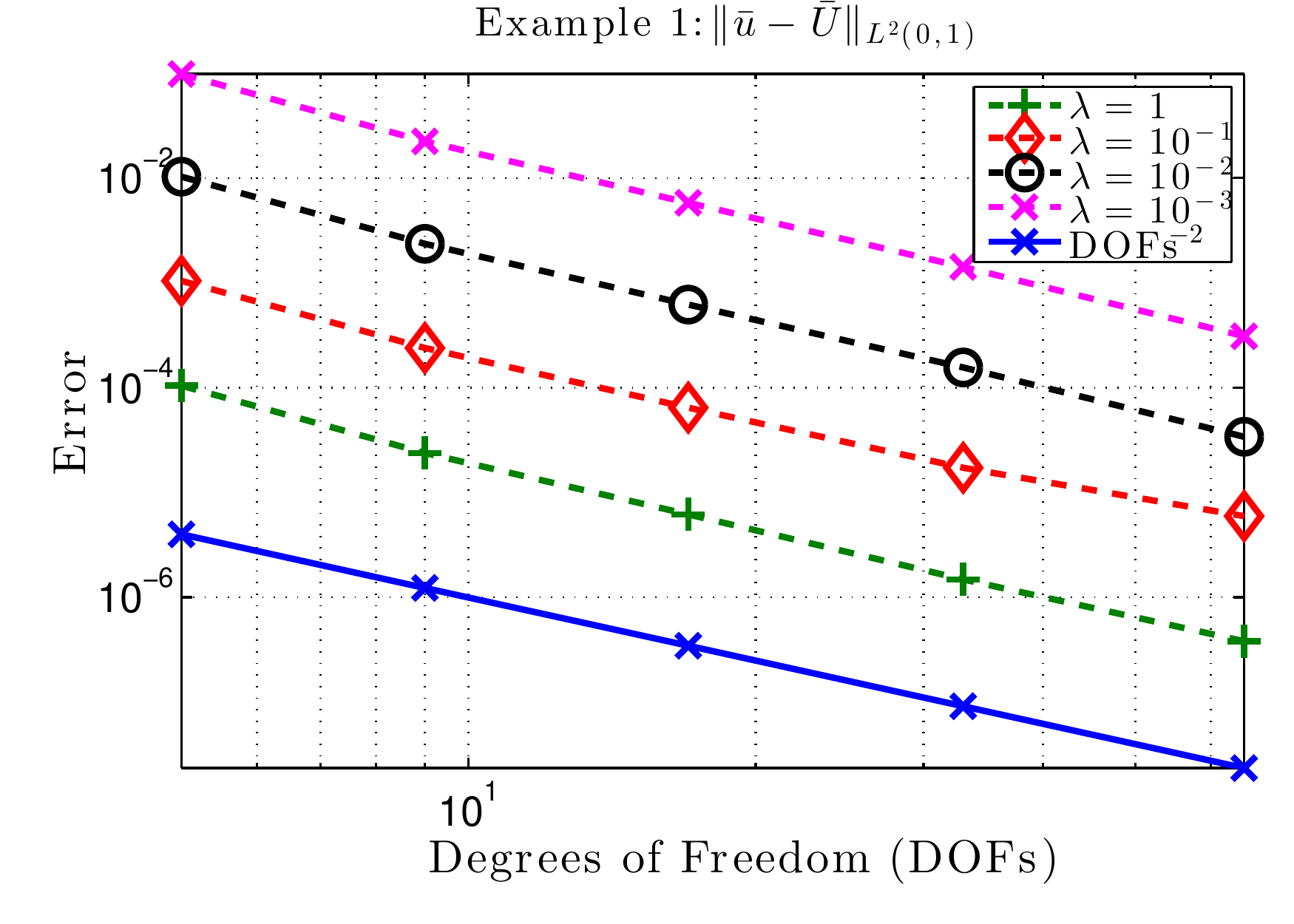}
\caption{\label{fig:Ex1_control_rate} Example 1: The left and the middle panel shows the error decay in $\sobZ1\infty{0,1}$ norm for $\bar\gamma$ and in $\sobZ1p\Omega$, $p = 2.1$ for $\bar{y}$ (dotted line) over several values of $\lambda$. The $DOFs$ are the total degree of freedom on $\Gamma$ and in $\Omega$. The right panel shows the error decay in $L^2$ norm for the optimal control.
The blue solid lines are for reference. In first two panels we observe linear rate, however for the last panel we observe quadratic rate.}
\end{figure}

\begin{figure}[h!]
\begin{minipage}[b]{0.64\linewidth}
\begin{tabular}{|cc|}
\hline
      $\del{\gop, \yop}$ 
    & $\uop$
\\ \hline     
      \multicolumn{2}{|c|}{$\lambda = 10^{-1}$, $\uop \in \intoo{-0.0027,0.0037}$}
\\ 
      \includegraphics[width=0.33\textwidth]{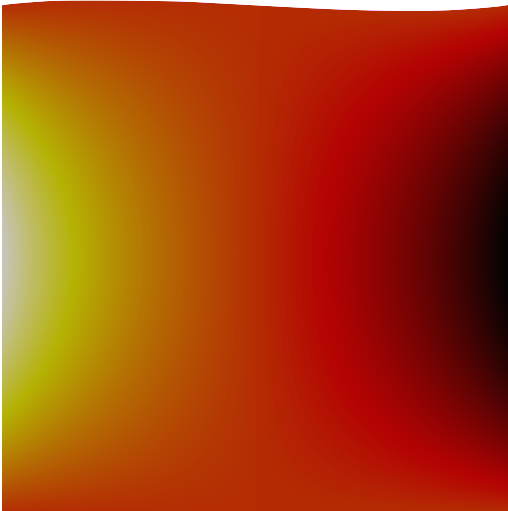}
    & \includegraphics[width=0.33\textwidth]{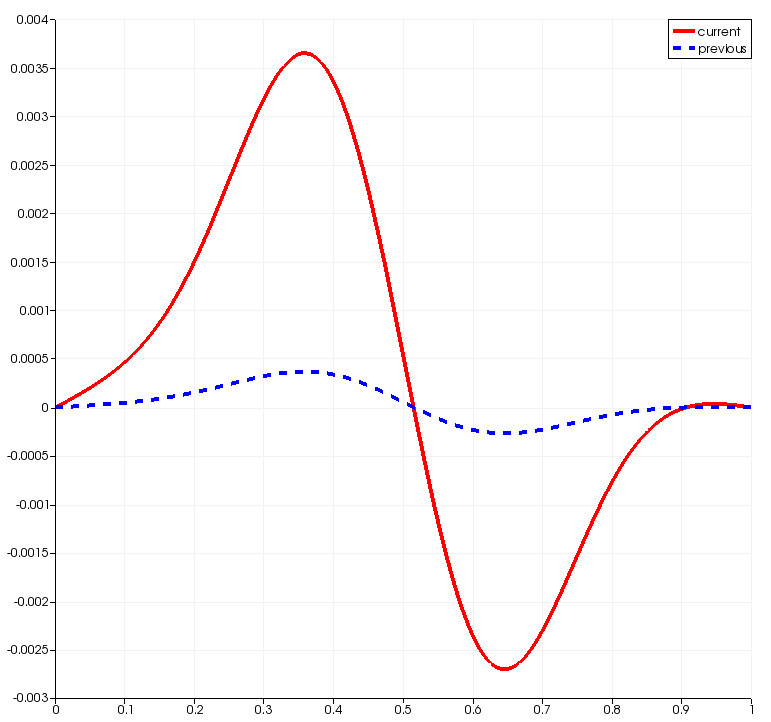}
\\ \hline
      \multicolumn{2}{|c|}{$\lambda = 10^{-2}$, $\uop \in \intoo{-0.0281,0.0335}$}
\\ 
      \includegraphics[width=0.33\textwidth]{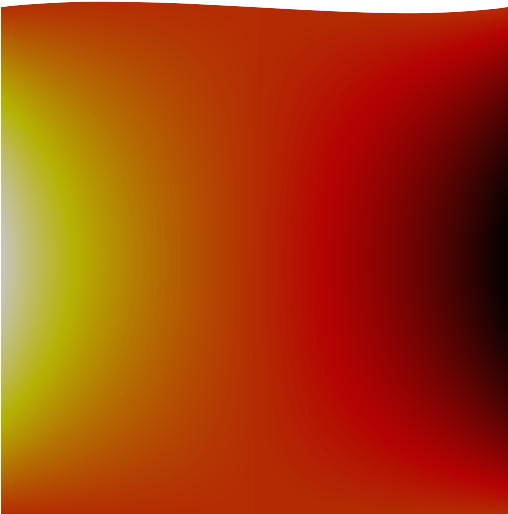}
    & \includegraphics[width=0.33\textwidth]{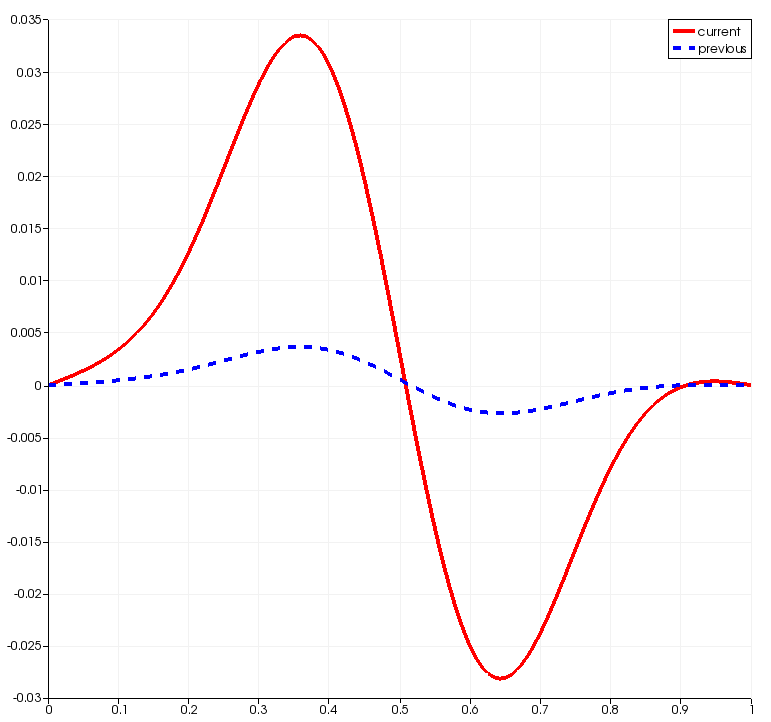}  
\\ \hline
      \multicolumn{2}{|c|}{$\lambda = 10^{-3}$, $\uop \in \intoo{-0.2424,0.2524}$}
\\ 
      \includegraphics[width=0.33\textwidth]{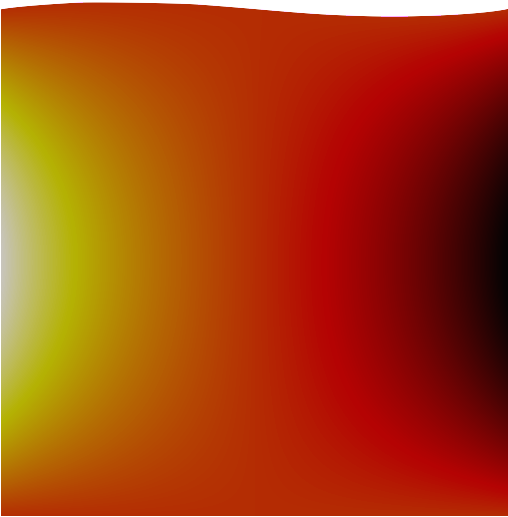}
    & \includegraphics[width=0.33\textwidth]{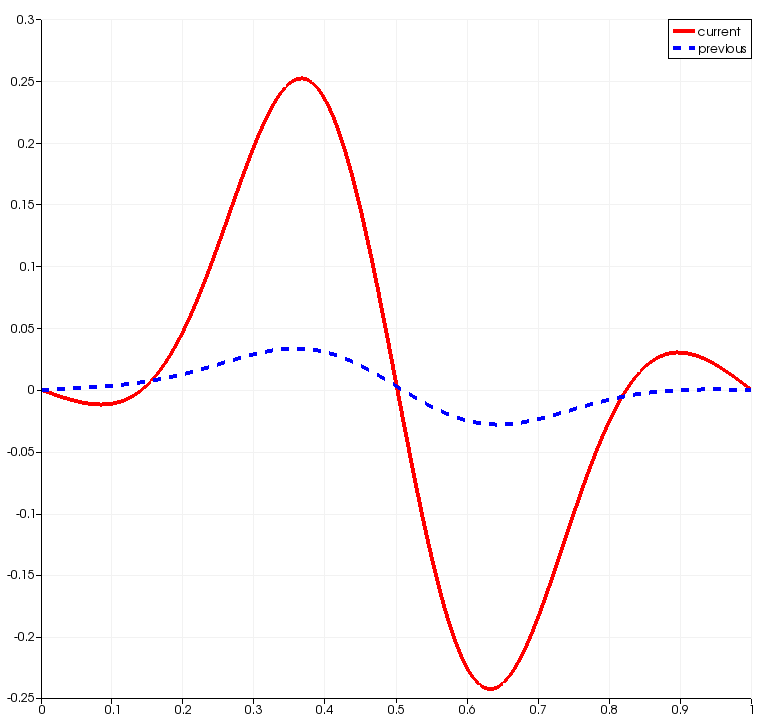}
\\ \hline
\end{tabular}
\end{minipage}
\hspace{-2.0cm}
\begin{minipage}[b]{0.64\linewidth}
\begin{tabular}{|cc|}
\hline
      $\del{\gop, \yop}$ 
    & $\uop$
\\ \hline
      \multicolumn{2}{|c|}{$\lambda = 10^{-4}$, $\uop \in \intoo{-1.3803,1.3663}$}    
\\ 
      \includegraphics[width=0.33\textwidth]{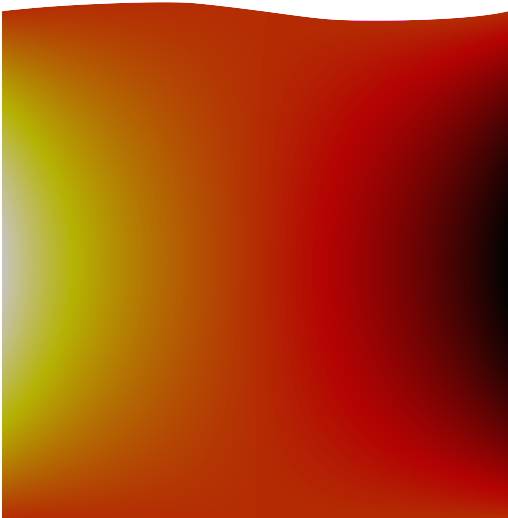}
    & \includegraphics[width=0.33\textwidth]{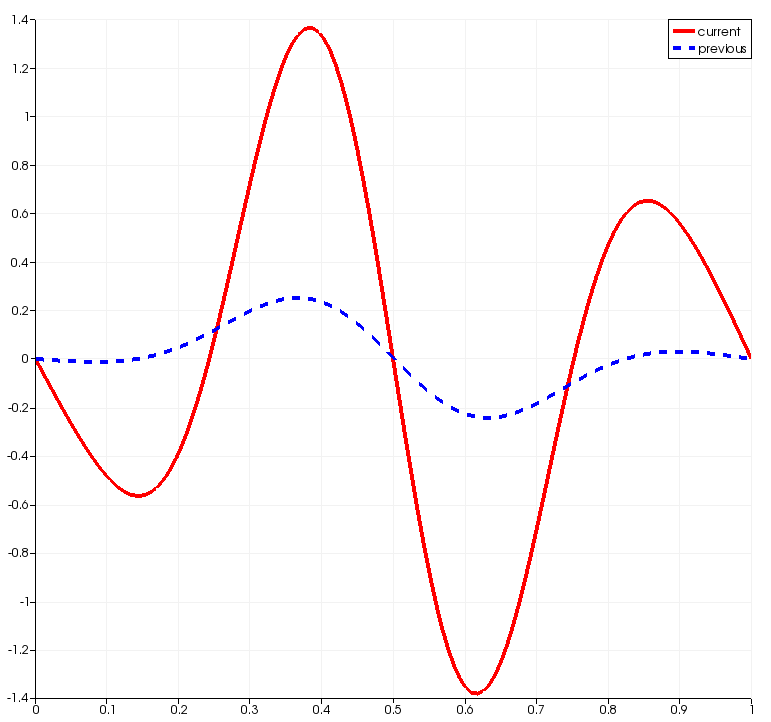}
\\ \hline
      \multicolumn{2}{|c|}{$\lambda = 10^{-5}$, $\uop \in \intoo{-1.6614,1.6472}$}    
\\ 
      \includegraphics[width=0.33\textwidth]{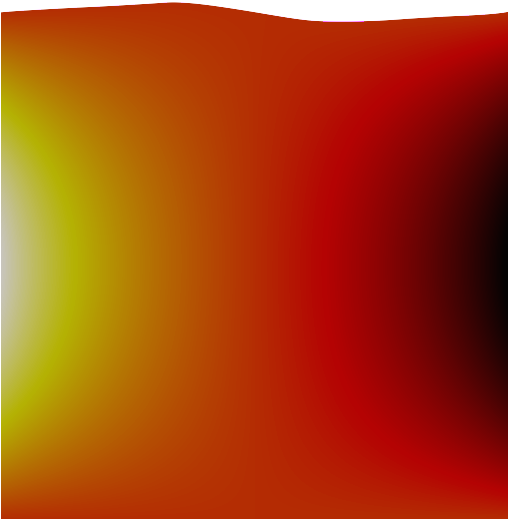}
    & \includegraphics[width=0.33\textwidth]{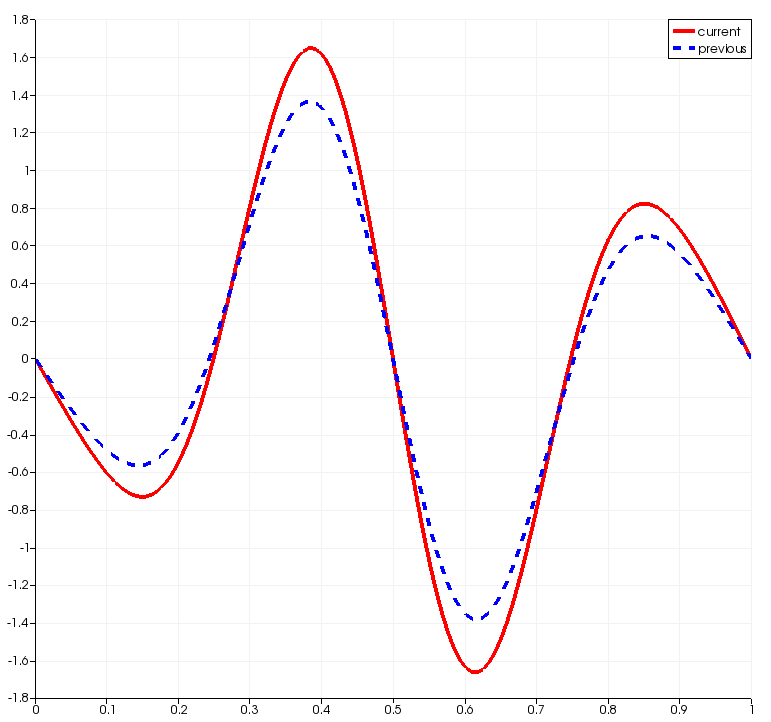}
\\ \hline
      \multicolumn{2}{|c|}{$\lambda = 10^{-6}$, $\uop \in \intoo{-1.6614,1.6472}$}        
\\ 
      \includegraphics[width=0.33\textwidth]{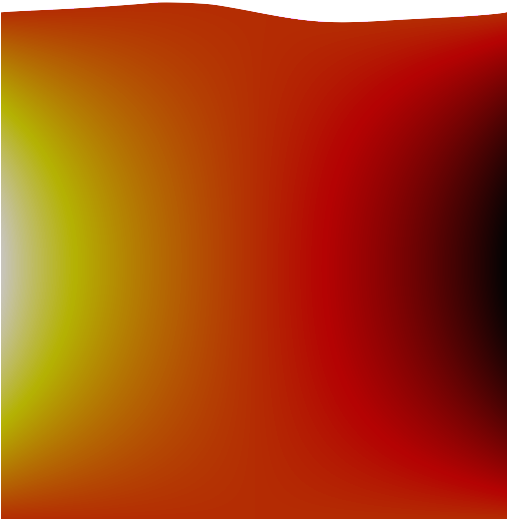}
    & \includegraphics[width=0.33\textwidth]{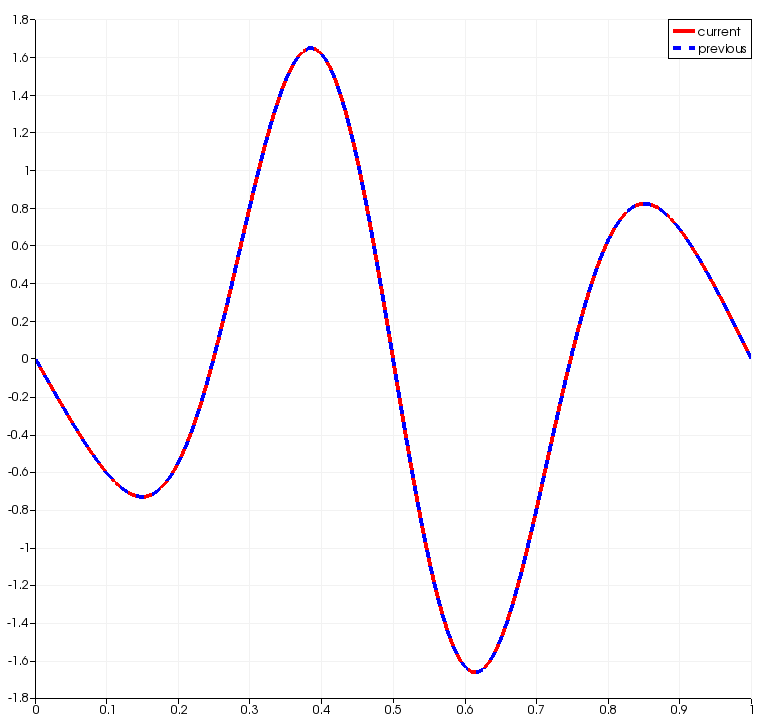}
\\ \hline
\end{tabular}
\end{minipage}
\caption{\label{fig:ex1_state_control}Example 1 ($\gamma_d$ smooth, $\uop$ constrained): The optimal state $\of{\gop, \yop}$, applied control $\uop$ in solid red, and previous control in dashed blue for comparison. The figures show the corresponding values of $\lambda$, from $\lambda =10^{-1}$ to $\lambda = 10^{-6}$, as well as the smallest and largest value of control. Notice that there is no visual difference between the optimal control for $\lambda = 10^{-5}, 10^{-6}$. This is because the control constraints are active.}
\end{figure}

Table~\ref{tab:Ex1} (left) provides the other relevant metrics. 
Finally, \figref{fig:ex1_state_control} shows the optimal state
$\del{\gop,\yop}$ and two consecutive optimal controls (solid:
current, dotted: previous) for $\lambda = 1e0$ to $\lambda=10$e-6. 
For $\lambda = 1$e-5 and 1e-6 the current and previous controls lie on 
top of each other because the constraints are active.

\begin{table}[h!]
\begin{minipage}[b]{0.50\linewidth}
\begin{center}
\begin{tabular}{|c|c|c|}\hline
$\lambda$ & $\calJ(\uop)$         & $\normLt{\uop}{0,1}$  \\ \hline 
$10^0$    & $1.59 \times 10^{-4}$ & $1.90 \times 10^{-4}$ \\
$10^{-1}$ & $1.58 \times 10^{-4}$ & $1.89 \times 10^{-3}$ \\
$10^{-2}$ & $1.57 \times 10^{-4}$ & $1.79 \times 10^{-2}$ \\
$10^{-3}$ & $1.46 \times 10^{-4}$ & $1.35 \times 10^{-1}$ \\
$10^{-4}$ & $1.07 \times 10^{-4}$ & $7.44 \times 10^{-1}$ \\
$10^{-5}$ & $7.19 \times 10^{-5}$ & $9.00 \times 10^{-1}$ \\
$10^{-6}$ & $6.83 \times 10^{-5}$ & $9.00 \times 10^{-1}$ \\ \hline
\end{tabular}
\end{center}
\end{minipage}
\begin{minipage}[b]{0.50\linewidth}
\begin{center}
\begin{tabular}{|c|c|c|}\hline
$\lambda$ & $\calJ(\uop)$	 & $\normLt{\uop}{0,1}$   \\ \hline
$10^0$    & $1.59 \times 10^{-4}$ & $1.90 \times 10^{-4}$      \\
$10^{-1}$ & $1.58 \times 10^{-4}$ & $1.89 \times 10^{-3}$ \\
$10^{-2}$ & $1.57 \times 10^{-4}$ & $1.79 \times 10^{-2}$ \\
$10^{-3}$ & $1.46 \times 10^{-4}$ & $1.35 \times 10^{-1}$ \\
$10^{-4}$ & $1.07 \times 10^{-4}$ & $7.44 \times 10^{-1}$ \\
$10^{-5}$ & $3.60 \times 10^{-5}$ & $2.21 $  \\
$10^{-6}$ & $5.37 \times 10^{-6}$ & $3.17 $  \\ \hline
\end{tabular}
\end{center}	
\end{minipage}
\caption{\label{tab:Ex1}The values of the cost function 
  $\calJ(\uop)$, and the $L^2$-norm of $\uop$ for Examples 1 (left) and Example 2 (right) as $\lambda$ varies from $1$ to $10^{-6}$.}
\vskip-0.7cm
\end{table}

\subsection{\HA{Example 2}}\label{s:ex2}
Let $\gamma_d$ be as in \secref{s:ex1}. We assume that $\Uad = L^2(0,1)$, i.e. the control is unconstrained. \figref{fig:Ex2_control_rate} shows the rate of convergence with respect to the degrees of freedom. In Table~\ref{tab:Ex1} (right) we collect the other relevant metrics.
\begin{figure}[h!]
\centering
\includegraphics[width=0.32\textwidth]{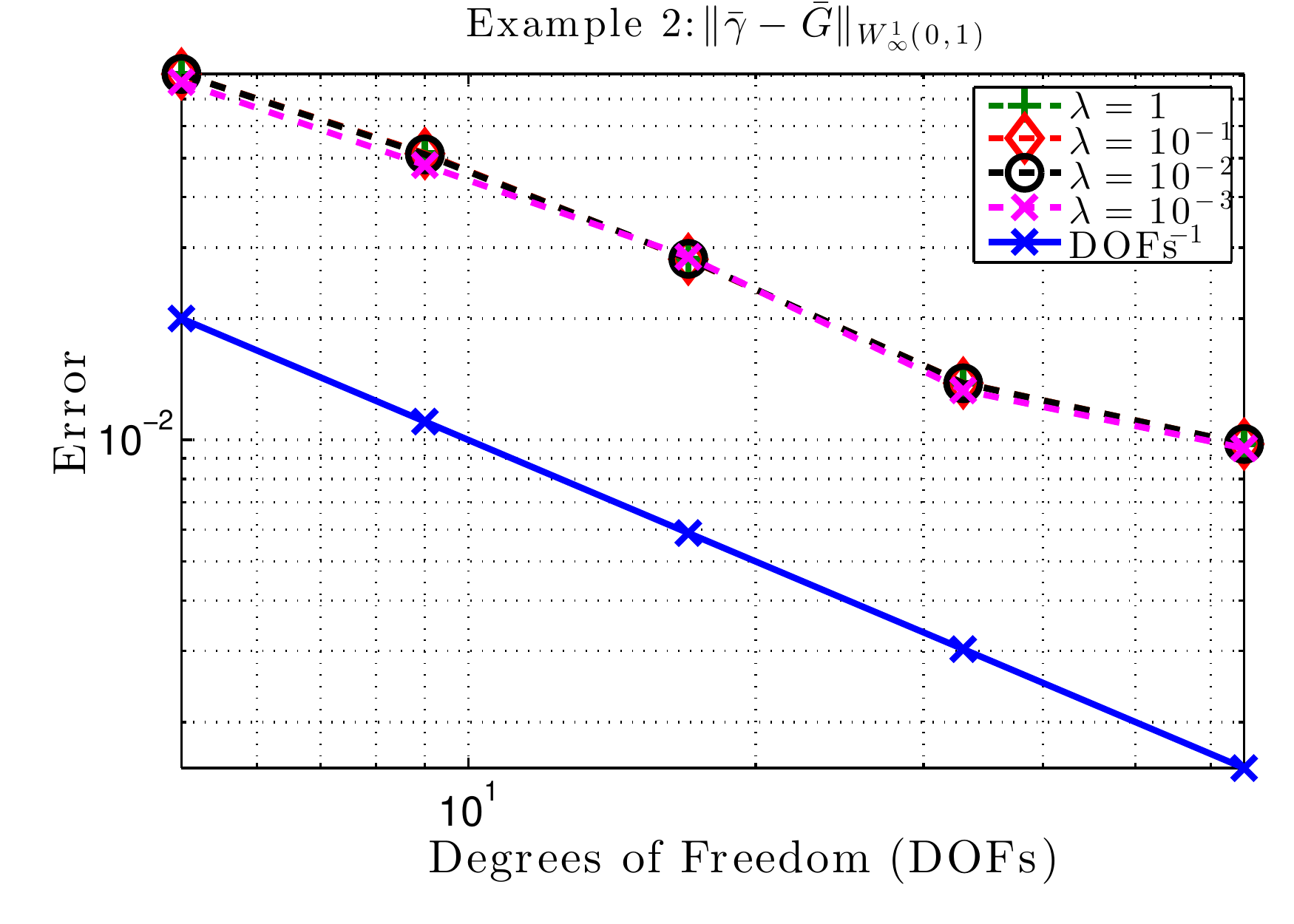}
\includegraphics[width=0.32\textwidth]{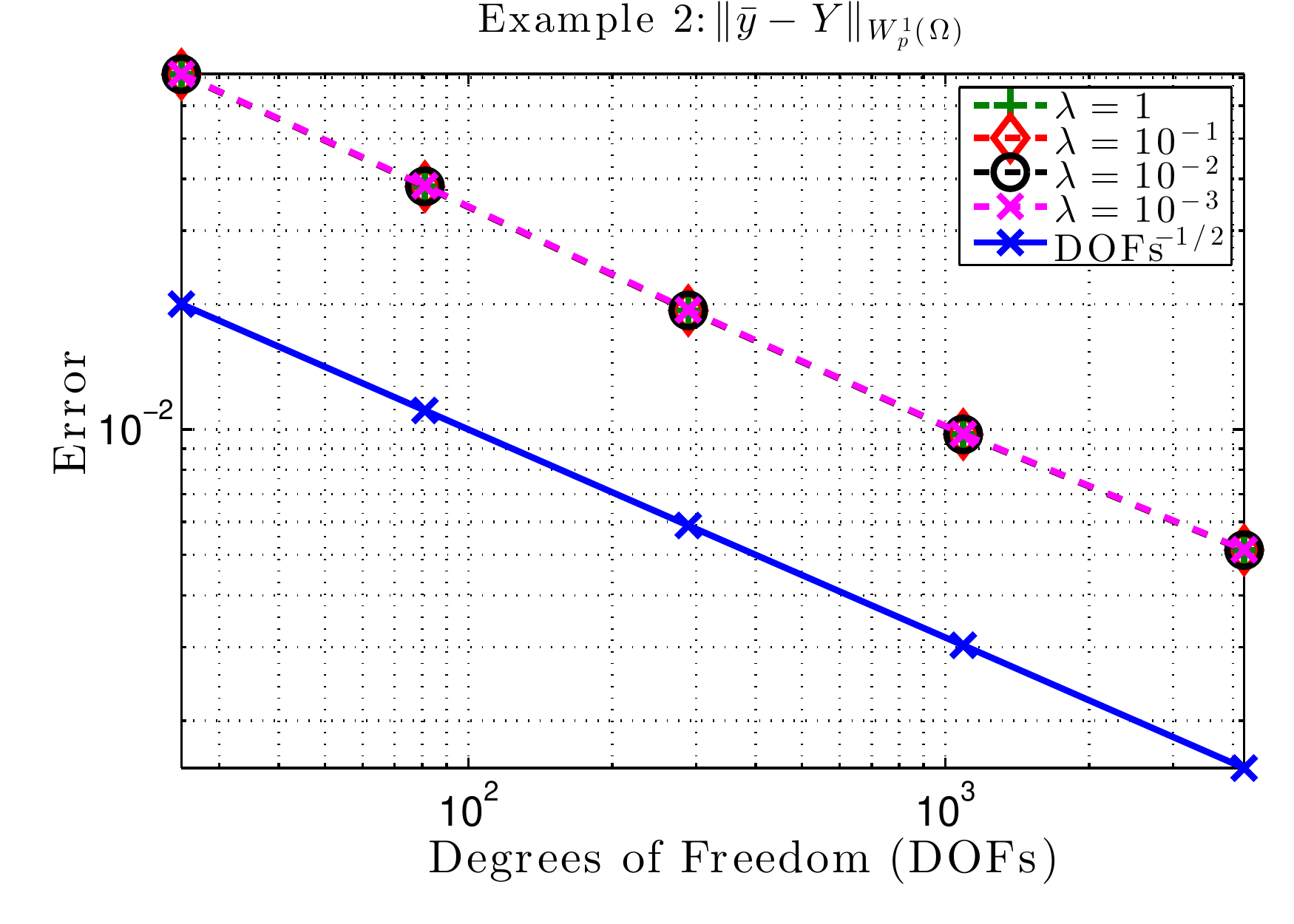}
\includegraphics[width=0.32\textwidth]{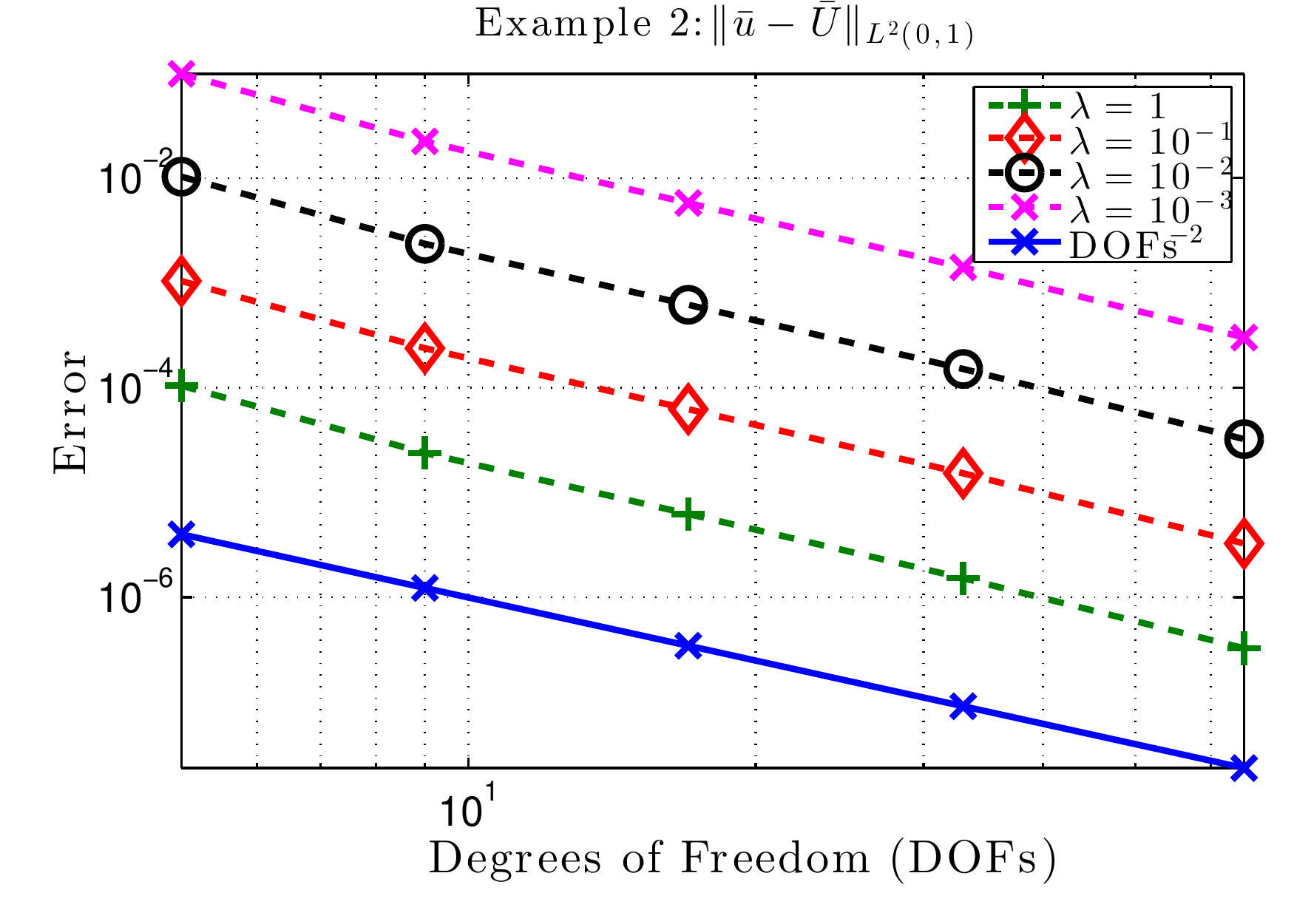}
\caption{\label{fig:Ex2_control_rate} Example 2: The left and the middle panel shows the error decay in $\sobZ1\infty{0,1}$ norm for $\bar\gamma$ and in $\sobZ1p\Omega$, $p = 2.1$ for $\bar{y}$ (dotted line) over several values of $\lambda$. The $DOFs$ are the total degree of freedom on $\Gamma$ and in $\Omega$. The right panel shows the error decay in $L^2$ norm for the optimal control.
The blue solid lines are for reference. In first two panels we observe linear rate, however for the last panel we observe quadratic rate.}
\end{figure}
%
				
The first column in \figref{fig:ex2_state_control} shows the optimal state 
$\del{\gop,\yop}$ as $\lambda$ approaches zero. The second column shows the control
$u$ applied (solid red); for reference we also plot the previous control (dotted
blue). Since the control is unconstrained and $\gamma_d$ is smooth, $\gop$ matches
  $\gamma_d$ almost perfectly. 

\begin{figure}[h!]
\begin{minipage}[b]{0.64\linewidth}
\begin{tabular}{|cc|}
\hline
      $\del{\gop, \yop}$ 
    & $\uop$
\\ \hline     
      \multicolumn{2}{|c|}{$\lambda = 10^{-1}$, $\uop \in \intoo{-0.0027,0.0036}$}
\\ 
      \includegraphics[width=0.33\textwidth]{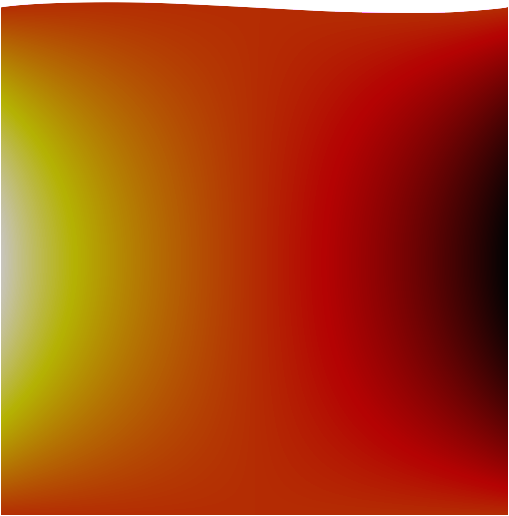}
    & \includegraphics[width=0.33\textwidth]{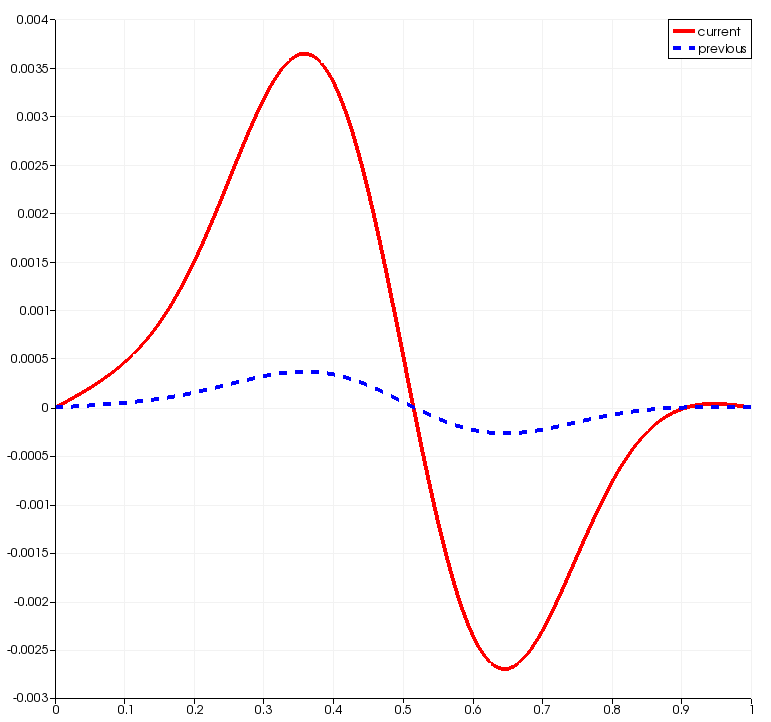}
\\ \hline
      \multicolumn{2}{|c|}{$\lambda = 10^{-2}$, $\uop \in \intoo{-0.0281,0.0335}$}
\\ 
      \includegraphics[width=0.33\textwidth]{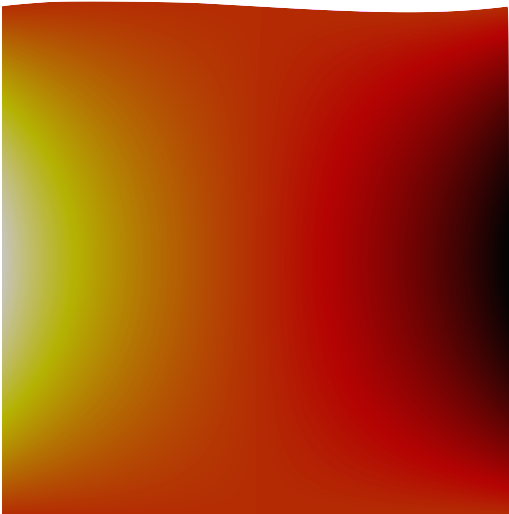}
    & \includegraphics[width=0.33\textwidth]{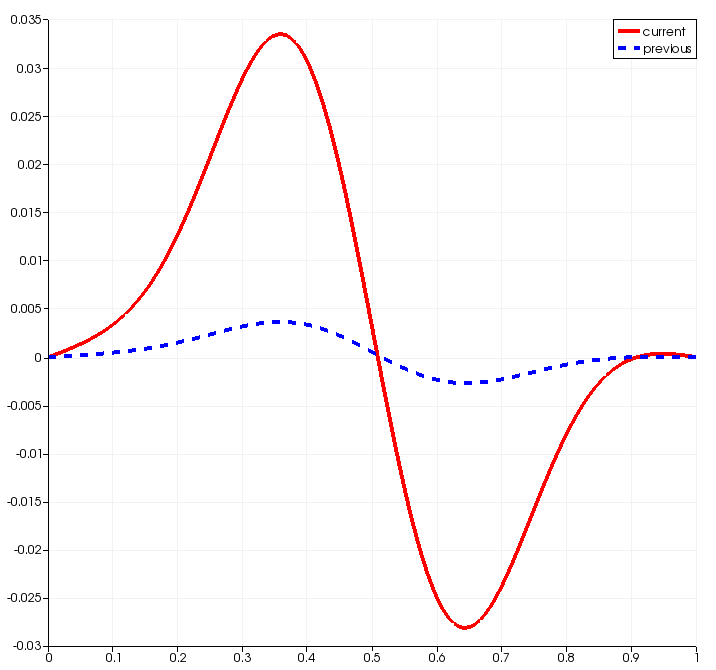}  
\\ \hline
      \multicolumn{2}{|c|}{$\lambda = 10^{-3}$, $\uop \in \intoo{-0.2424,0.2524}$}
\\ 
      \includegraphics[width=0.33\textwidth]{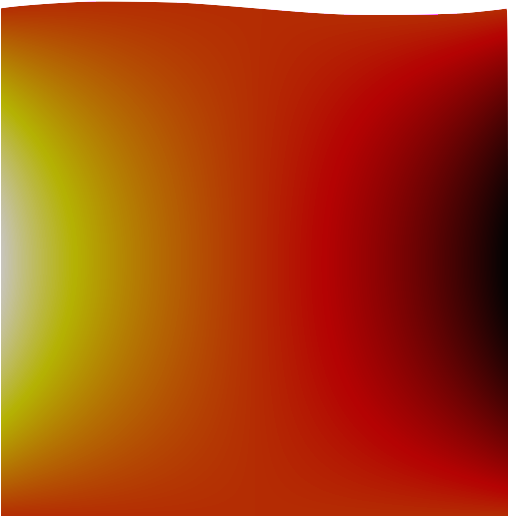}
    & \includegraphics[width=0.33\textwidth]{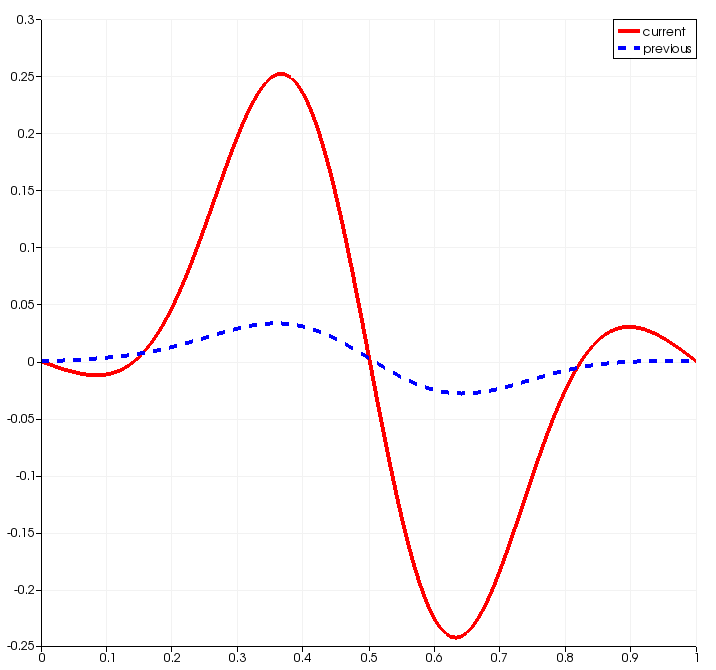}
\\ \hline
\end{tabular}
\end{minipage}
\hspace{-2.0cm}
\begin{minipage}[b]{0.64\linewidth}
\begin{tabular}{|cc|}
\hline
      $\del{\gop, \yop}$ 
    & $\uop$
\\ \hline
      \multicolumn{2}{|c|}{$\lambda = 10^{-4}$, $\uop \in \intoo{-1.3797,1.3677}$}    
\\ 
      \includegraphics[width=0.33\textwidth]{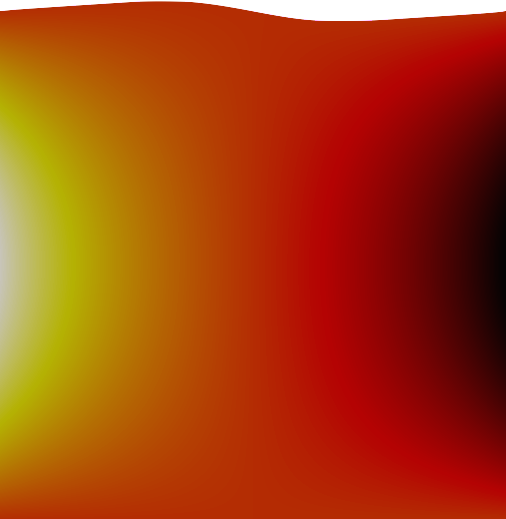}
    & \includegraphics[width=0.33\textwidth]{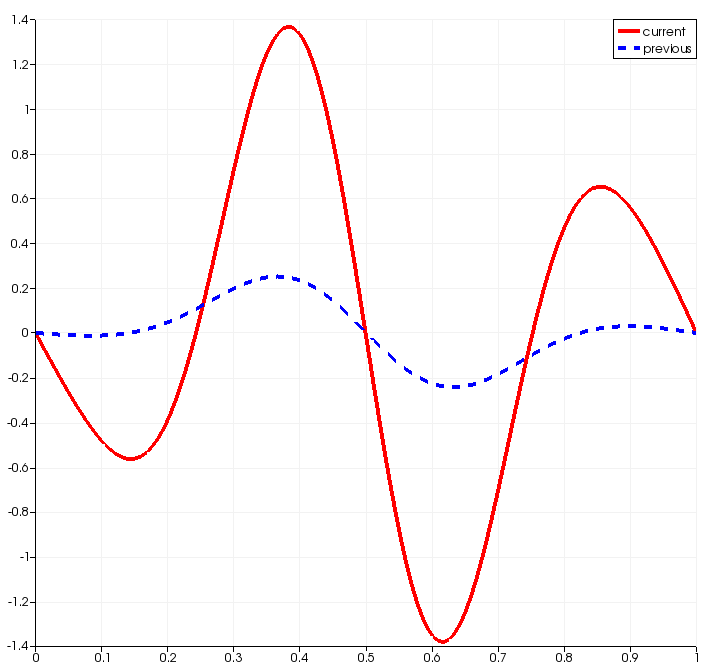}
\\ \hline
      \multicolumn{2}{|c|}{$\lambda = 10^{-5}$, $\uop \in \intoo{-4.2298,4.2350}$}    
\\ 
      \includegraphics[width=0.33\textwidth]{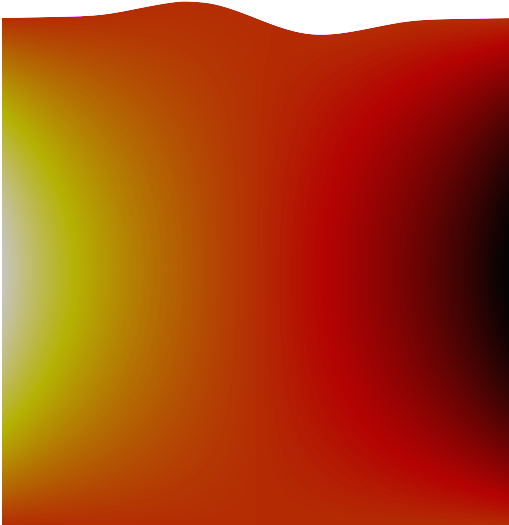}
    & \includegraphics[width=0.33\textwidth]{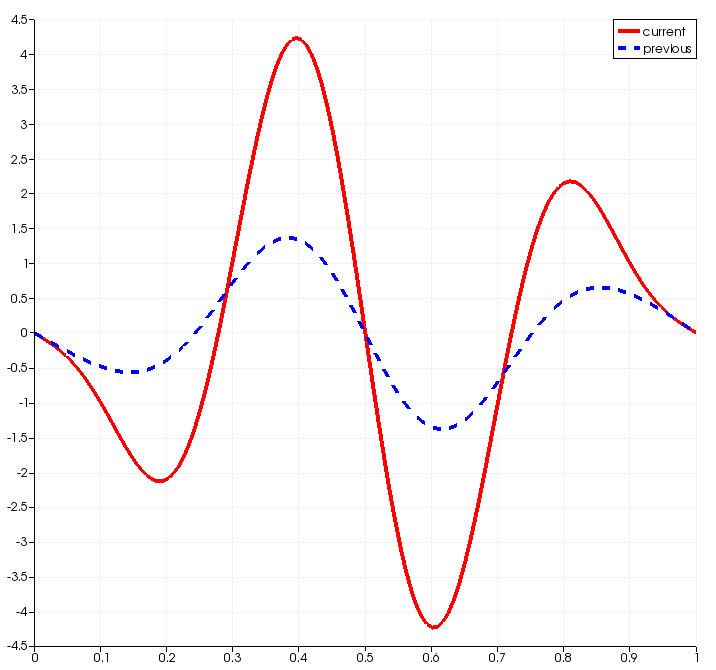}
\\ \hline
      \multicolumn{2}{|c|}{$\lambda = 10^{-6}$, $\uop \in \intoo{-6.2378,6.2829}$}        
\\ 
      \includegraphics[width=0.33\textwidth]{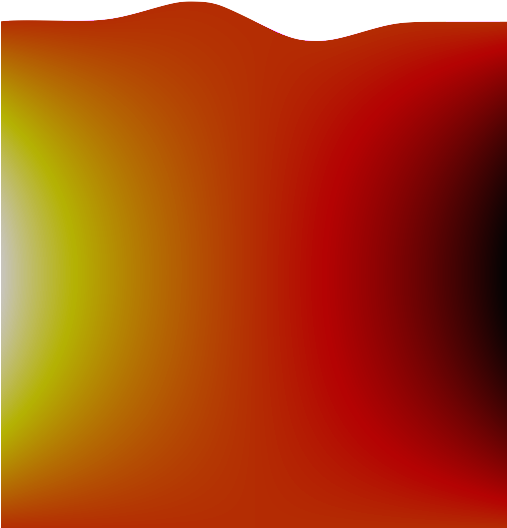}
    & \includegraphics[width=0.33\textwidth]{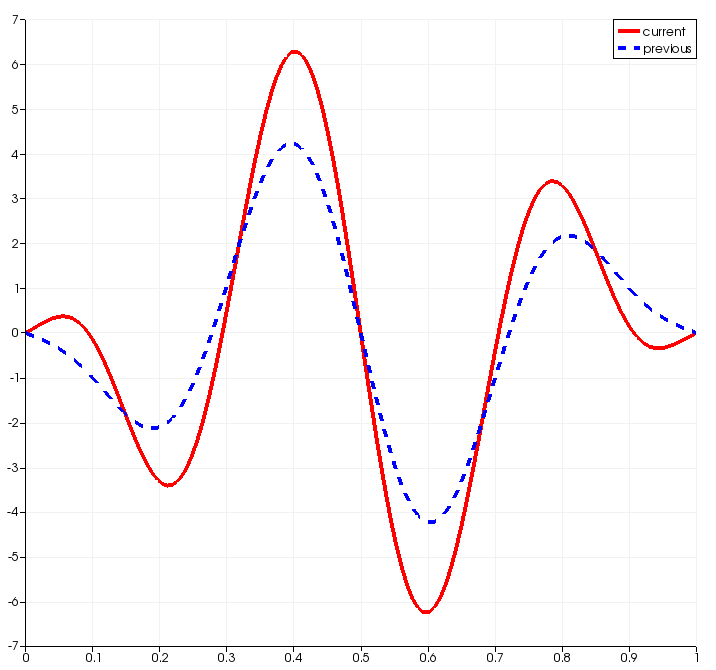}
\\ \hline
\end{tabular}
\end{minipage}
\caption{\label{fig:ex2_state_control}Example 2 ($\gamma_d$
    smooth, $\uop$ unconstrained): The optimal state $\of{\gop,
      \yop}$, applied control $\uop$ in solid red, and previous
    control in dashed blue for comparison. The figures show the
  corresponding value of $\lambda$, from $\lambda =10^{-1}$ to
  $\lambda = 10^{-6}$, as well as the smallest and largest value of
  control $u$. As $u$ is unconstrained and $\gamma_d$
    smooth, the matching of $\bar\gamma$ and $\gamma_d$ is almost
  perfect for small values of $\lambda$.}
\end{figure}

\subsection{\HA{Example 3}}\label{s:ex3}
Let the desired configuration $\gamma_d$ be an inverted hat function (see \figref{fig:desired_sine}). As $\gamma$ satisfies the state equations \eqref{eq:state_eqn}, and the second-order regularity $\gamma \in \sob2\infty{0,1}$ (see \thmref{thm:second_order_state}), the profile $\gamma_d$ is not achievable. 
In Table~\ref{tab:Ex3} we collect the several relevant metrics.
\begin{table}[h!]
\centering
\begin{tabular}{|c|c|c|}\hline
$\lambda$ & $\calJ(\uop)$         & $\normLt{\uop}{0,1}$  \\ \hline 
$10^0$    & $4.11 \times 10^{-2}$ & $2.88 \times 10^{-2}$ \\
$10^{-1}$ & $3.77 \times 10^{-2}$ & $2.64 \times 10^{-1}$ \\
$10^{-2}$ & $2.08 \times 10^{-2}$ & $1.43 $  \\
$10^{-3}$ & $4.36 \times 10^{-3}$ & $2.60 $  \\
$10^{-4}$ & $7.52 \times 10^{-4}$ & $3.37 $  \\
$10^{-5}$ & $1.31 \times 10^{-4}$ & $4.46 $  \\
$10^{-6}$ & $2.30 \times 10^{-5}$ & $5.89 $  \\ \hline
\end{tabular}
\caption{\label{tab:Ex3}Example 3: The values of the cost function
  $\calJ(\uop)$, the $L^2$-norm of $\uop$, and the smallest eigenvalue of $\calJ''(\uop)$, as $\lambda$ varies from $1$ to $10^{-6}$.}
\vskip-0.8cm
\end{table}

The first column in \figref{fig:ex1_state_control} shows the optimal state 
$\del{\gop,\yop}$ as $\lambda$ approaches zero. The second column shows the control
applied (solid red); for reference we also plot the previous control (dotted
blue). For $\lambda=10^{-1}$ to $\lambda=10^{-3}$ one can see that the control acts at the center 
and tries to move $\gamma$ towards $\gamma_d$. For $\lambda=10^{-4}$ the control 
needs to push $\gamma$ in the right-half up, and in the left-half down and therefore it
adjusts accordingly. For $\lambda=10^{-6}$ the control again mostly
acts at the center. Moreover $\gamma$ matches $\gamma_d$ well but not
exactly because $\bar\gamma\in W^2_\infty(0,1)$.

\begin{figure}[h!]
\begin{minipage}[b]{0.64\linewidth}
\begin{tabular}{|c|c|}
\hline
      $\del{\gop, \yop}$ 
    & $\uop$
\\ \hline    
      \multicolumn{2}{|c|}{$\lambda = 10^{-1}$, $\uop \in \intoo{-0.3784,-0.0087}$}
\\ 
      \includegraphics[width=0.33\textwidth]{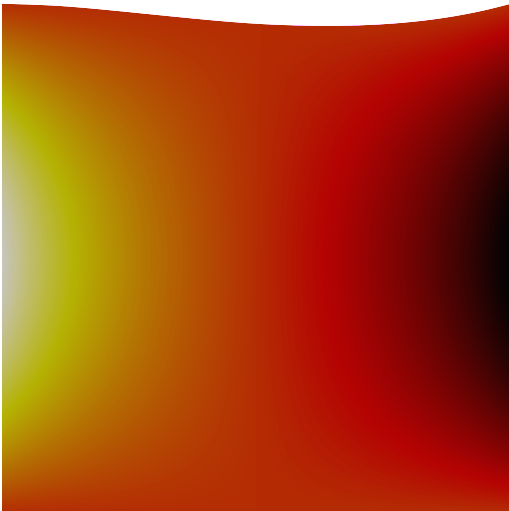}
    & \includegraphics[width=0.33\textwidth]{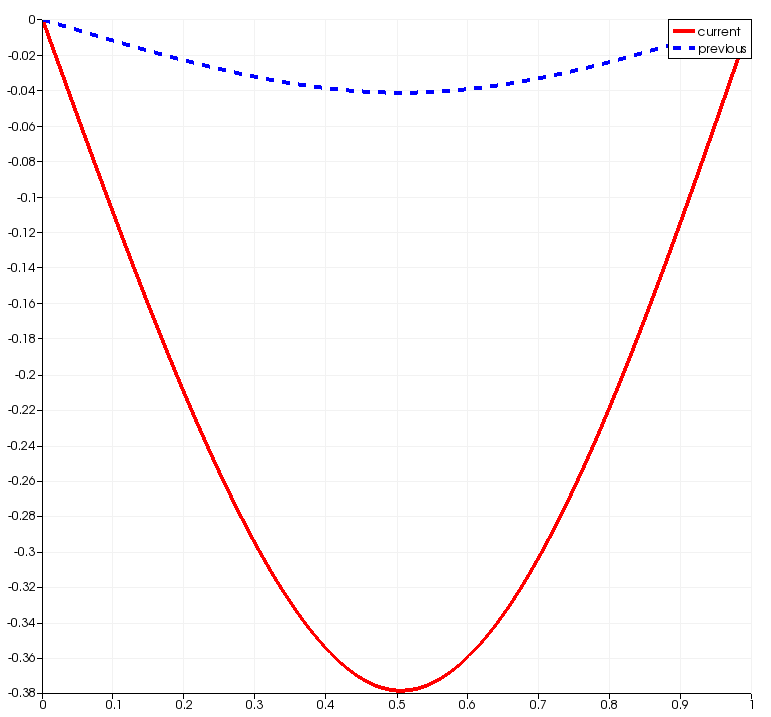}
\\ \hline
      \multicolumn{2}{|c|}{$\lambda = 10^{-2}$, $\uop \in \intoo{-2.0737,-0.0454}$}
\\ 
      \includegraphics[width=0.33\textwidth]{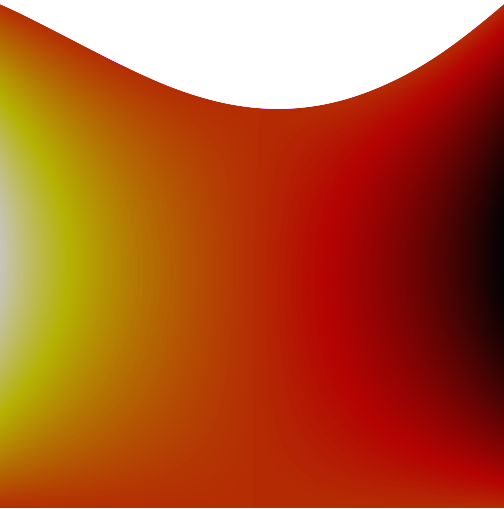}
    & \includegraphics[width=0.33\textwidth]{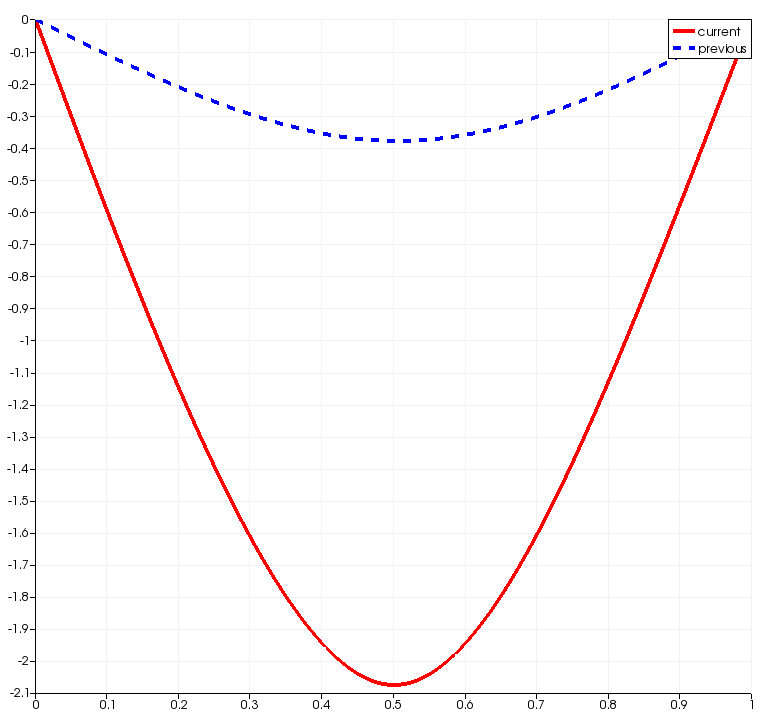}  
\\ \hline
      \multicolumn{2}{|c|}{$\lambda = 10^{-3}$, $\uop \in \intoo{-4.2068,-0.0297}$}
\\ 
      \includegraphics[width=0.33\textwidth]{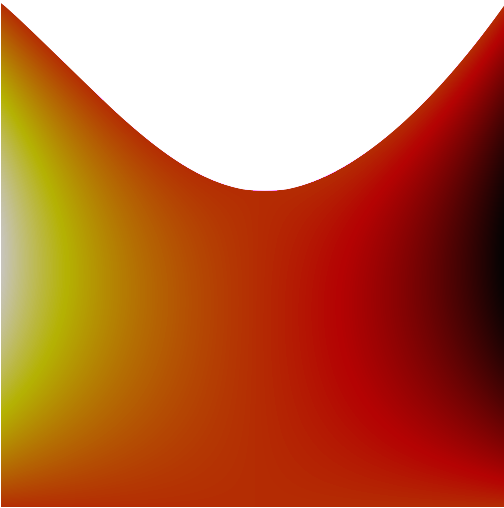}
    & \includegraphics[width=0.33\textwidth]{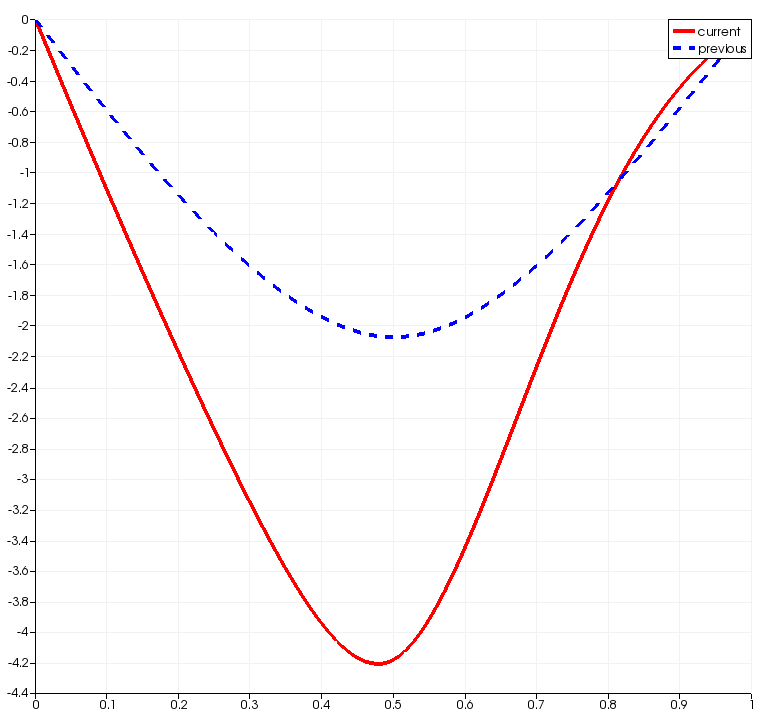}
\\ \hline
\end{tabular}
\end{minipage}
\hspace{-2.0cm}
\begin{minipage}[b]{0.64\linewidth}
\begin{tabular}{|c|c|}
\hline
      $\del{\gop, \yop}$ 
    & $\uop$
\\ \hline
      \multicolumn{2}{|c|}{$\lambda = 10^{-4}$, $\uop \in \intoo{-7.1013,1.2484}$}    
\\ 
      \includegraphics[width=0.33\textwidth]{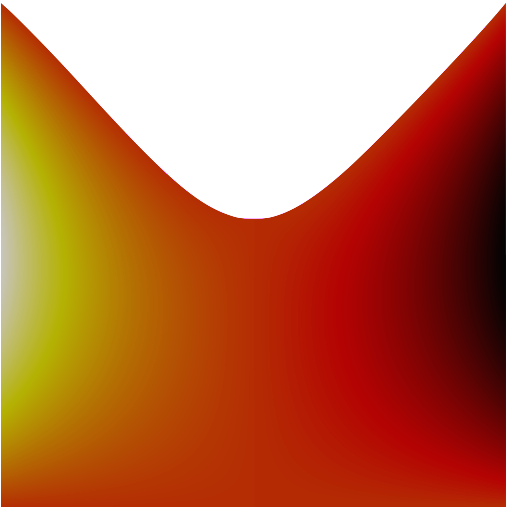}
    & \includegraphics[width=0.33\textwidth]{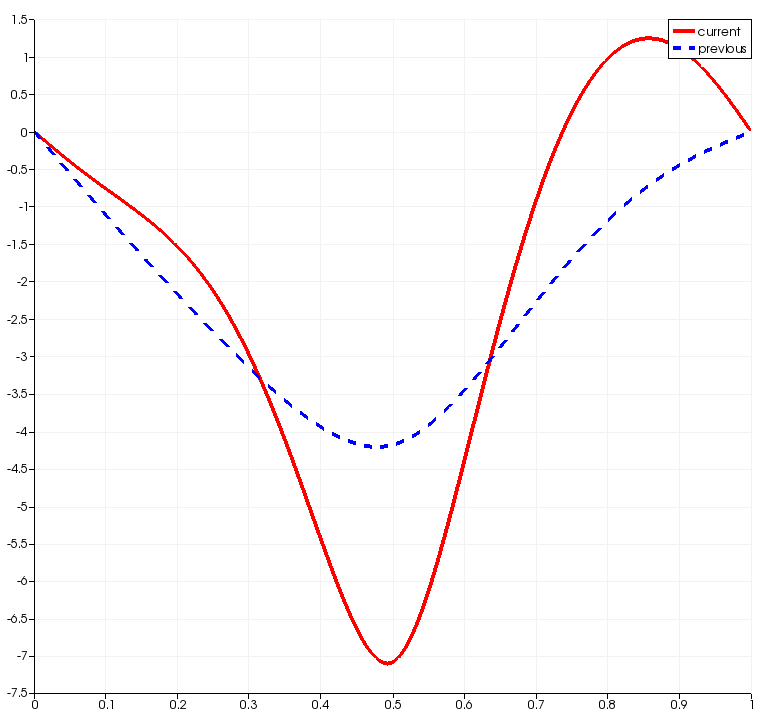}
\\ \hline
      \multicolumn{2}{|c|}{$\lambda = 10^{-5}$, $\uop \in \intoo{-12.5864,1.7676}$}    
\\ 
      \includegraphics[width=0.33\textwidth]{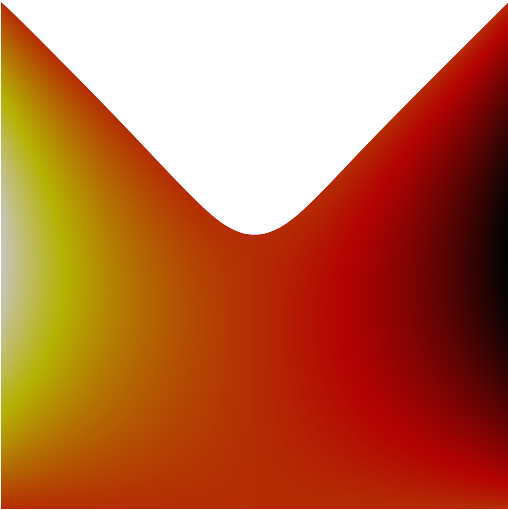}
    & \includegraphics[width=0.33\textwidth]{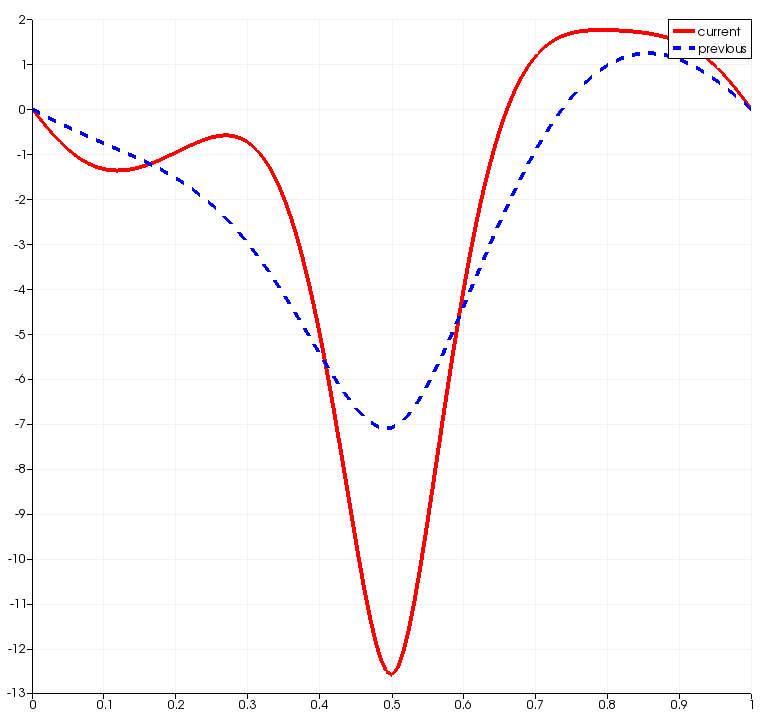}
\\ \hline
      \multicolumn{2}{|c|}{$\lambda = 10^{-6}$, $\uop \in \intoo{-22.4179,1.7714}$}        
\\ 
      \includegraphics[width=0.33\textwidth]{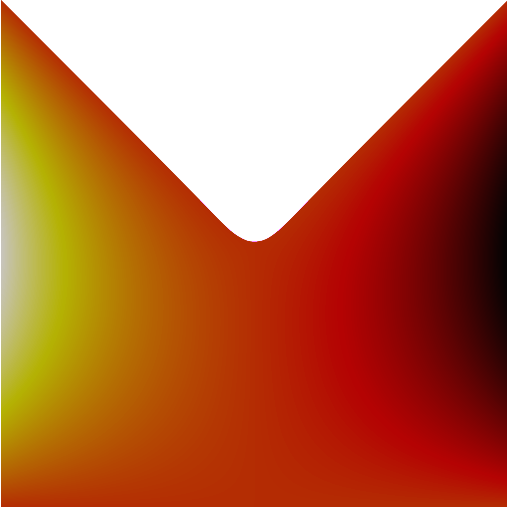}
    & \includegraphics[width=0.33\textwidth]{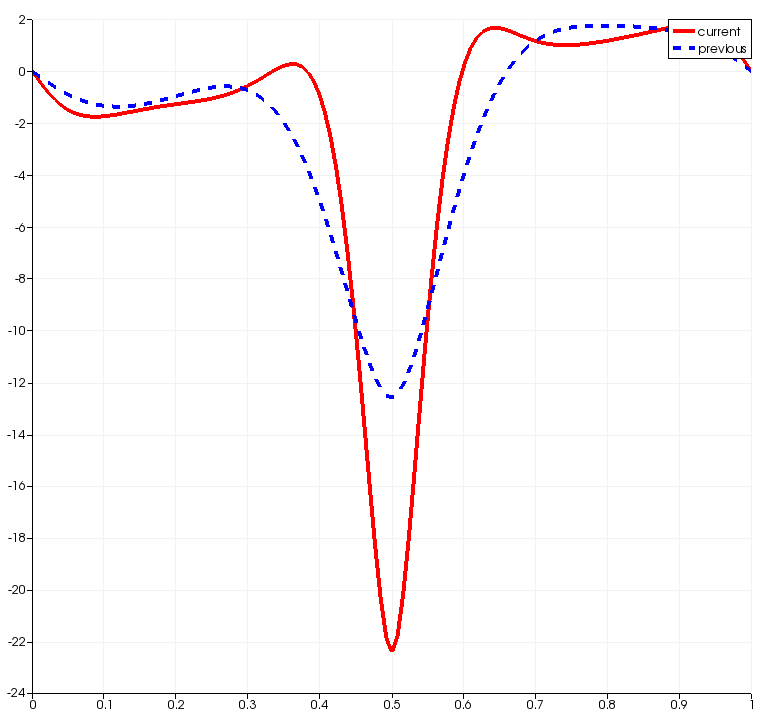}
\\ \hline
\end{tabular}
\end{minipage}
\caption{\label{fig:ex3_state_control} Example 3 ($\gamma_d$ rough,
  $\uop$ unconstrained): The optimal state $\of{\gop, \yop}$,
  applied control $\uop$ in solid red, and previous control in
  dashed blue for comparison. The figures show the corresponding value
  of $\lambda$, from $\lambda =10^{-1}$ to $\lambda = 10^{-6}$, as
  well as the smallest and largest value of control. Since $\gamma_d$
  is rough, we cannot expect a perfect match of $\bar\gamma$ and
  $\gamma_d$ even for unconstrained control.}
\end{figure}

\section*{Acknowledgments}
We thank Abner Salgado for his constant support on deal.II. \HA{We are 
grateful to the referees for their insightful comments and suggestions.}

{\small
  \bibliographystyle{siam}
  \bibliography{references}

\def\ocirc#1{\ifmmode\setbox0=\hbox{$#1$}\dimen0=\ht0 \advance\dimen0
  by1pt\rlap{\hbox to\wd0{\hss\raise\dimen0
  \hbox{\hskip.2em$\scriptscriptstyle\circ$}\hss}}#1\else {\accent"17 #1}\fi}
  \def\cprime{$'$} \def\cprime{$'$}
\begin{thebibliography}{10}

\bibitem{HAbels_HGarcke_GGruen_2010a}
{\sc H.~Abels, H.~Garcke, and G.~Gr{\"u}n}, {\em Thermodynamically consistent
  diffuse interface models for incompressible two-phase flows with different
  densities}, arXiv:1011.0528,  (2010).

\bibitem{HAntil_RHNochetto_PSodre_2014a}
{\sc H.~Antil, R.~H. Nochetto, and P.~Sodr{\'e}}, {\em Optimal {C}ontrol of a
  {F}ree {B}oundary {P}roblem: {A}nalysis with {S}econd-{O}rder {S}ufficient
  {C}onditions}, SIAM J. Control Optim., 52 (2014), pp.~2771--2799.

\bibitem{NArada_ECasas_FTroeltzsch_2002a}
{\sc N.~Arada, E.~Casas, and F.~Tr{\"o}ltzsch}, {\em Error estimates for the
  numerical approximation of a semilinear elliptic control problem}, Comput.
  Optim. Appl., 23 (2002), pp.~201--229.

\bibitem{HBae_2011a}
{\sc H.~Bae}, {\em Solvability of the free boundary value problem of the
  {N}avier-{S}tokes equations}, Discrete Contin. Dyn. Syst., 29 (2011),
  pp.~769--801.

\bibitem{WBangerth_RHartmann_GKanschat_2007a}
{\sc W.~Bangerth, R.~Hartmann, and G.~Kanschat}, {\em {deal.II} -- a general
  purpose object oriented finite element library}, ACM Trans. Math. Softw., 33
  (2007), pp.~24/1--24/27.

\bibitem{BBerge_2000a}
{\sc B.~Berge and J.~Peseux}, {\em Variable focal lens controlled by an
  external voltage: An application of electrowetting}, The European Physical
  Journal E, 3 (2000), pp.~159--163.

\bibitem{SCBrenner_RLScott_2008a}
{\sc S.~C. Brenner and L.~R. Scott}, {\em The Mathematical Theory of Finite
  Element Methods}, vol.~15 of Texts in Applied Mathematics, Springer, New
  York, third~ed., 2008.

\bibitem{ECasas_MMateos_FTroeltzsch_2005a}
{\sc E.~Casas, M.~Mateos, and F.~Tr{\"o}ltzsch}, {\em Error estimates for the
  numerical approximation of boundary semilinear elliptic control problems},
  Comput. Optim. Appl., 31 (2005), pp.~193--219.

\bibitem{ECasas_FTroeltzsch_2002a}
{\sc E.~Casas and F.~Tr{\"o}ltzsch}, {\em Error estimates for the
  finite-element approximation of a semilinear elliptic control problem},
  Control Cybernet., 31 (2002), pp.~695--712.
\newblock Well-posedness in optimization and related topics (Warsaw, 2001).

\bibitem{PDeuflhard2004a}
{\sc P.~Deuflhard}, {\em Newton Methods for Nonlinear Problems - Affine
  Invariance and Adaptive Algorithms}, vol.~35 of Springer Series in Comp.
  Math., Springer-Verlag, Berlin, 2004.

\bibitem{RBFair_2007a}
{\sc R.~B. Fair}, {\em Digital microfluidics: is a true lab-on-a-chip
  possible?}, Microfluidics and Nanofluidics, 3 (2007), pp.~245--281.

\bibitem{DGilbarg_NTrudinger_2001a}
{\sc D.~Gilbarg and N.~S. Trudinger}, {\em Elliptic Partial Differential
  Equations of Second Order}, Classics in Mathematics, Springer-Verlag, Berlin,
  2001.
\newblock Reprint of the 1998 edition.

\bibitem{MDGunzburger_2003a}
{\sc M.D. Gunzburger}, {\em Perspectives in flow control and optimization},
  vol.~5 of Advances in Design and Control, SIAM, Philadelphia, PA, 2003.

\bibitem{HHartshorne_JBackhouse_WLLee_2004a}
{\sc H.~Hartshorne, C.~J. Backhouse, and W.~E. Lee}, {\em Ferrofluid-based
  microchip pump and valve}, Sensors and Actuators B: Chemical, 99 (2004),
  pp.~592 -- 600.

\bibitem{JHeikenfeld_2011a}
{\sc J.~Heikenfeld, P.~Drzaic, J.-S. Yeo, and T.~Koch}, {\em Review paper: A
  critical review of the present and future prospects for electronic paper},
  Journal of the Society for Information Display, 19 (2011), pp.~129--156.

\bibitem{MHinze_2005a}
{\sc M.~Hinze}, {\em A variational discretization concept in control
  constrained optimization: the linear-quadratic case}, Comput. Optim. Appl.,
  30 (2005), pp.~45--61.

\bibitem{MHinze_RPinnau_MUlbrich_SUlbrich_2009a}
{\sc M.~Hinze, R.~Pinnau, M.~Ulbrich, and S.~Ulbrich}, {\em Optimization with
  {PDE} constraints}, vol.~23 of Mathematical Modelling: Theory and
  Applications, Springer, New York, 2009.

\bibitem{BJJin_2005a}
{\sc B.~J.~Jin}, {\em Free boundary problem of steady incompressible flow with
  contact angle {$\frac\pi2$}}, J. Differential Equations, 217 (2005),
  pp.~1--25.

\bibitem{BJJin_MPadula_2004a}
{\sc B.~J. Jin and M.~Padula}, {\em Steady flows of compressible fluids in a
  rigid container with upper free boundary}, Math. Ann., 329 (2004),
  pp.~723--770.

\bibitem{OLavrova_LTobiska_2006a}
{\sc O~Lavrova, G~Matthies, T~Mitkova, V~Polevikov, and L~Tobiska}, {\em
  Numerical treatment of free surface problems in ferrohydrodynamics}, Journal
  of Physics: Condensed Matter, 18 (2006), p.~S2657.

\bibitem{GLippmann_1875a}
{\sc G.~Lippmann}, {\em Relations entre les ph{\'e}nomenes {\'e}lectriques et
  capillaires}, Ann. Chim. Phys., 5 (1875), pp.~494--549.

\bibitem{HLu_KGlasner_ABertozzi_CKim_2007a}
{\sc H.W. Lu, K.~Glasner, AL~Bertozzi, and C.J. Kim}, {\em A diffuse-interface
  model for electrowetting drops in a hele-shaw cell}, Journal of Fluid
  Mechanics, 590 (2007), pp.~411--435.

\bibitem{MMiwa_HHarita_TNishigami_RKaneko_2003a}
{\sc M.~Miwa, H.~Harita, T.~Nishigami, R.~Kaneko, and H.~Unozawa}, {\em
  Frequency characteristics of stiffness and damping effect of a ferrofluid
  bearing}, Tribology Letters, 15 (2003), pp.~97--105.

\bibitem{FMugele_JCBaret_2005a}
{\sc F.~Mugele and J.C. Baret}, {\em Electrowetting: from basics to
  applications}, Journal of Physics: Condensed Matter, 17 (2005), p.~R705.

\bibitem{FMugele_MDuits_DVandenEnde_2010a}
{\sc F.~Mugele, M.~Duits, and D.~Van~den Ende}, {\em Electrowetting: A
  versatile tool for drop manipulation, generation, and characterization},
  Advances in colloid and interface science, 161 (2010), pp.~115--123.

\bibitem{RHNochetto_AJSalgado_SWWalker_2011a}
{\sc R.~H. Nochetto, A.~J. Salgado, and S.~W. Walker}, {\em A diffuse interface
  model for electrowetting with moving contact lines}, Math. Models Methods
  Appl. Sci., 24 (2014), pp.~67--111.

\bibitem{MPadula_VASolonnikov_2010}
{\sc M.~Padula and V.~A. Solonnikov}, {\em On the local solvability of free
  boundary problem for the {N}avier-{S}tokes equations}, J. Math. Sci. (N. Y.),
  170 (2010), pp.~522--553.
\newblock Problems in mathematical analysis. No. 50.

\bibitem{TQian_XPWang_PSheng_2003a}
{\sc T.~Qian, X.-P. Wang, and P.~Sheng}, {\em Generalized {N}avier boundary
  condition for the moving contact line}, Commun. Math. Sci., 1 (2003),
  pp.~333--341.

\bibitem{TQian_XPWang_PSheng_2006a}
\leavevmode\vrule height 2pt depth -1.6pt width 23pt, {\em A variational
  approach to moving contact line hydrodynamics}, J. Fluid Mech., 564 (2006),
  pp.~333--360.

\bibitem{ARoesch_2006a}
{\sc A.~R{\"o}sch}, {\em Error estimates for linear-quadratic control problems
  with control constraints}, Optim. Methods Softw., 21 (2006), pp.~121--134.

\bibitem{PSaavedra_RScott_1991}
{\sc P.~Saavedra and L.~R. Scott}, {\em Variational formulation of a model
  free-boundary problem}, Math. Comp., 57 (1991), pp.~451--475.

\bibitem{VASolonnikov_1982a}
{\sc V.~A. Solonnikov}, {\em On the stokes equations in domains with non-smooth
  boundaries and on viscous incompressible flow with a free surface}, in
  Nonlinear partial differential equations and their applications: Coll{\`e}ge
  de France seminar, vol.~3, Pitman Publishing (UK), 1982, pp.~340--423.

\bibitem{FTroltzsch_2010a}
{\sc F.~Tr{\"o}ltzsch}, {\em Optimal control of partial differential
  equations}, vol.~112 of Graduate Studies in Mathematics, American
  Mathematical Society, Providence, RI, 2010.
\newblock Theory, methods and applications, Translated from the 2005 German
  original by J{\"u}rgen Sprekels.

\bibitem{KJVinoy_1996a}
{\sc K.~J. Vinoy and R.~M. Jha}, {\em Radar absorbing materials: From theory to
  design and characterization}, Kluwer academic publishers Boston, 1996.

\bibitem{SWalker_ABonito_RNochetto_2010a}
{\sc S.~W. Walker, A.~Bonito, and R.~H. Nochetto}, {\em Mixed finite element
  method for electrowetting on dielectric with contact line pinning},
  Interfaces Free Bound., 12 (2010), pp.~85--119.

\bibitem{MZahn_2001a}
{\sc M.~Zahn}, {\em Magnetic fluid and nanoparticle applications to
  nanotechnology}, Journal of Nanoparticle Research,  (2001), pp.~73 -- 78.

\end{thebibliography}
}

\end{document}